\documentclass[11pt]{amsart}

\usepackage{mathtools}
\usepackage{amsmath}
\usepackage{amssymb, stmaryrd, bm}
\usepackage{amsthm}
\usepackage{graphicx}
\usepackage{tikz}
\usepackage{xcolor}
\usepackage{enumitem}
\usepackage{algorithm}

\usetikzlibrary{calc, backgrounds}

\usepackage[foot]{amsaddr}
\usepackage[pdftex,colorlinks,backref=page,citecolor=blue]{hyperref}
\mathtoolsset{showonlyrefs}

\allowdisplaybreaks

\usepackage[margin=1.2in]{geometry}
\usepackage{bookmark}

\def\E{\mathbb{E}}
\def\var{\mathbb{Var}}
\def\cov{\textrm{Cov}}

\def\N{\mathbb{N}}

\def\E{\mathbb{E}}
\def\R{\mathbb{R}}
\def\P{\mathbb{P}}

\def\eps{\varepsilon}
\def\del{\delta}

\def\cB{\mathcal {B}}

\def\cE{\mathcal {E}}

\def\cI{\mathcal {I}}

\def\1{\mathbf{1}}

\def\lam {\lambda}

\def\tce{t_c + \eps}
\def\tce2{t_c + \frac{\eps}{2}}

\def\var{\textup{var}}
\def\cov{\textup{cov}}

\DeclareMathOperator{\aut}{aut}

\def\bsf#1{\mathbf{\mathsf{#1}}}

\newtheorem*{theorem*}{Theorem}
\newtheorem{theorem}{Theorem}
\numberwithin{theorem}{section}
\newtheorem{lemma}[theorem]{Lemma}
\newtheorem{cor}[theorem]{Corollary}
\newtheorem{defn}[theorem]{Definition}
\newtheorem*{defn*}{Definition}

\newtheorem*{prop*}{Proposition}
\newtheorem{conj}{Conjecture}
\newtheorem*{conj*}{Conjecture}
\newtheorem{claim}[theorem]{Claim}

\newtheorem*{fact*}{Fact}

\newtheorem*{obs*}{Observation}
\newtheorem{obs}[theorem]{Observation}
\newtheorem{condition}{Condition}
\newtheorem{definition}{Definition}

\newtheorem{assumption}{Assumption}

\numberwithin{equation}{section}


\begin{document}
\title[Non-existence probabilities and lower tails in the critical regime]{Non-existence probabilities and lower tails in the critical regime via Belief Propagation }

\author{Matthew Jenssen}
\author{Will Perkins}
\author{Aditya Potukuchi}
\author{Michael Simkin}

\address{King's College London, Department of Mathematics}
\email{matthew.jenssen@kcl.ac.uk}

\address{Georgia Institute of Technology, School of Computer Science}
\email{math@willperkins.org}

\address{Max Planck Institute for Software Systems}
\email{apotu@mpi-sws.org}

\address{Massachusetts Institute of Technology, Department of Mathematics}
\email{msimkin@mit.edu}

\date{\today}

\maketitle

\begin{abstract}
    We compute the logarithmic asymptotics of the non-existence probability (and more generally the lower-tail probability)  for a wide variety of combinatorial problems for a range of parameters  in the `critical regime' between the regime amenable to hypergraph container methods and that amenable to Janson's inequality.  Examples include lower tails and non-existence probabilities for  subgraphs of random graphs and  for $k$-term arithmetic progressions in random sets of  integers. 

    Our methods apply in the general framework of estimating the probability that a $p$-random subset of vertices in a $k$-uniform hypergraph induces significantly fewer hyperedges than expected.  We show that under some simple structural conditions on the hypergraph and an upper bound on $p$ determined by a phase transition in the hard-core model on the infinite $k$-uniform, $\Delta$-regular, linear hypertree, this probability can be accurately approximated  by the Bethe free energy evaluated at the unique fixed point of a Belief Propagation operator on the hypergraph. 
\end{abstract}

\section{Introduction}
\label{secIntro}

Many problems in probabilistic and enumerative combinatorics and non-linear large deviations in probability theory can be formulated in terms of the probability that a $p$-random subset of vertices of a hypergraph induces no (or few) edges.  Techniques to estimate this probability fall into two broad categories, applicable for different ranges of $p$.  When $p$ is small (as a function of parameters of the hypergraph such as its average degree and size of its hyperedges), we are in the `Poisson paradigm' and tools like Janson's Inequality give accurate bounds.  When $p$ is large, on the other hand, the global structure of the hypergraph becomes important, and techniques like the regularity lemma and the method of hypergraph containers are effective.
What happens in between these two regimes?  Very few results are known here, even for special cases.  

We introduce a new method, based on message-passing algorithms for statistical physics models on hypergraphs, that gives logarithmic asymptotics of the non-existence and lower-tail probabilities in a portion of the `critical regime' under mild local sparsity conditions on the hypergraph.  The formula is obtained from the Bethe free energy evaluated at the unique fixed point of a Belief Propagation operator associated to the hypergraph.  Such formulas have been shown to be asymptotically correct for some spin models on sparse random (and locally treelike) graphs; here we establish their validity in a deterministic, combinatorial setting.  Our proofs utilize a technique from approximate counting and sampling: contraction on a computational hypertree. One novelty of our analysis is that we show  contraction assuming only   the local sparsity condition along with weak spatial mixing (Gibbs uniqueness) on the infinite hypertree. 

We will state our main results in generality below in Section~\ref{subsecHypergraphIntro}, but first we give some  applications of the results to widely studied problems.

\subsection{$H$-free graphs}
\label{subsecHfreeIntro}

For a given graph $H$, two central questions in combinatorics are `how many $H$-free graphs are there on $n$ vertices and $m$ edges?' and the closely related question `what is the probability that the random graph $G(n,p)$ contains no copy of $H$?'  The answers to these questions involve some of the main results and techniques in probabilistic and extremal combinatorics, including Turán-type theorems, regularity methods, hypergraph containers, and Janson's inequality.  

Taking the case of triangles, $H=K_3$, as an example, we briefly review what is known.  In perhaps the first extremal combinatorics theorem, Mantel~\cite{mantel1907problem} proved that a triangle-free graph on $n$ vertices can have at most $\lfloor n^2/4 \rfloor$ edges and that this is achieved by a balanced, complete bipartite graph.  This extremal structure is reflected in the counting problem: Erd\H{o}s, Kleitman, and Rothschild~\cite{erdosasymptotic} showed that almost every triangle-free graph is bipartite, and hence the asymptotic enumeration problem for triangle-free graphs (equivalently, the probability that $G(n,1/2)$ is triangle-free) reduces to the elementary task of asymptotically enumerating bipartite graphs. This behavior persists for far smaller values of $p$.  

One can ask for different degrees of accuracy in approximating $\P (\mathbf X_H=0)$ (where $\mathbf X_H$ is the number of copies of a subgraph $H$ and the probability is with respect to either $G(n,p)$ or $G(n,m)$).  A first step would be to approximate the order of the exponent, that is, determine $\log \P (\mathbf X_H =0)$ up to constant factors.  This was completely resolved in~\cite{janson1987uczak} (and in the greater generality of the lower-tail problem described below in Section~\ref{subsecLowerTailsIntro}).  A next step would be to find first-order asymptotics of $\log \P(\mathbf X_H=0)$, a task known in large deviation theory as computing the \textit{rate function}.  This will be our goal in this paper; as we will see shortly, this has been resolved except in what we will call the `critical regime'.  Finally, even more ambitiously one can ask for first-order asymptotics of $\P(\mathbf X_H=0)$; this is known for cliques for regimes of $p,m$ in which almost all $K_{r+1}$-free graphs are $r$-partite~\cite{osthus2003densities,balogh2016typical} (and somewhat beyond this in the case of triangles~\cite{jenssen2023evolution}). This is also known when $p$ is a small polynomial factor below the critical density~\cite{wormald1996perturbation,stark2018probability,mousset2020probability}. 

We now state some known results on logarithmic asymptotics precisely. 
Define the $2$-density of a graph $H$ to be 
\begin{equation}
    m_2(H)= \max_{F \subseteq H, |V(F)| \ge 3} \frac{|E(F)|-1}{|V(F)|-2} \,.
\end{equation}
A graph $H$ with $|E(H)| \geq 3$ is \textit{strictly $2$-balanced} if  this maximum is uniquely achieved at $F= H$.  For instance, a clique $K_r$ (with $r \ge 3$) is strictly $2$-balanced while a clique with a pendant edge is not.

Here and in what follows, along with the usual $o, O, \omega, \Omega, \Theta$ asymptotic notation, we write  $f(n) \sim g(n)$ if $\frac{f(n)}{g(n)} \to 1$ as $n \to \infty$.  We use $\P_p, \E_p$ to denote probabilities and expectations with respect to $G(n,p)$ and we use $\P_m, \E_m$ to denote probabilities and expectations with respect to $G(n,m)$.
\begin{theorem}[\cite{janson1987uczak,luczak2000triangle,balogh2015independent,saxton2015hypergraph}]
    \label{thmHfreePrevios}
    Let $H$ be a strictly $2$-balanced graph with chromatic number  $\chi(H) \ge 3$.  Then
    \begin{itemize}
        \item If $p = o \left(n^{-1/m_2(H)} \right)$, then $\log \P_p(\mathbf X_H=0) \sim - \E_p \mathbf X_H $.
         \item If $ p =\omega \left ( n^{-1/m_2(H)} \right )$, then $\log \P_p(\mathbf X_H=0) \sim   \frac{n^2}{2 (\chi(H)-1)}  \log (1-p)$.
    \end{itemize}
   Similarly,
    \begin{itemize}
    \item If $m =o \left ( n^{2-1/m_2(H)} \right)$, then $\log \P_m(\mathbf X_H=0) \sim  - \E_m \mathbf X_H$. 
    \item If $m= o(n^2)$ and  $ m = \omega \left ( n^{2-1/m_2(H)} \right)$, then $\log \P_m(\mathbf X_H=0) \sim m \log \left(1 - \frac{1}{\chi(H)-1}   \right) $.
    \end{itemize}
\end{theorem}
(If $m = \Theta(n^2)$ the logarithm of the probability is also asymptotic to logarithm of the probability of being $(\chi(H)-1)$-partite with a slightly more complicated formula).
Thus, on the level of the rate function, the probability of $H$-freeness has been completely resolved apart from the regime $p = \Theta \left( n^{-1/m_2(H)}\right) $ and $m = \Theta \left( n^{2-1/m_2(H)}\right)$.  As the nature of the rate function (and the typical structure of the conditional distribution) changes here, we call this the \textit{critical regime}, and below in Corollary~\ref{corPhaseTransition} we prove a result justifying this terminology. 

In~\cite{jenssen2024lower}, the current  authors proved an asymptotic formula for the logarithm of the probability of triangle-freeness (and the lower-tail problem discussed below in Section~\ref{subsecLowerTailsIntro}) in $G(n,p)$ with $p \sim c/\sqrt{n}$ for $c < e^{-1/2}$ and for $G(n,m)$ with  $m \sim b n^{3/2}$ with $b <\frac{1}{2}\sqrt{W(2/e)}  \approx 0.3402$.   The method of~\cite{jenssen2024lower}, however,  is very specific to triangles in random graphs, and does not apply to other subgraphs $H$ or to the other non-existence and lower-tail problems described below.

Here we prove an asymptotic formula for the logarithm of the  probability of $H$-freeness for all strictly $2$-balanced $H$ inside their respective critical regimes.  Let $\aut(H)$ be the set of automorphisms of the graph $H$, and let $h = |V(H)|$, $k= |E(H)|$.   Let
\begin{equation}
    \Delta_H = \frac{2k (n-2)_{h-2}}{|\aut(H)|} \sim \frac{2k}{| \aut(H) |} n^{h-2} \,,
\end{equation}
where  $(n)_j = n (n-1) \cdots (n-j+1)$.
The parameter $\Delta_H$ is the number of copies of $H$ in the complete graph $K_n$ containing a specified edge (not to be confused with the maximum degree of $H$).  We use $\Delta_H$ to parameterize the problem to match the more general main result of Section~\ref{subsecHypergraphIntro} below. Finally, define
\begin{equation}
\label{eqXast0}
    x_k^\ast(c) =  \left( \frac{ W((k-1)c^{k-1})}{k-1}   \right )^{\frac{1}{k-1}}  \,,
\end{equation}
where $W(\cdot)$ is the Lambert-W function, the inverse of $ye^y$.

\begin{theorem}\label{thmHFree}
   Let $H$ be a strictly $2$-balanced graph with $k$ edges and fix $c < \left(\frac{e} {k-1} \right)^{ \frac{1}{k-1}}$.  If 
   \begin{equation}
       p \sim c \cdot \Delta_H^{-\frac{1}{k-1}} \sim c \cdot \left(  \frac{|\aut(H)|}{2k}\right)^{\frac{1}{k-1}} n^{- \frac{1}{m_2(H)}} \,,
   \end{equation} 
   then
\begin{equation}
   \lim_{n \to \infty} \Delta_H^{ \frac{1}{k-1}} \binom{n}{2}^{-1}  \log \P_p (\mathbf X_H=0)  =  x^\ast +\left(1-\frac{1}{k} \right)(x^\ast)^k - c  \,,
 \end{equation}
 where $x^\ast = x_k^\ast(c)$ as defined in~\eqref{eqXast0}.
    Moreover, with $b < (k-1)^{- \frac{1}{k-1}}$ fixed, and 
    \begin{equation}
        m \sim b \cdot \Delta_H^{-\frac{1}{k-1}} \binom{n}{2} \sim b \cdot  \left(  \frac{|\aut(H)|}{2k}\right)^{\frac{1}{k-1}}   n^{2- \frac{1}{m_2(H)}} \,,
    \end{equation}
    there holds
    \begin{equation}
        \lim_{n \to \infty} \Delta_H^{ \frac{1}{k-1}} \binom{n}{2}^{-1} \log \P_m (\mathbf X_H=0)  = -  \frac{b^k}{k} \,.
\end{equation}
In other words,
\begin{equation}
    \log \P_m (\mathbf X_H=0) \sim - \E_m \mathbf X_H \,.
\end{equation}
\end{theorem}

 Theorem~\ref{thmHFree} implies that the $H$-freeness problem (for non-bipartite, strictly $2$-balanced graphs $H$) undergoes a phase transition in the critical regime, in the sense of a non-analyticity of the rate function.

Define the $H$-free rate functions for $G(n,p)$ and $G(n,m)$ respectively as
\begin{align}
    \varphi_{H}(c) &= \liminf_{n \to \infty} n^{\frac{1}{m_2(H)} -2} \log \P_p (\mathbf X_H=0) \,, \quad \text{with } p = c n^{- \frac{1}{m_2(H)}} \,, \intertext{and}
     \widehat \varphi_{H}(b) &= \liminf_{n \to \infty} n^{\frac{1}{m_2(H)} -2} \log \P_m (\mathbf X_H=0) \,, \quad \text{with } m = b n^{2- \frac{1}{m_2(H)}} \,.
\end{align}

\begin{cor}
    \label{corPhaseTransition}
    For each strictly $2$-balanced $H$ with $\chi(H)\geq 3$, there exist $c_H^\ast \in (0, \infty)$ and $b_H^\ast \in (0,\infty)$ so that $\varphi_{H}$ and $\widehat \varphi_{H}$ are non-analytic at $c_H^\ast$ and $b_H^\ast$ respectively.
\end{cor}
The case of $H= K_3$ was proved in~\cite{jenssen2024lower}.
Here we require $\chi(H) \ge 3$ to ensure that for $c$ or $b$ sufficiently large, typical $H$-free graphs are approximately $(\chi(H)-1)$-partite (as shown via the method of hypergraph containers~\cite{balogh2015independent,saxton2015hypergraph}; see also~\cite{luczak2000triangle,kohayakawa1997k}).  The proof of Corollary~\ref{corPhaseTransition} is then quite simple: the formula for the rate function for small $c$ given by Theorem~\ref{thmHFree} is an analytic function of $c$; yet the lower bound on the rate function given by considering $(\chi(H)-1)$-partite graphs shows that this formula cannot hold for large $c$; by uniqueness of analytic continuation the $H$-free rate function must have a non-analytic point. 

We leave it as an open problem to show that  $\varphi_{H}$ and $\widehat \varphi_{H}$ have unique points of non-analyticity and to determine their locations.

\subsection{Lower tails for subgraphs}
\label{subsecLowerTailsIntro}

A more general question is to estimate the probability that $G(n,p)$ or $G(n,m)$ has significantly fewer copies of a subgraph $H$ than expected.  This is the \textit{lower-tail} problem.  In particular, we will be interested in the logarithmic asymptotics of 
\begin{equation}
    \P ( \mathbf X_H \le \eta \E \mathbf X_H)
\end{equation}
where $\eta \in [0,1)$ is fixed and the probability and expectation is with respect to either $G(n,p)$ or $G(n,m)$.

Here again much is known outside of the critical regime; for $p = o \left ( n^{-1/m_2(H)} \right)$ the lower tail exhibits Poisson-like behavior~\cite{janson1987uczak,janson1990poisson,janson2016lower}; while for $p =\omega \left ( n^{-1/m_2(H)} \right)$ the rate function is given by the solution of an entropy maximization problem over graphons~\cite{chatterjee2011large,kozma2023lower}.  This maximization problem is only partially solved however~\cite{zhao2017lower}.

We now generalize Theorem~\ref{thmHFree} to the lower-tail problem. First we introduce some notation.  Generalizing~\eqref{eqXast0} (the $\zeta=1$ case), define 
\begin{equation}
\label{eqxastdef}
      x_k^\ast(c,\zeta) =  \left( \frac{W((k-1)c^{k-1}\zeta)}{(k-1)\zeta } \right)^{\frac{1}{k-1}} \,.
\end{equation}
Let 
\begin{equation}
\label{eqEtaStarDef}
\eta_k^\ast = e^{-k/(k-1)} \,,
\end{equation}
and define $\overline c_k:[0,1]\to \R\cup\{\infty\}$ as follows:
\begin{equation}
\label{eqOverlineCeta}
\overline c_k(\eta) =
\begin{cases} 
\left( \frac{e}{(k-1)(1-\eta/\eta^\ast_k)} \right)^{\frac{1}{k-1}}   & \text{if } \eta < \eta_k^\ast\\
+ \infty & \text{if } \eta \geq \eta_k^\ast\, . 
\end{cases}
\end{equation}

\begin{theorem}\label{thmHLowerTail}
   Let $H$ be a strictly $2$-balanced graph with $k$ edges. 
 Fix $c>0$ and $\eta\in [0,1)$ satisfying $c<\overline c_k(\eta)$.
If
 \begin{equation}
       p \sim c \cdot \Delta_H^{-\frac{1}{k-1}}  \,,
   \end{equation} 
   then 
\begin{align}
\label{eqMainThmAsymptotics}
  \lim_{n\to\infty}  \Delta_H^{ \frac{1}{k-1}} \binom{n}{2}^{-1} \log \P_p(\mathbf X_H\leq \eta \E \mathbf X_H) &=  x^\ast + (x^\ast)^k  \left ( 1- \frac{1}{k} \right) \zeta   - \log (1-\zeta) \frac{\eta c^k}{k} - c \, , 
\end{align}
where $\zeta$ is the unique solution in $[0,1]$ to the equation
 \begin{equation}
 \label{eqZetaDef}
    (1-\zeta) (x_k^\ast(c,\zeta))^k= \eta c^k\, ,
 \end{equation}
 and $x^\ast = x_k^\ast(c,\zeta)$ is given by~\eqref{eqxastdef}.
    
    Similarly, fix $\eta \in [0,1)$ and  suppose $b$ satisfies 
     \begin{equation}
 (k-1)  b^{k-1} (1-\eta)  < 1 \,.
    \end{equation}
    If $m \sim b \cdot \Delta_H^{-\frac{1}{k-1}} \binom{n}{2}$, then
    \begin{equation}
   \lim_{n \to \infty}\Delta_H^{ \frac{1}{k-1}} \binom{n}{2}^{-1}\log \P_m (\mathbf X_H \le \eta \E \mathbf X_H )  =  -  \frac{b^k (1- \eta +\eta \log \eta)}{k} \,.
\end{equation}
\end{theorem}
Just as the probability of $H$-freeness in $G(n,m)$ exhibited   Poisson behavior  in Theorem~\ref{thmHFree}, the lower-tail probability in $G(n,m)$ matches the Poisson lower tail.

\subsection{Avoiding $k$-APs}

A $k$-term arithmetic progression ($k$-AP) of integers is a set $\{a, a+b, \dots a+ (k-1)b \}$ where $a, b \in \mathbb N$.  A central topic in arithmetic combinatorics is to understand the maximum density of a subset of the integers $\{1, \dots, n \}$ with no $k$-AP; Szemer{\'e}di's Theorem~\cite{szemeredi1975sets} states that this maximum density is vanishing as $n \to \infty$; proving better bounds on this density remains a very active area of research~\cite{kelley2023strong,leng2024improved}.  Here we will consider the probability that a random subset of the integers $\{1, \dots, n\}$ of a given density has no $k$-AP.

Let $[n] = \{ 1, \dots, n \}$ and let $[n]_p$ be a random subset of $[n]$ formed by including each element independently with probability $p$.  Here we use $\P_p, \E_p$ to denote probabilities and expectations with respect to $[n]_p$.  For $k \ge 3$, let $\mathbf X_k$  denote the number of $k$-APs in $[n]_p$. Again, Janson's Inequality~\cite{janson1987uczak} and the method of hypergraph containers~\cite{balogh2015independent,saxton2015hypergraph} determine the asymptotics of the logarithm of the non-existence probability when $p$ is sufficiently small or large respectively:
\begin{align}
\log \P_p ( \mathbf X_k=0) &\sim - \E_p \mathbf X_k   \quad \text{when} \quad p \ll n^{-\frac{1}{k-1}} \\
    \log \P_p ( \mathbf X_k=0) &\sim n \log (1-p)  \quad \text{when} \quad p \gg n^{-\frac{1}{k-1}} \,.
\end{align}

Here we give the asymptotics of the logarithm of the non-existence probability when $p \sim c n^{-\frac{1}{k-1}}$ and $c$ is a sufficiently small constant.  Unlike the case of subgraphs in random graphs which is symmetric with respect to edges of the complete graph, this problem is not symmetric with respect to elements of $[n]$, and so the asymptotic formula for the log probability is  more complicated. We start with some definitions.

Fix $k\ge 3$, fix $c>0$, and define the operator $\Phi^{(k)}_c$ on the set of measurable functions $x: [0,1] \to [0,c]$ as follows:
\begin{equation}
\label{eqPhikc}
    \Phi^{(k)}_c x(t) = c \exp \left [ - \sum_{j=1}^k  \int_{0}^{ \min \{\frac{t}{j-1} , \frac{1-t}{k-j} \}}   \,\prod_{i=1}^{j-1}  x(t -(j-i) s)  \cdot \prod_{i=j+1}^k x(t + (i-j)s)  \, ds \right]  \,.
\end{equation}

We first show $\Phi^{(k)}_c$ has a unique functional fixed point.
\begin{lemma}
\label{lemkAPfixedpoint}
   Fix $k \ge 3$ and let $\alpha_k:= \frac{1}{2}\sum_{i=1}^{k}
\min\left(
 \frac{1}{i-1},
 \frac{1}{k-i}\right)$.  For $c <  \left(\frac{e} {(k-1)\alpha_k} \right)^{ \frac{1}{k-1}}$, there is a unique  function $x_{k,c}^* : [0,1] \to [0,c]$ so that 
   \begin{equation}
        \Phi^{(k)}_c x_{k,c}^* = x_{k,c}^* \,.
   \end{equation}
   Moreover the function $ x_{k,c}^*$ is a continuous function on $[0,1]$ and for each $t \in[0,1]$, $x_{k,c}^*(t)$ is an analytic function of $c$ on the interval $\left(0, \left(\frac{e} {(k-1)\alpha_k} \right)^{ \frac{1}{k-1}}  \right) $.
\end{lemma}
For $k=3$, $c =1$, we have plotted  $x_{k,c}^*$ in Figure~\ref{fig:3APdensity}.

\begin{figure}
    \centering
    \includegraphics[width=0.7\linewidth]{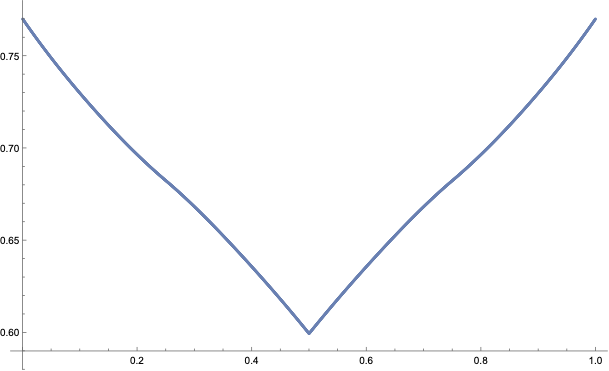}
    \caption{The function $x_{3,1}^*$ from Lemma~\ref{lemkAPfixedpoint} with $k=3$, $c=1$, representing the density of elements in a random $3$-AP-free subset of $[n]$ chosen with prior density $n^{-1/2}$.}
    \label{fig:3APdensity}
\end{figure}

Our main result on the probability of the non-existence of arithmetic progressions is the following.  
\begin{theorem}
    \label{thmApavoidInterval}
      Fix $k \ge 3$ and  $c <  \left(\frac{e} {(k-1)\alpha_k} \right)^{ \frac{1}{k-1}}$. Then with $p \sim c n^{-\frac{1}{k-1}}$,
    \begin{equation}
       \lim_{n\to\infty} n^{- \frac{k-2}{k-1} } \cdot \log \P_p ( \mathbf X_k=0) = \int_0^1 \int_0^c  \frac{x^*_{k,t}(s)}{t}\, dt  \, ds  - c  \,.
    \end{equation}
    Moreover, for each $j \in [n]$,
    \begin{equation}
    n^{\frac{1}{k-1}} \cdot \P_p \left ( j \in [n]_p \big | \mathbf X_k=0 \right)   =  x_{k,c}^*(j/n) +o(1)  \,.
    \end{equation}
\end{theorem}
The second statement of the theorem is about marginal probabilities in the conditional measure: how likely a specified integer is to be in a random subset given that the subset is $k$-AP free.  We note that due to the lack of symmetry in the problem, the marginals in Theorem~\ref{thmApavoidInterval} are not uniform; see e.g.\ Figure~\ref{fig:3APdensity}.

The formula for the rate function given in Theorem~\ref{thmApavoidInterval}  involves solving a functional fixed point equation for each $t \in [0,c]$, or alternatively, a two-dimensional functional fixed-point equation. In Section~\ref{secAPs} we give an alternative formula written  in terms of  a single one-dimensional functional fixed point. 

Unlike in the case of $H$-freeness (with $\chi(H)\ge 3$) we do not prove the existence of a phase transition for $k$-AP-freeness.  In fact, we conjecture there is no phase transition for this problem or for $C_{2k}$-freeness in $G(n,p)$ for  $k\ge 2$.  
\begin{conj}
    \label{conjNoPhase}
    The following limits exist and define analytic functions of $c$ on $(0,\infty)$:
    \begin{align}
        \phi_{k-\mathrm{AP}}(c) &= \lim_{n \to \infty} n^{-\frac{k-2}{k-1}} \log \P_p [ \mathbf X_k=0], \, \, \text{ for } k\ge 3, p= c n^{-\frac{1}{k-1}} \\
        \phi_{C_{2k}}(c) &= \lim_{n \to \infty} n^{- \frac{2k}{2k-1}} \log \P_p [ \mathbf X_{C_{2k}}=0], \, \, \text{ for } k\ge 2,  p= c n^{- \frac{2k-2}{2k-1}} \,.
    \end{align}
\end{conj}
What the two problems have in common, and the reason for our conjecture, is that in both cases the extremal construction is sparse; i.e. any $C_{2k}$-free graph on $n$ vertices has $o(n^2)$ edges and any $k$-AP-free subset of $[n]$ has $o(n)$ elements. Hypergraph container methods have been used to show that the non-existence probabilities in the supercritical regimes match, on the level of logarithms, that of having no edges or no elements. We conjecture that there is an analytic interpolation from the formulas we have proved to these supercritical formulas. We note that the $k$-AP and $C_{2k}$ problems are chosen as illustrative examples, and we expect Conjecture~\ref{conjNoPhase} to hold for a wide class of arithmetic patterns and graphs. For example, we expect an analogue of the conjecture to hold with $C_{2k}$ replaced by any strictly $2$-balanced bipartite graph.

\subsection{Independent and sparse sets in hypergraphs}
\label{subsecHypergraphIntro}

A hypergraph $G=(V,E)$ consists of a set of vertices $V$ and a set $E$ of non-empty subsets of vertices, called edges (or hyperedges). In a $k$-uniform hypergraph, each edge has cardinality $k$.  A $2$-uniform hypergraph is a usual graph.  An \textit{independent set} in a hypergraph is a set of vertices  $I \subseteq V$ so that for all $e \in E$, $e \nsubseteq I$; that is, $I$ induces no edges.  

The above problems about $H$-free graphs and $k$-AP-free sets, along with many other combinatorial problems,  can be phrased in terms of independent sets  in hypergraphs.    Consequently, methods for understanding independent sets in hypergraphs can be very general and powerful tools in combinatorics, see e.g.~\cite{kleitman1982number,janson1987uczak,janson1990poisson,balogh2018method,saxton2015hypergraph,mousset2020probability}. The corresponding lower-tail problems can be phrased in terms of vertex subsets of hypergraphs inducing few edges. In this section we will state our main results in this generality.

As an example, consider the family of $H$-free graphs on $n$ vertices.  This family is in one-to-one correspondence with the set of independent sets of  the hypergraph $G_{H}$, with $\binom{n}{2}$ vertices representing the edges of the complete graph $K_n$ and an $e(H)$-uniform hyperedge for each set of edges of $K_n$ that forms a copy of $H$.

We now describe the setting of our main results. 
Let $G = (V,E)$ be a $k$-uniform hypergraph.  For $p \in (0,1)$, let $V_p$ denote a random subset of $V$ chosen by including each vertex independently with probability $p$.  Let $\mathbf X = | \{ e \in E : e \subseteq V_p \}|$ be the number of edges induced by $V_p$.  We will be interested in the probability that $V_p$ is an independent  set, that is 
\begin{equation}
\label{eqnoexistenceDef}
    \P_p(\mathbf X =0) \,,
\end{equation}
 as well as the more general lower-tail probability $\P_p( \mathbf X \le \eta \E \mathbf X)$
for $\eta \in [0,1)$ fixed (noting that $\E \mathbf X = |E| p^k$).  In~\cite{mousset2020probability}  the probability in~\eqref{eqnoexistenceDef} is called the `probability of non-existence in a binomial random subset'.

Our results will apply to a class of hypergraphs that  resemble -- in terms of degrees and codegrees --  linear hypertrees (a hypergraph is linear if two edges intersect in at most one vertex).

 Given $S\subseteq V(G)$ define $d_G(S)=|\{e\in E(G):S\subseteq e\}|$.
For $\ell\geq 1$, define the maximum $\ell$-degree of $G$ to be
\begin{equation}
\label{eqDeltaL}
   \Delta_\ell(G)= \max_{S\subseteq V(G): |S|=\ell} d_G(S)\, . 
\end{equation}
The maximum $1$-degree is the usual notion of maximum vertex degree, and we denote this by $\Delta(G)$.  We denote the minimum vertex degree of $G$ by $\delta(G)$.
Next, the $(k - 1)$-codegree of a pair of distinct vertices $v, v'$, denoted $\Gamma(v,v')$, is the number of tuples $S\in \binom{V(G)}{k-1}$ such that $S\cup\{v\}$ and $S\cup\{v'\}$ are both in $E(G)$. Define the maximum $(k-1)$-codegree to be
\begin{equation}
\label{eqGammaG}
 \Gamma(G)=\max_{v\neq v'} \Gamma(v,v')\, .   
\end{equation}

We now define a notion of asymptotically tree-like and approximately regular hypergraphs.
\begin{definition}\label{def:tree}
    We say a sequence $(G_n)$  of $k$-uniform hypergraphs is \emph{asymptotically tree-like} if the following conditions hold for $G=G_n$ as $n \to \infty$:
    \begin{enumerate}
        \item $\Delta(G) \to \infty$.
          \item For each $\ell \in \{2, \dots, k-1\}$, $\Delta_{\ell}(G)=o( \Delta(G)^{\frac{k-\ell}{k-1}})$.
        \item $\Gamma(G)=o(\Delta(G))$.
        \item $|E(G)|/|V(G)|=\Omega(\Delta(G))$.
    \end{enumerate}
    We further say that $G$ is \emph{approximately regular} if  $\del(G) = (1-o(1)) \Delta(G)$.
\end{definition}

For brevity, we will often say that $G$ is asymptotically tree-like (or approximately regular) when the underlying sequence of hypergraphs is clear from the context.   Intriguingly, very similar conditions on degrees and co-degrees of  hypergraphs appear in at least two other contexts: in the work of Bennett and Bohman~\cite{bennett2016note} on the random greedy algorithm for finding a large independent set in a hypergraph and in the setting of hypergraph containers~\cite{balogh2015independent,saxton2015hypergraph}. The fourth condition, on the average degree, is convenient for the proofs but could be removed.

Our main results will hold under the condition that the probability $p$ is small enough as a function of the vertex degree $\Delta$, the uniformity  of the hyperedges  $k$, and the strength  of the lower-tail event  $\eta$.   We treat the approximately regular case first.

Let us recall the definitions of $x_k^\ast(c,\zeta), \eta_k^\ast$, and $\overline c_k(\eta)$ from~\eqref{eqxastdef}, \eqref{eqEtaStarDef}, \eqref{eqOverlineCeta}:
\begin{align}
      x_k^\ast(c,\zeta) &=  \left( \frac{W((k-1)c^{k-1}\zeta)}{(k-1)\zeta } \right)^{\frac{1}{k-1}} \\
      \eta_k^\ast &= e^{-k/(k-1)} \\
      \overline c_k(\eta) &=
\begin{cases} 
\left( \frac{e}{(k-1)(1-\eta/\eta^\ast_k)} \right)^{\frac{1}{k-1}}   & \text{if } \eta < \eta_k^\ast\\
+ \infty & \text{if } \eta \geq \eta_k^\ast\, .
\end{cases} 
\end{align}

 We  parameterize $p \sim c \cdot \Delta^{-\frac{1}{k-1}}$ (this corresponds to the critical regime in the applications above) and require that $c<\overline c_k(\eta)$. We recall that we let $\mathbf X = | \{ e \in E : e \subseteq V_p \}|$ denote the number of edges induced by a $p$-random subset of $V$.
 
\begin{theorem}
    \label{thmMain}
    Fix $k \ge 2$, $\eta \in[0,1)$, and $0<c<\overline c_k(\eta)$. 
    Let $G_n=(V_n,E_n)$ be a sequence of asymptotically tree-like and approximately regular $k$-uniform hypergraphs with maximum degree $\Delta=\Delta(n)$.  Suppose $p \sim c \Delta^{-\frac{1}{k-1}}$.  Then
\begin{align}
  \lim_{n\to\infty} \Delta^{\frac{1}{k-1}} |V_n|^{-1} \log \P_p ( \mathbf X \le \eta \E \mathbf X) 
  &= x^\ast + (x^\ast)^k  \left ( 1- \frac{1}{k} \right) \zeta  - \log (1-\zeta) \frac{\eta c^k}{k}  - c \,, 
\end{align}
where  $\zeta$ is the unique solution in $[0,1]$ to the equation
 \begin{equation}
 \label{eqZetaDef2}
    (1-\zeta) (x^\ast)^k= \eta c^k\, ,
 \end{equation}
and $x^\ast = x^\ast_k(c,\zeta)$.
 
 In particular,   if $c <  \left(\frac{e} {k-1} \right)^{ \frac{1}{k-1}}$ and  $p \sim c \Delta^{-\frac{1}{k-1}}$, then (taking $\zeta =1$),
    \begin{align}
       \lim_{n\to\infty} \Delta^{\frac{1}{k-1}} |V_n|^{-1}  \log \P_p ( \mathbf X =0)
       &= x^\ast +\left(1-\frac{1}{k} \right)(x^\ast)^k - c \,.
    \end{align}
\end{theorem}

The threshold $\overline c_k(\eta)$ corresponds to the Gibbs uniqueness (or weak spatial mixing) threshold of a statistical physics model  on the infinite $k$-uniform, $\Delta$-regular linear hypertree. This connection is explained below in Section~\ref{secMethods}.

We can also consider the case of choosing $S\subseteq V$ by selecting exactly $m$ vertices uniformly at random from the hypergraph; we denote probabilities in this model by $\P_m$ and as before let $\mathbf{X}$ denote the number of edges induced by this random set of vertices.  In this case we parameterize $m = b |V_n| \Delta ^{- \frac{1}{k-1}}$ and require that $b$ is  small enough as a function of $\eta$ to obtain a formula for the rate function.  
\begin{theorem}
    \label{thmMainGnm}
    Fix $k \ge 2$, $\eta \in [0,1)$, and $b>0$ satisfying 
    \begin{equation}
    \label{eqBcondition}
      (k-1)  b^{k-1} (1-\eta)  < 1 \,.
    \end{equation}
     Let $G_n=(V_n,E_n)$ be a sequence of asymptotically tree-like and approximately regular $k$-uniform hypergraphs with maximum degree $\Delta=\Delta(n)$.  Suppose $m \sim b |V_n| \Delta^{-\frac{1}{k-1}}$.  Then
\begin{equation}
   \lim_{n\to\infty} \Delta^{\frac{1}{k-1}} |V_n|^{-1}  \log \P_m ( \mathbf X \le \eta \E \mathbf X) =   -  \frac{b^k (1- \eta +\eta \log \eta)}{k}  \,.
\end{equation}
In particular, if $b < (k-1)^{- \frac{1}{k-1}}$,
\begin{equation}
   \lim_{n\to\infty} \Delta^{\frac{1}{k-1}} |V_n|^{-1}  \log \P_m ( \mathbf X =0) =   -  \frac{b^k }{k}  \,.
\end{equation}
\end{theorem}
In other words, in this range of parameters $\log \P_m (\mathbf X=0) \sim - \E_m \mathbf X$.

\subsection{Non-regular hypergraphs}

Next we consider hypergraphs that are asymptotically tree-like but not necessarily approximately regular.  Since the structure of such hypergraphs can be quite complex, the formula for log probability will  be more complicated than that in the approximately regular case.  This formula will be in terms of a fixed point of a functional mapping the set of real-valued functions on vertices of $G$ to itself.

Let $G=(V,E)$ be a $k$-uniform hypergraph with maximum degree $\Delta$. Define the function $F_{c,\zeta}^{G}: [0,c]^{V}\to [0,c]^V$  by
\begin{equation}\label{eq:BPF}
    \left(F_{c,\zeta}^{G} (\mathbf x)\right)_v = c \cdot \exp\left(- \frac{\zeta}{\Delta} \sum_{e\ni v}\prod_{u\in e \backslash\{v\}} x_u\right) \,.
\end{equation}
This is an approximate version of the \textit{Belief Propagation (BP) operator} (see, e.g.,~\cite{mezard2009information}) adapted to our setting.  Along with the Belief Propagation operator comes the \textit{Bethe Free Energy}, a function of $G$, $c$, $\zeta$, and $\mathbf x \in [0,c]^V$, defined in our setting by
\begin{equation}\label{eq:Bdef}
    \mathcal B^G_{c,\zeta}(\mathbf x) =  - \frac{\zeta}{\Delta} \sum_{e \in E}  \prod_{u \in e} x_u - \sum_{v \in V} x_v \left[\log \frac{x_v}{c}  -1\right] \,.
\end{equation}

  As a first step, we  show that if $\zeta$ and $c$ are small enough, then there is a unique BP fixed point.  We will use the following condition often.
\begin{condition}
    \label{CondZeta}
    Fix $k \ge 2, c>0$.  Let $\zeta_n \in (0,1]$ so that \begin{equation}
        \limsup_{n \to \infty} \zeta_n (k-1) c^{k-1} <e \,.
    \end{equation}
\end{condition}

\begin{lemma}
\label{lemNonRegularZeroUniqueFixed}
Let $k, c, \zeta_n$ satisfy Condition~\ref{CondZeta}.  For any sequence $G_n = (V_n,E_n)$ of   asymptotically tree-like $k$-uniform hypergraphs, for large enough $n$ there is a unique $\mathbf x^\ast \in [0,c]^{V_n}$ so that 
    \begin{equation}
         F_{c,\zeta_n}^{G_n} (\mathbf x^\ast) = \mathbf x^\ast \,.
    \end{equation}
\end{lemma}

We remark that if $G$ is $\Delta$-regular and $x^\ast$ is the unique fixed point of the equation 
\begin{align}\label{eq:1dBP}
x=c e^{-\zeta x^{k-1}}\, .
\end{align}
then $F^G_{c,\zeta}$ has the constant vector $\mathbf x^\ast\equiv x^\ast$ as a fixed point. We note that $x^\ast=x^\ast_k(c,\zeta)$ where $x^\ast_k(c,\zeta) $ is the quantity introduced at~\eqref{eqxastdef}.

We next show that when $c$ is small enough, there is a choice of $\zeta$ so that the BP fixed point `achieves' the desired lower-tail event. Define $c_k:[0,1)\to \R$ by 
\begin{align}\label{eq:ckdef}
c_k(\eta)=\left(\frac{e}{(1-\eta)(k-1)} \right)^{\frac{1}{k-1}}\, .
\end{align}
Note that $c_k(0) = \overline c_k(0)$, while for $\eta>0$, $c_k(\eta)< \overline c_k(\eta)$.
\begin{lemma}
    \label{lemZetaNonRegular}
    Fix $k \ge 2$, $\eta \in[0,1)$, and $0<c<c_k(\eta)$.  For any sequence $G_n = (V_n,E_n)$ of   asymptotically tree-like $k$-uniform hypergraphs, for $n$ large enough, there exists $\zeta=\zeta(n,c,\eta) \in (0,1]$ satisfying Condition~\ref{CondZeta} so that
    \begin{align}
    (1 - \zeta)  \sum_{e \in E_n} \prod_{u \in e} x^\ast_u   &=(1+o(1)) \eta c^k | E_n| \,,
    \end{align}
    where $\mathbf x^\ast = \mathbf x^\ast(c,\zeta)$ is as guaranteed by Lemma~\ref{lemNonRegularZeroUniqueFixed}. If, in addition, $G_n$ is approximately regular, then the same conclusion holds under the weaker assumption $0<c<\overline c_k(\eta)$.
\end{lemma}

Our main result of this section expresses the rate function in terms of the Bethe free energy evaluated at the  BP fixed point associated to this value of $\zeta$.  
\begin{theorem}
    \label{thmNonRegularRate}
      Fix $k \ge 2$, $\eta \in[0,1)$, and $0<c<c_k(\eta)$.   For any sequence of  asymptotically tree-like $k$-uniform hypergraphs $G_n=(V_n,E_n)$ of maximum degree $\Delta=\Delta(n) $, if  $p \sim c \Delta^{-\frac{1}{k-1}}$, then
     \begin{align}\label{eq:mainBethe}
  \log \P_p ( \mathbf X \le \eta \E \mathbf X)  \sim  \Delta^{-\frac{1}{k-1}}\mathcal B^{G_n}_{c,\zeta}(\mathbf x^\ast) -   \log (1-\zeta)\eta \E\mathbf X  -   p|V_n| \,.
    \end{align}
    where $\zeta=\zeta(n,c,\eta)$ and $\mathbf x^\ast= \mathbf x^\ast(\zeta)$ are given by Lemmas~\ref{lemNonRegularZeroUniqueFixed} and~\ref{lemZetaNonRegular}.  

    In particular,  if $c <  \left(\frac{e} {k-1} \right)^{ \frac{1}{k-1}}$, then
    \begin{align}\label{eq:mainBethe2}
       \Delta^{\frac{1}{k-1}} |V_n|^{-1}  \log \P_p ( \mathbf X =0)
       &=  |V_n|^{-1} \mathcal B^{G_n}_{c,1}(\mathbf x^\ast)    -c +o(1)\,,
    \end{align}
    where $\mathbf x^\ast$ is the unique fixed point of $F_{c,1}^{G_n}$.
\end{theorem}
Note that Theorem~\ref{thmMain} gives the same conclusion for nearly regular hypergraphs under the weaker condition $c<\overline c_k(\eta)$, since it can be checked that the constant function $x^\ast \equiv x_k^\ast(c,\zeta)$ from~\eqref{eqxastdef} is indeed a fixed point of $F_{c,\zeta}^{G}$ in the nearly regular case.
We also note that the limit as $n\to\infty$ of the RHS of~\eqref{eq:mainBethe2} need not exist in general. 
 When the sequence of hypergraphs has additional structure, one can use this theorem to prove the existence of the limit (as in the case of avoiding $k$-APs in Theorem~\ref{thmApavoidInterval}).

\subsection{Methods}
\label{secMethods}

Here we outline the methods we use to prove our results.

We first reformulate the main problem in terms of an appropriate statistical physics model.  The $\eta =0$ (non-existence) case is especially clean and instructive. Let $G$ be a hypergraph on $N$ vertices.  Let $\cI(G)$ denote the set of all independent sets of $G$.   The \textit{hard-core model} on $G$ at activity $\lam \ge 0$ is the distribution $\mu_{G,\lam}$ on $\cI(G)$ defined by
\begin{align}
   \mu_{G,\lam}(I) &= \frac{\lam^{|I|}}{Z_G(\lam)} 
   \intertext{where}
   Z_G(\lam) &= \sum_{I \in \cI(G)} \lam^{|I|} \,.
\end{align}
We call $\mu_{G,\lam}$ the Gibbs measure and $Z_G(\lam)$ the partition function.  The partition function is directly related to  probability of non-existence defined in~\eqref{eqnoexistenceDef}.  If we set $p = \frac{\lam}{1+\lam}$, then we have the identity
\begin{equation}
\label{eqHardCoreProbIdentity}
    \P_p (\mathbf X =0) = (1-p)^{N} Z_G(\lam) \,,
\end{equation}
and so estimating the logarithm of the non-existence probability reduces to estimating $\log Z_G(\lam)$.   For the more general case of the lower-tail problem, one can relate the logarithm of the lower-tail probability to the partition function of a more general edge-penalty Gibbs measure that penalizes (but does not necessarily forbid) hyperedges. For $\lam \ge 0$, $\zeta \in[0,1]$, the edge-penalty model is  distribution on subsets $S\subseteq V(G)$ given by:
\begin{align} \label{eq:mulamzetadef}
   \mu_{G,\lam,\zeta}(S) &= \frac{\lam^{|S|}(1-\zeta)^{|E(S)|}}{Z_G(\lam,\zeta)} 
\end{align}
where $E(S)=\{e\in E(G) : e\subseteq S\}$ and
\begin{align}\label{eq:Zlamzetadef}
   Z_G(\lam,\zeta) &= \sum_{S\subseteq V} \lam^{|S|}(1-\zeta)^{|E(S)|} \,.
\end{align}
This distribution penalizes fully occupied edges by a factor $(1-\zeta)$; taking $\zeta=1$ forbids fully occupied edges and recovers the  hard-core model.

We remark that Condition~\ref{CondZeta} corresponds to the
uniqueness threshold for the edge-penalty model on the infinite $\Delta$-regular linear hypertree. More precisely, on the
infinite $\Delta$-regular linear hypertree, under the critical scaling
$\lambda\sim c\Delta^{-1/(k-1)}$, the BP recursion with constant messages converges to the map $x\mapsto c e^{-\zeta x^{k-1}}$.
This map has a unique fixed point for all $c,\zeta$, but the fixed point is
locally stable when $(k-1)\zeta c^{k-1}<e$ and unstable when
$(k-1)\zeta c^{k-1}>e$. Equivalently, $(k-1)\zeta c^{k-1}=e$ is the threshold for uniqueness of the associated Gibbs measure on the infinite hypertree.

There is a close connection between the value of the log partition function $\log Z_{G}(\lam,\zeta)$ and its \textit{vertex marginals}, the probabilities $\mu_{G,\lam,\zeta}(v \in \mathbf S)$, $v \in V$.  
Computing or estimating marginals in statistical physics models is an important task in many different fields, from physics, to Bayesian statistics, to algorithms, and a variety of approaches to the problem have been proposed, including Markov chain Monte Carlo, variational approaches, and message passing algorithms (see, e.g.,~\cite{koller2009probabilistic,mezard2009information} for background). 

Belief Propagation is an iterative message-passing algorithm that aims to express marginals in terms of a fixed point of an operator relating `messages' between adjacent vertices of a hypergraph (given in our setting by~\eqref{eq:BPF}). The Bethe free energy~\eqref{eq:Bdef} is a functional of these messages that is intended to approximate the log partition function.

One of the main results on Belief Propagation from statistical physics and statistics is that Belief Propagation and the Bethe free energy are exact on trees; that is, if the underlying graph or hypergraph has no cycles, then the BP operator has a unique fixed point. Moreover, evaluating the Bethe free energy at this fixed point gives the exact value of the log partition function~\cite{yedidia2003understanding,mezard2009information}.  Understanding when such a statement is approximately correct in a graph or hypergraph with cycles is a major area of study~\cite{tatikonda2002loopy,dembo2013factor,dembo2014replica}. 

To prove our main results, we will show that BP and the Bethe free energy are approximately correct when $G$ is asymptotically tree-like and $\lam$ and $\zeta$ are small enough.

To do this we need to do three things: 
\begin{enumerate}
    \item Show that the BP operator $F_{c,\zeta}^{G}$ has a unique fixed point in our parameter regime.
    \item Show that this fixed point (a vector indexed by $V(G)$) approximates the vertex marginals of $\mu_{G,\lam,\zeta}$ well.
    \item Show that evaluating the Bethe free energy $\mathcal B^G_{c,\zeta}$ at this fixed point approximates $\log Z_G(\lam,\zeta)$ well.  
\end{enumerate}

To do the first two steps, we use algorithmic ideas from the field of approximate counting and sampling.

As a computational problem, with the hypergraph $G$ as an input, the algorithmic tractability of estimating marginals depends on the parameters of the problem (the uniformity $k$, the maximum degree $\Delta$, the activity parameter $\lam$).  For the case of the hard-core model on graphs ($k=2$), a beautiful connection between computational tractability and statistical physics phase transitions has been established.  We briefly describe this connection since the concepts involved are at the heart of the method we use in this paper. 

Define 
\begin{equation}
    \label{eqLamcD}
    \lam_c(\Delta) = \frac{ (\Delta-1)^{\Delta-1}} { (\Delta-2)^\Delta }  \,.
\end{equation}

This value marks a statistical physics phase transition for the hard-core model on the infinite $\Delta$-regular tree: when $\lam < \lam_c(\Delta)$, there is a unique infinite volume Gibbs measure and when $\lam >\lam_c(\Delta)$ there are multiple distinct Gibbs measures~\cite{kelly1985stochastic} (see also~\cite{georgii2011gibbs} for precise definitions and context). 

Weitz gave an efficient (polynomial-time for fixed $\Delta$) algorithm to estimate marginals $\mu_{G,\lam} ( v \in \mathbf I)$ for any graph $G$ of maximum degree $\Delta$ when $\lam<\lam_c(\Delta)$~\cite{weitz2006counting}.  The complementary hardness result of Sly~\cite{sly2010computational} with extensions in~\cite{sly2014counting,galanis2016inapproximability}, is that there is no polynomial-time algorithm to approximate marginals in the hard-core model on graphs of maximum degree $\Delta$ when $\lam > \lam_c(\Delta)$ unless NP$=$RP.  Thus the computational threshold coincides with the physics phase transition. 

Weitz's algorithm has several components.  The first is the construction of a \textit{computational tree},  $T_v(G)$. This tree is a pruned tree of self-avoiding walks in $G$ starting from a distinguished vertex $v$, and it has the property that for any $\lam$, the hard-core marginal of the root $r$ in $T_v(G)$ is exactly equal to the marginal of $v$ in $G$; that is, $\mu_{T_v(G),\lam}(r \in \mathbf I) = \mu_{G,\lam}(v \in \mathbf I)$.  This is promising since a marginal in a tree can be computed in linear time via message passing algorithms;  however, in general $T_v(G)$ can be exponentially large in the size of $G$.  

To overcome this, Weitz proves that in the hard-core model on the $\Delta$-regular tree, \textit{weak spatial mixing}  implies \textit{strong spatial mixing}.  Weak spatial mixing is the decay as $L \to \infty$ of the dependence of the marginal of the root on boundary conditions imposed at depth $L$ of the tree; this is equivalent to uniqueness of the infinite volume Gibbs measure and holds when $\lam < \lam_c(\Delta)$.  Strong spatial mixing is decay in the difference in marginals at the root between two boundary conditions in the distance from the root to the first disagreement in the boundary conditions.  For the hard-core model, this is equivalent to weak spatial mixing on all subtrees of the infinite $\Delta$-regular tree.  In general strong spatial mixing is much stronger than weak spatial mixing, and so proving their equivalence for the hard-core model on the tree is very significant.   See~\cite{dyer2004mixing} for background on weak and strong spatial mixing. 

The algorithmic significance is that if $G$ has maximum degree $\Delta$, then $T_v(G)$ is a subtree of the $\Delta$-regular tree, and so when $\lam<\lam_c(\Delta)$,  weak spatial mixing holds for the hard-core model on $T_v(G)$, and this in turn means that to approximate the marginal of the root, one can truncate the tree at a certain depth, run a message passing algorithm on the truncated tree, and obtain a good approximation to the true marginal.  The size of the truncated computational tree is exponential only in its depth, not in the size of $G$; truncating at depth $O(\log n)$ gives a polynomial-time algorithm.

The threshold $\lam_c(\Delta)$ is not only the phase transition point on the tree and the algorithmic tractability threshold for graphs of maximum degree $\Delta$, it is also a threshold for complex zero-freeness~\cite{peters2019conjecture}, fast mixing of local Markov chains (Glauber dynamics)~\cite{anari2021spectral}, and more.  A very similar picture extends to more general anti-ferromagnetic two-spin models on graphs~\cite{sinclair2014approximation}.

For the hard-core model on hypergraphs, however, the picture is much more complicated and it seems that none of these thresholds coincide.   For example, consider the infinite $k$-uniform, $\Delta$-regular linear hypertree when $k \ge 3$.  The weak spatial mixing threshold is at $\lam = \Theta( \Delta^{ - \frac{1}{k-1}})$~\cite{bezakova2019approximation}. The strong spatial mixing threshold, however, is at most the strong spatial mixing threshold for graphs, $\lam_c(\Delta) \sim \frac{e}{\Delta}$.  So while Weitz's algorithmic approach can be applied to hypergraphs~\cite{liu2014fptas,bezakova2019approximation}, it is only applicable for $\lam$ far below the tree uniqueness threshold.  Nevertheless, some good behavior of the model does extend past the strong spatial mixing threshold in the bounded-degree setting, including mixing of Markov chains~\cite{hermon2019rapid,jenssen2024sampling} and complex zero-freeness~\cite{liu2025phase}.

The main idea of our proofs is that if a hypergraph has large maximum degree and locally resembles a linear hypertree, then weak spatial mixing on the tree is enough for Weitz's method to work approximately.  In particular, in our setting (aiming for logarithmic asymptotics on hypergraphs with growing degree) it is enough to approximate the vertex marginals with relative error that vanishes as $\Delta \to \infty$.  In particular, for an approximately regular and asymptotically tree-like hypergraph, each marginal matches, to first order, the marginal of the root of the infinite regular hypertree.  This is no longer true for an irregular hypergraph, but the principle is the same: a message-passing algorithm (Belief Propagation) gives the correct marginals to first order.  Thus the general result (Theorem~\ref{thmNonRegularRate}) is stated in terms of the unique fixed point of a Belief Propagation operator.

To connect the marginals to the log partition function,  we can use an identity  relating the derivative of $\log Z_G(\lam,\zeta)$ to $\E_{G,\lam,\zeta}|\mathbf S|$ the expected size of the set $\mathbf S$ drawn from $\mu_{G,\lam,\zeta}$:
\begin{equation}
 \log Z_G(\lam,\zeta) = \int_{0}^\lam \frac{\E_{G,t,\zeta}|\mathbf S|}{t } \, dt  = \int_{0}^\lam  \left (\frac{1}{t} \sum_{v \in V(G)} \mu_{G,t,\zeta}(v \in \mathbf S)\right) \,  dt \,.
\end{equation}

We can then use this identity along with the fixed point conditions to show that the Bethe free energy gives the first-order asymptotics of $\log Z_G(\lam,\zeta)$.

The methods of~\cite{jenssen2024lower} also involved reformulating non-existence and lower-tail problems in terms of a statistical physics model. There, the  cluster expansion was applied after a step of `local conditioning' which established a relation between vertex marginals in the $3$-uniform hypergraph encoding triangles and the vertex marginals in an appropriate graph.  This approach is fundamentally limited to the case of triangles in random graphs. Attempting to apply it to larger cliques or other subgraphs would require applying the cluster expansion for the hard-core model on hypergraphs in a regime where convergence fails~\cite{galvin2024zeroes,zhang2025hypergraph}. Moreover, the local conditioning step uses the specific structure of the triangle; the method fails even for other problems that can be formulated via a $3$-uniform hypergraph such as avoiding $3$-APs.

The use of computational trees for approximate counting and sampling algorithms has proliferated since Weitz's work,  e.g.~\cite{gamarnik2007correlation,li2013correlation,sinclair2014approximation,liu2019fisher,shao2021contraction,chen2023strong}, including~\cite{liu2014fptas,bezakova2019approximation} for hypergraph hard-core models.

Proving the accuracy of Belief Propagation and the Bethe free energy is an important topic in the study of spin models on random graphs.  For example, the works~\cite{gerschenfeld2007reconstruction,dembo2010ising,dembo2013factor,SS14,dembo2014replica,galanis2016ferromagnetic,coja2018belief} carry this out for specific models and specific parameter regimes; some require the graph be random while others work under the condition of few short cycles.  

Local sparsity conditions and other approximate tree-like conditions for hypergraphs arise often in extremal combinatorics. Very similar maximum co-degree conditions arise in the method of hypergraph containers~\cite{balogh2015independent,saxton2015hypergraph}. In~\cite{bennett2016note}, Bennett and Bohman   prove a lower bound on the size of an independent set produced by the random greedy algorithm applied to a $k$-uniform hypergraph satisfying some degree and co-degree conditions  nearly identical to the `asymptotically tree-like' conditions  we require in our main results.

\subsection{Outline}

    In Section~\ref{secPrelim} we give some preliminaries on hypergraphs, Gibbs measures, and partition functions. In Section~\ref{secWeitzConstruction}, we describe the construction of the Weitz hypertree and prove a structural lemma about this construction for asymptotically tree-like hypergraphs. In Section~\ref{secNonRegular} we show that under the conditions of the main results, there is a unique BP fixed point for the appropriate Gibbs measure. In Section~\ref{subsecVariances} we express vertex and edge marginals, bound variances, and approximate the log partition function in terms of this unique BP fixed point.  In Section~\ref{secLowerTailPartition} we relate  lower-tail probabilities to partition functions to prove the main results.  In Section~\ref{secSubgraphs} we derive the results about lower-tails for subgraph counts in random graphs from the main results.  In Section~\ref{secAPs} we do the same for the results on $k$-APs.  Appendix~\ref{App:Contract} contains the proof of a contraction lemma for $k$-APs.

\section{Preliminaries: Hypergraphs and Gibbs measures}
\label{secPrelim}
Here we introduce some basic facts and terminology related to hypergraphs and  Gibbs measures. 

\subsection{Hypergraph basics and terminology}
For our proofs it will be convenient to work with \emph{multihypergraphs}. A multihypergraph is a pair $G=(V,E)$ where $V$ is a set and $E$ is a multiset whose elements are subsets of $V$. We denote by $V$  the set of vertices of the hypergraph and $E$  the set of hyperedges. We highlight that we allow for the empty hyperedge here (with multiplicity). If all hyperedges appear with multiplicity one, then we refer to $G$ as a (simple) hypergraph. If $|e|=k$ for all $e\in E$ then we say that $G$ is $k$-uniform. 

We consider the following ways to modify a multihypergraph $G=(V,E)$.
\begin{enumerate}
\item For $U\subseteq V$, we let $G - U$ denote the submultihypergraph of $G$ induced by $V\setminus U$, that is, the multihypergraph on vertex set $V\setminus U$ and edge multiset $\{e\in E: e\cap U=\emptyset\}$.
\item For $U\subseteq V$, we let $G\ominus U$ denote the multihypergraph obtained by removing the elements of $U$ from each edge of $G$. Formally,  $G\ominus U$ is the hypergraph on vertex set $V\setminus U$ and edge multiset $\{e\setminus U: e\in E\}$.
\item If $F\subseteq E$ (as multisets), we let $G-F$ denote the subhypergraph of $G$ obtained by deleting the edges in $F$, that is, $G-F$ is the hypergraph on vertex set $V$ and edge multiset $E\backslash F$.
\end{enumerate}

 A \textit{self-avoiding walk} (SAW) of length $\ell$ from $v$ to $u$ in $G$ is an alternating sequence $(v_1, e_1, v_2, e_2, \dots, e_{\ell},v_{\ell+1})$ of vertices and edges from $G$ with the following properties:
\begin{itemize}
    \item $v_1 = v$, $v_{\ell+1}=u$
    \item $v_1,\ldots, v_{\ell+1}$ and $e_1,\ldots, e_{\ell}$ are all distinct\footnote{Here we treat copies of the same edge in a multihypergraph as distinct.}
    \item $\{v_i,v_{i+1}\}\subseteq e_i$ for $i=1,\ldots, \ell$.
\end{itemize}
Note that for $v\in V$, we count $(v)$ as a SAW of length $0$. 

A linear multihypergraph $G=(V,E)$ is one in which any two distinct edges $e,f \in E$ intersect in at most one vertex. A linear hypertree is a linear multihypergraph in which there is a unique SAW between any distinct pair of vertices, see Figure~\ref{fig:hypertree}. We note that this definition of linear hypertree allows for edges of size $1$ to have multiplicity $>1$.

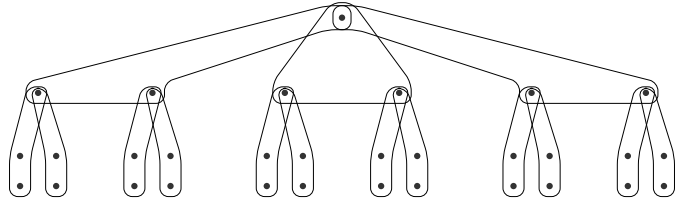
\begin{figure}[htbp]
  \centering
  \scalebox{0.4}{
\begin{tikzpicture}[
    vertex/.style={
        circle,
        fill=black!80,
        draw=none,
        inner sep=0pt,
        minimum size=6pt
    },
    hyperedge/.style={
        rounded corners=8pt,
        draw=black,
        thick,
        fill=none
    },
]

\node[vertex] (r) at (1.7, 0) {};

% --- Depth-1 vertices ---
\node[vertex] (v1) at (-8.4, -2.5) {};
\node[vertex] (v2) at (-4.6, -2.5) {};
\node[vertex] (v3) at (-0.2, -2.5) {};
\node[vertex] (v4) at (3.6, -2.5) {};
\node[vertex] (v5) at (8.0, -2.5) {};
\node[vertex] (v6) at (11.8, -2.5) {};

\node[vertex] (v1a1) at (-9.0, -4.6) {};
\node[vertex] (v1a2) at (-9.0, -5.6) {};
\node[vertex] (v1b1) at (-7.8, -4.6) {};
\node[vertex] (v1b2) at (-7.8, -5.6) {};

\node[vertex] (v2a1) at (-5.2, -4.6) {};
\node[vertex] (v2a2) at (-5.2, -5.6) {};
\node[vertex] (v2b1) at (-4.0, -4.6) {};
\node[vertex] (v2b2) at (-4.0, -5.6) {};

\node[vertex] (v3a1) at (-0.8, -4.6) {};
\node[vertex] (v3a2) at (-0.8, -5.6) {};
\node[vertex] (v3b1) at (0.4, -4.6) {};
\node[vertex] (v3b2) at (0.4, -5.6) {};

\node[vertex] (v4a1) at (3.0, -4.6) {};
\node[vertex] (v4a2) at (3.0, -5.6) {};
\node[vertex] (v4b1) at (4.2, -4.6) {};
\node[vertex] (v4b2) at (4.2, -5.6) {};

\node[vertex] (v5a1) at (7.4, -4.6) {};
\node[vertex] (v5a2) at (7.4, -5.6) {};
\node[vertex] (v5b1) at (8.6, -4.6) {};
\node[vertex] (v5b2) at (8.6, -5.6) {};

\node[vertex] (v6a1) at (11.2, -4.6) {};
\node[vertex] (v6a2) at (11.2, -5.6) {};
\node[vertex] (v6b1) at (12.4, -4.6) {};
\node[vertex] (v6b2) at (12.4, -5.6) {};

\begin{scope}[on background layer]

    % e1: {r, v1, v2} — extends left
    \draw[hyperedge]
        ($(r) + (0.3, 0.4)$) --
        ($(r) + (-0.5, 0.4)$) --
        ($(v1) + (-0.4, 0.35)$) --
        ($(v1) + (-0.4, -0.35)$) --
        ($(v2) + (0.4, -0.35)$) --
        ($(v2) + (0.4, 0.35)$) --
        ($(r) + (-0.5, -0.4)$) --
        ($(r) + (0.3, -0.4)$) -- cycle;

    % e2: {r, v3, v4} — extends center
    \draw[hyperedge]
        ($(r) + (-0.35, 0.5)$) --
        ($(v3) + (-0.4, 0.35)$) --
        ($(v3) + (-0.4, -0.35)$) --
        ($(v4) + (0.4, -0.35)$) --
        ($(v4) + (0.4, 0.35)$) --
        ($(r) + (0.35, 0.5)$) -- cycle;

    % e3: {r, v5, v6} — extends right
    \draw[hyperedge]
        ($(r) + (-0.3, -0.4)$) --
        ($(r) + (0.5, -0.4)$) --
        ($(v5) + (-0.4, 0.35)$) --
        ($(v5) + (-0.4, -0.35)$) --
        ($(v6) + (0.4, -0.35)$) --
        ($(v6) + (0.4, 0.35)$) --
        ($(r) + (0.5, 0.4)$) --
        ($(r) + (-0.3, 0.4)$) -- cycle;

    % Depth-1 hyperedges (vertical capsules)
    \draw[hyperedge]
        ($(v1) + (-0.35, 0.2)$) --
        ($(v1a1) + (-0.35, 0.0)$) --
        ($(v1a2) + (-0.35, -0.35)$) --
        ($(v1a2) + (0.35, -0.35)$) --
        ($(v1a1) + (0.35, 0.0)$) --
        ($(v1) + (0.35, 0.2)$) -- cycle;

    \draw[hyperedge]
        ($(v1) + (-0.25, 0.15)$) --
        ($(v1b1) + (-0.35, 0.0)$) --
        ($(v1b2) + (-0.35, -0.35)$) --
        ($(v1b2) + (0.35, -0.35)$) --
        ($(v1b1) + (0.35, 0.0)$) --
        ($(v1) + (0.25, 0.15)$) -- cycle;

    \draw[hyperedge]
        ($(v2) + (-0.35, 0.2)$) --
        ($(v2a1) + (-0.35, 0.0)$) --
        ($(v2a2) + (-0.35, -0.35)$) --
        ($(v2a2) + (0.35, -0.35)$) --
        ($(v2a1) + (0.35, 0.0)$) --
        ($(v2) + (0.35, 0.2)$) -- cycle;

    \draw[hyperedge]
        ($(v2) + (-0.25, 0.15)$) --
        ($(v2b1) + (-0.35, 0.0)$) --
        ($(v2b2) + (-0.35, -0.35)$) --
        ($(v2b2) + (0.35, -0.35)$) --
        ($(v2b1) + (0.35, 0.0)$) --
        ($(v2) + (0.25, 0.15)$) -- cycle;

    \draw[hyperedge]
        ($(v3) + (-0.35, 0.2)$) --
        ($(v3a1) + (-0.35, 0.0)$) --
        ($(v3a2) + (-0.35, -0.35)$) --
        ($(v3a2) + (0.35, -0.35)$) --
        ($(v3a1) + (0.35, 0.0)$) --
        ($(v3) + (0.35, 0.2)$) -- cycle;

    \draw[hyperedge]
        ($(v3) + (-0.25, 0.15)$) --
        ($(v3b1) + (-0.35, 0.0)$) --
        ($(v3b2) + (-0.35, -0.35)$) --
        ($(v3b2) + (0.35, -0.35)$) --
        ($(v3b1) + (0.35, 0.0)$) --
        ($(v3) + (0.25, 0.15)$) -- cycle;

    \draw[hyperedge]
        ($(v4) + (-0.35, 0.2)$) --
        ($(v4a1) + (-0.35, 0.0)$) --
        ($(v4a2) + (-0.35, -0.35)$) --
        ($(v4a2) + (0.35, -0.35)$) --
        ($(v4a1) + (0.35, 0.0)$) --
        ($(v4) + (0.35, 0.2)$) -- cycle;

    \draw[hyperedge]
        ($(v4) + (-0.25, 0.15)$) --
        ($(v4b1) + (-0.35, 0.0)$) --
        ($(v4b2) + (-0.35, -0.35)$) --
        ($(v4b2) + (0.35, -0.35)$) --
        ($(v4b1) + (0.35, 0.0)$) --
        ($(v4) + (0.25, 0.15)$) -- cycle;

    \draw[hyperedge]
        ($(v5) + (-0.35, 0.2)$) --
        ($(v5a1) + (-0.35, 0.0)$) --
        ($(v5a2) + (-0.35, -0.35)$) --
        ($(v5a2) + (0.35, -0.35)$) --
        ($(v5a1) + (0.35, 0.0)$) --
        ($(v5) + (0.35, 0.2)$) -- cycle;

    \draw[hyperedge]
        ($(v5) + (-0.25, 0.15)$) --
        ($(v5b1) + (-0.35, 0.0)$) --
        ($(v5b2) + (-0.35, -0.35)$) --
        ($(v5b2) + (0.35, -0.35)$) --
        ($(v5b1) + (0.35, 0.0)$) --
        ($(v5) + (0.25, 0.15)$) -- cycle;

    \draw[hyperedge]
        ($(v6) + (-0.35, 0.2)$) --
        ($(v6a1) + (-0.35, 0.0)$) --
        ($(v6a2) + (-0.35, -0.35)$) --
        ($(v6a2) + (0.35, -0.35)$) --
        ($(v6a1) + (0.35, 0.0)$) --
        ($(v6) + (0.35, 0.2)$) -- cycle;

    \draw[hyperedge]
        ($(v6) + (-0.25, 0.15)$) --
        ($(v6b1) + (-0.35, 0.0)$) --
        ($(v6b2) + (-0.35, -0.35)$) --
        ($(v6b2) + (0.35, -0.35)$) --
        ($(v6b1) + (0.35, 0.0)$) --
        ($(v6) + (0.25, 0.15)$) -- cycle;

\end{scope}

\end{tikzpicture}
}
\caption{A 3-uniform  linear hypertree.}
  \label{fig:hypertree}
\end{figure}

\subsection{Gibbs measures and partition functions}
For $\lam\geq 0$ and $\zeta\in[0,1]$ and hypergraph $G$, recall the definition of the Gibbs measure $\mu_{G,\lam,\zeta}$ and partition function $Z_G(\lam,\zeta)$ from~\eqref{eq:mulamzetadef} and~\eqref{eq:Zlamzetadef}. These objects naturally extend to the setting of multihypergraphs where now we interpret the parameter $E(S)=\{e\in E(G) : e\subseteq S\}$ with multiplicity.

For $v\in V$, let $Z^{\text{in}}_v(G,\lam,\zeta), Z^{\text{out}}_v(G,\lam,\zeta)$ denote the contribution to the partition function from sets containing, avoiding $v$ respectively. More formally, 
\begin{align}
    Z^{\text{in}}_v(G,\lam,\zeta)= \sum_{S\subseteq V: v\in S} \lam^{|S|}(1-\zeta)^{|E(S)|}
\end{align}
and $Z^{\text{out}}_v(G,\lam,\zeta)=Z_G(\lam,\zeta)-Z^{\text{in}}_v(G,\lam,\zeta)$. For $U\subseteq V$, we let 
\begin{align}
    Z^{\text{in}}_U(G,\lam,\zeta)= \sum_{S\subseteq V: U\subseteq S} \lam^{|S|}(1-\zeta)^{|E(S)|}\, .
\end{align}
For $v\in V$, let 
\begin{align}\label{eq:Rvdef}
R_v(G,\lam,\zeta)=\frac{Z^{\text{in}}_v(G,\lam,\zeta)}{Z^{\text{out}}_v(G,\lam,\zeta)} = \frac{\mu_{G,\lam,\zeta}(v\in \mathbf{S})}{\mu_{G,\lam,\zeta}(v\notin \mathbf{S})}\, .
\end{align}

We record the following basic observations.
\begin{obs}\label{obs:identities}
    Let $G=(V,E)$ be a multihypergraph with $u,v\in V$ and $e\in E$. We have
    \begin{enumerate}
        \item $Z^{\text{in}}_v(G,\lam,\zeta)=\lam Z_{G\ominus v}(\lam,\zeta)$. \label{item:vominus}
        \item $Z^{\text{out}}_v(G,\lam,\zeta)= Z_{G-v}(\lam,\zeta)$. \label{item:vminus}
        \item $Z_{G}(\lam,\zeta)= Z_{G-\{e\}}(\lam,\zeta) - \zeta Z_e^{\text{in}}(G-\{e\},\lam,\zeta)$.\label{item:edelete}
        \item $\mu_{G,\lam,\zeta}(u\in \mathbf S\mid v\in \mathbf S)=\mu_{G\ominus v,\lam,\zeta}(u\in \mathbf S )$. \label{item:condcontract}
    \end{enumerate}
\end{obs}

In light of this observation we describe the operation $G\ominus v$ as \emph{setting $v$ to be occupied} and the operation $G - v$ as \emph{setting $v$ to be unoccupied}.

With this observation in hand, we record a tree recursion for hypertrees adapted from~\cite{bencs2023optimal} to the setting of the $\mu_{G,\lam,\zeta}$.  Note that we employ the usual convention that the empty product is equal to $1$.

\begin{lemma}\label{lem:tree_recursion2}
Let $T=(V,E)$ be a linear hypertree and let $v\in V$. Suppose that $\{e_1,\dots,e_d\}$ is the set of edges incident to $v$. For $u\in N_G(v)$, let
$T_u$ be the connected component of $u$ in $T-v$. Then we have
\[
R_{v}(T,\lambda,\zeta)=\lambda\prod_{i=1}^d\left(1-\zeta\prod_{u\in e_i\backslash\{v\}}\frac{R_u(T_u,\lambda,\zeta)}{1+R_u(T_u,\lambda,\zeta)}\right).
\]
\end{lemma}
\begin{proof}
For brevity we write $R_{v}(G)$ in place of $R_{v}(G,\lam,\zeta)$ and similarly drop $\lam, \zeta$ from partition function notation. By Observation~\ref{obs:identities} part \eqref{item:vominus},
\[
R_{v}(T)=\frac{Z_{v}^\textrm{in}(T)}{Z_{v}^\textrm{out}(T)}=\frac{\lambda  Z(T\ominus v) }{Z(T-v) }
\]
Note that $T-v$ consists of the disjoint union of the trees $T_u$ where $u\in N_G(v)$. Moreover, $T\ominus v$ is the union of $T-v$ and the edges $e_i'=e_i\backslash\{v\}$ for $i=1,\ldots,d$. $T\ominus v$ is therefore a disjoint union of trees $T_1,\ldots, T_d$ where $T_i$ is the union $\cup_{u\in e_i'}T_u$ with the edge $e_i'$ added\footnote{It may be the case that $e_i=\{v\}$ for some $i$ in which case $T_i=(\emptyset, \{\emptyset\})$ i.e. the hypergraph with empty vertex set and a single hyperedge which is the empty hyperedge. In this case $Z(T_i)=1-\zeta$.}. Note that by Observation~\ref{obs:identities} part~\eqref{item:edelete},
\[
Z(T_i)= Z(T_i-e_i')-\zeta\cdot Z_{e_i'}^{\textrm{in}}(T_i-e_i')\, ,
\]
and
\[
Z(T_i-e_i')=\prod_{u\in e_i'}Z(T_u) \quad\text{and}\quad Z_{e_i'}^{\textrm{in}}(T_i-e_i')=\prod_{u\in e_i'}Z^{\textrm{in}}_{u}(T_u)\, ,
\]
for all $i$.
We conclude that
\begin{align}
R_{v}(T)= \lam    \prod_{i=1}^d\frac{Z(T_i)}{Z(T_i-e_i')}
&=
\lam    \prod_{i=1}^d\left(1-\zeta \frac{Z_{e_i'}^{\textrm{in}}(T_i-e_i')}{Z(T_i-e_i')} \right)\\
&=
\lam    \prod_{i=1}^d\left(1-\zeta \prod_{u\in e_i'}\frac{Z^{\textrm{in}}_{u}(T_u)}{Z(T_u)} \right)\\
&=
\lambda   \prod_{i=1}^d\left(1-\zeta\prod_{u\in e_i'}\frac{R_u(T_u)}{1+R_u(T_u)}\right)\, .
\end{align}
\end{proof}

We will also make use of the following simple fact. Given a set $V$ and $p\in[0,1]$, we let $\mu_{V,p}$ denote the law of the $p$-random set $V_p$.
\begin{lemma}\label{lemStochDom}
For every (multi)hypergraph $G=(V,E)$ and every $\lam \ge 0$, $\zeta\in [0,1]$, the distribution $\mu_{G,\lam,\zeta}$ is stochastically dominated by $\mu_{V,p}$ where $p=\lam/(1+\lam)$. That is, there is a coupling of $\mathbf S\sim \mu_{G,\lam,\zeta}$ and $\mathbf S' \sim \mu_{V,p}$ such that $\mathbf S\subseteq \mathbf S'$ with probability 1. 
\end{lemma}

\begin{proof}
Sample $\mathbf S \sim\mu_{G,\lam,\zeta}$ by sampling one vertex of $G$ at a time with the correct conditional probability given the previous history. The probability of including a vertex $v$ is then
\[
\frac{\lam (1 - \zeta)^{t(v)}}{1+\lam (1 - \zeta)^{t(v)}}\leq \frac{\lam}{1+\lam}=p \,,
\]
where $t(v)$ denotes the number of edges of $G$ that $v$ forms with vertices included in previous steps.
 We can therefore couple the sampling with a vertex-by-vertex sampling of $\mu_{V,p}$ so that a vertex is present in the sample from $\mu_{G,\lam,\zeta}$ only if it is present in the sample from $\mu_{V,p}$.  
\end{proof}

\section{The Weitz hypertree}
\label{secWeitzConstruction}

\subsection{Construction of the Weitz hypertree}

In this section we present the construction of the Weitz hypertree, described above in the introduction.  This construction appears explicitly in~\cite[Section 2.2]{bencs2023optimal} and implicitly in~\cite{liu2014fptas}. In~\cite{bencs2023optimal} the Weitz hypertree is defined inductively; here we give a more constructive description. Given a multihypergraph $G=(V,E)$ of maximum degree $\Delta$ and a distinguished vertex $v\in V$, we will construct a rooted linear hypertree, $T_v(G)$, of maximum degree at most $\Delta$ with the property that for all $\lam, \zeta$, the marginal of $v$ in the edge-penalty model on $G$ is exactly equal to the marginal of the root in the edge-penalty model on $T_v(G)$.

Given a multihypergraph $G = (V,E)$ and a vertex $v \in V$,
we can construct a linear hypertree of self-avoiding walks $T_{\mathrm{SAW},v}(G)$, with vertices of this hypertree labeled by vertices of $G$.   To avoid confusion we will refer to vertices of $G$ as \emph{vertices}, and vertices of $T_{\mathrm{SAW},v}(G)$ as \emph{nodes}; each node will be labeled with a vertex of $G$, but vertices may appear as labels of multiple nodes. Furthermore, we use $u,v,w\ldots$ to denote vertices of $G$ and $\bsf{u},\bsf{v},\bsf{w},\ldots$ to denote nodes of the tree, and $\bsf{r}$ to denote the root node of the tree, when the context is clear. To define the construction formally, we first introduce some notation.

Given a SAW $\bsf{w}=(v_1,e_1,\ldots, v_{\ell})$ in $G$, we let $\pi(\bsf{w})=v_\ell$, the terminating vertex of $\bsf{w}$. We say that a SAW $\bsf{w}'$ \emph{extends} $\bsf{w}$, denoted $\bsf{w}\prec \bsf{w}'$, if $\bsf{w}'=(v_1,e_1,\ldots, v_{\ell}, e_\ell, v_{\ell+1})$. We now define $T_{\mathrm{SAW},v}(G)$: it is the hypergraph whose set of nodes is the set of all SAWs in $G$ starting from $v$ and $\{\bsf{w}, \bsf{w}_1,\ldots, \bsf{w}_{t}\}$ ($t\geq 1$) forms an edge if $\bsf{w} \prec \bsf{w}_i$ for all $i$; $\bsf{w}_1,\ldots, \bsf{w}_{t}$ differ only in their final coordinate; and $\{\pi(\bsf{w}), \pi(\bsf{w}_1),\ldots, \pi(\bsf{w}_{t})\}\in E$. Moreover, if $\{\pi(\bsf w)\}\in E$ with multiplicity $m$, then $\{\bsf w\}$ is an edge of $T_{\mathrm{SAW},v}(G)$ with multiplicity $m$.

We think of $\pi$ as a labeling of the nodes of $T_{\mathrm{SAW},v}(G)$ by vertices in $G$.
Note that if $\bsf{e}=\{\bsf{w}_1, \ldots, \bsf{w}_k\}$ is an edge of $T_{\mathrm{SAW},v}(G)$ then $\pi(\bsf{e})=\{\pi(\bsf{w}_1), \ldots, \pi(\bsf{w}_k)\}$ is an edge of $G$ and we refer to $\pi(\bsf{e})$ as the label of $\bsf{e}$. We call the node $\bsf{r}$ where $\bsf{r}=(v)$, the \emph{root} of $T_{\mathrm{SAW},v}(G)$.
If for $\bsf{w}, \bsf{w}'$ there exists a sequence $\bsf{w}=\bsf{w}_1\prec \ldots \prec \bsf{w}_\ell=\bsf{w}'$ we say that $\bsf{w}'$ is a \emph{descendant} of $\bsf{w}$ and $\bsf{w}$ is an \emph{ancestor} of $\bsf{w}'$. We refer to the tree induced by $\bsf{w}$ and its descendants as the \emph{subtree rooted at $\bsf{w}$}. Moreover, we refer to the edge of $T_{\mathrm{SAW},v}(G)$ that contains $\bsf{w}$ and none of its descendants as the \emph{parent edge} of $\bsf{w}$. We call the unique node $\bsf{w}'$ in the parent edge of $\bsf{w}$ such that $\bsf{w}'\prec \bsf{w}$ the \emph{parent node} of $\bsf{w}$. If the SAW $\bsf{w}$ has length $\ell$ we say $\bsf{w}$ is at \emph{depth} $\ell$ in $T_{\mathrm{SAW},v}(G)$. 
We note that by construction $T_{\mathrm{SAW},v}(G)$ is a linear hypertree. Moreover, the maximum size of an edge in $T_{\mathrm{SAW},v}(G)$ is at most that in $G$ and the same holds for the maximum degree.

The \textit{Weitz hypertree}, $T_v(G)$, is a linear hypertree which we obtain from $T_{\mathrm{SAW},v}(G)$ by operations of removing edges and removing nodes from edges (i.e.\ setting nodes to be occupied).  These operations will depend on two orderings: an arbitrary ordering of $V$ and an arbitrary ordering of $E$, both of which we denote by $<$. The node and edge operations are as follows. 

    Suppose $\bsf{w}$ has a parent edge $\bsf{e}$ with label $\pi(\bsf{e})=\{u_1,\ldots, u_k\}$ and parent vertex $\bsf{w}'$ and $f_1, \ldots, f_d$ are the edges of $G$ incident to $\pi(\bsf{w}')$. Then the node/edge operations on $T_{\mathrm{SAW},v}(G)$ at $\bsf{w}$ are
     \begin{enumerate}
      \item\label{op1}  set any descendant of $\bsf{w}$ with label $u_i$ where $u_i< \pi(\bsf{w})$ to be occupied,
      
      \item\label{op2} delete all edges in the subtree rooted at $\bsf{w}$ with label $f_i$ where  $f_i< \pi(\bsf{e})$.
     \end{enumerate}

It will be useful to make the following convention: when we set a node to be occupied in the above procedure, any edge that gets contracted as a result retains its original label.
To obtain $T_v(G)$ from $T_{\mathrm{SAW},v}(G)$ we apply the operations~\eqref{op1} and ~\eqref{op2} at each (non-root) node $\bsf{w}$ of $T_{\mathrm{SAW},v}(G)$ (unless it has been set to occupied by a previous operation).
The final tree $T_v(G)$ is the connected component of the root $\bsf{r}$ after performing these operations.  We note that if $G$ is a linear hypertree and $v\in V(G)$, then both  $T_{\mathrm{SAW},v}(G)$ and $T_v(G)$ are isomorphic to $G$.

The significance of the Weitz hypertree is that the marginal of the root  in the edge-penalty model on $T_v(G)$ is equal to the marginal of $v$ in the edge-penalty  model on $G$. This is proved (for simple hypergraphs) in~\cite{bencs2023optimal} for the hard-core model i.e.\ the case $\zeta =1$. We give the proof of the extension to the edge-penalty model here. 

\begin{lemma}
\label{lemBencs}
Let $G = (V,E)$ be a multihypergraph and let $v \in V$. For any $\lam \ge 0$, $\zeta \in[0,1]$ we have,
    \[ \mu_{G,\lam, \zeta} ( v \in \mathbf S) = \mu_{T_v(G),\lam, \zeta} ( \bsf{r} \in \mathbf S) \, ,\]
    where $\bsf{r}$ is the root of $T_v(G)$.
\end{lemma}

\begin{proof}[Proof of Lemma~\ref{lemBencs}]
We begin by observing that we may assume that $G$ does not contain $\{v\}$ as an edge. Indeed, suppose that $G$ contains the edge $\{v\}$ with multiplicity $m\geq 1$ and let $G'$ denote $G$ with all copies of the edge $\{v\}$ removed. Let  $T=T_v(G)$ and note that $T$ contains the edge $\{\bsf{r}\}$ with multiplicity $m$ by definition. We let $T'$ denote $T$ with all edge $\{\bsf{r}\}$ removed. We observe that $R_v(G,\lam,\zeta)=(1-\zeta)^mR_v(G',\lam,\zeta)$ and $R_\bsf{r}(T,\lam,\zeta)=(1-\zeta)^mR_\bsf{r}(T',\lam,\zeta)$. Moreover, $T'=T_v(G')$ since no SAW in $G$ can contain the edge $\{v\}$ and $\bsf r$ is the only node in $T$ with label $v$. It therefore suffices to show that $R_v(G',\lam,\zeta)= R_\bsf{r}(T_v(G'),\lam,\zeta)$ i.e. we may assume that $m=0$.

For brevity, let us denote $R_v(G) = R_v(G,\lam,\zeta)$. When the context is clear, we also drop $\lam, \zeta$ from partition function notation. 

Our goal is to show that $R_v(G)=R_{\bsf{r}}(T)$. We proceed by induction on $|V|$. If $V=\{v\}$ then $G$ and $T$ both have one vertex and no edges so the result follows. Suppose then that $|V|\geq 2$. Let $v\in V$ and let $e_1<e_2<\ldots<e_d$ denote the edges of $G$ incident to $v$. Let us denote $e_i=\{v,v^i_1,\ldots, v^i_{\ell_i}\}$ where $v^i_1<\ldots< v^i_{\ell_i}$  and let $E_i=\{e_1,\ldots, e_i\}$ for $i\in [d]$. Noting that $G-v=G-E_d\ominus v$ we have
 \begin{align}\label{eq:RvG}
R_v(G)=\lam\cdot \frac{ Z(G\ominus v)}{Z((G-E_d)\ominus v)}=\lam\cdot \prod_{i=1}^d \frac{Z((G-E_{i-1})\ominus v)}{Z((G-E_{i})\ominus v)}\, .
 \end{align}

 Let $\hat{G}_i=(G-E_{i-1})\ominus v$ and $\hat{e}_i := e_i \setminus \{v\}$, and note that part~\eqref{item:edelete} in Observation~\ref{obs:identities} gives us 
 \begin{align}\label{eq:RvG2}
\frac{Z((G-E_{i-1})\ominus v)}{Z((G-E_{i})\ominus v)}=\frac{Z(\hat{G}_i)}{Z(\hat{G}_i-\{\hat{e}_i\})}=1-\zeta\cdot\frac{ Z_{\hat{e}_i}^{\text{in}}(\hat{G}_i-\{\hat{e}_i\})}{Z(\hat{G}_i-\{\hat{e}_i\})}.
\end{align}

We now consider $R_{\bsf{r}}(T)$. Note that in $T$, the root $\bsf{r}$ has incident edges $\bsf{e}_1',\ldots, \bsf{e}_d'$ where $\pi(\bsf{e}_i')=e_i$ for each $i$. Write $\bsf{e}'_i=\{\bsf{r},\bsf{r}^i_1,\ldots, \bsf{r}^i_{\ell_i}\}$ where $\pi(\bsf{r}^i_j)=v^i_j$. Let $T=T_v(G)$, then by  Lemma~\ref{lem:tree_recursion2}
\begin{align}
R_{\bsf{r}}(T)= \lam \cdot\prod_{i=1}^d\left(1-\zeta\prod_{j=1}^{\ell_i}\frac{R_{\bsf{r}^{i}_j}(T_j^i)}{1+R_{\bsf{r}^{i}_j}(T_j^i)}\right)
\end{align}
where $T_j^i$ is the component of $\bsf{r}_j^i$ in $T-\bsf{r}$.
Comparing this to~\eqref{eq:RvG} and~\eqref{eq:RvG2} we see that it suffices to show  
\begin{align}\label{eq:treejump}
\frac{ Z_{\hat{e}_i}^{\text{in}}(\hat{G}_i-\{\hat{e}_i\})}{Z(\hat{G}_i-\{\hat{e}_i\})} = \prod_{j=1}^{\ell_i}\frac{R_{\bsf{r}^{i}_j}(T_j^i)}{1+R_{\bsf{r}^{i}_j}(T_j^i)}\quad \forall i\in [d].
\end{align}
Let $G_j^i=(G-E_i)\ominus\{v,v_1^i\ldots, v_{j-1}^i\}$ and let $S^i_{j}=\{v_1^i,\ldots, v^i_{j}\}$. Then 
\begin{align}\label{eq:einGprime}
    \frac{ Z_{\hat{e}_i}^{\text{in}}(\hat{G}_i-\{\hat{e}_i\})}{Z(\hat{G}_i-\{\hat{e}_i\})}=\prod_{j=1}^{\ell_i} \frac{ Z_{S^i_j}^{\text{in}}(\hat{G}_i-\{\hat{e}_i\})}{Z_{S^i_{j-1}}^{\text{in}}(\hat{G}_i-\{\hat{e}_i\})} = \prod_{j=1}^{\ell_i} \frac{ Z_{v^i_{j}}^{\text{in}}(G_j^i)}{Z(G_j^i)}= \prod_{j=1}^{\ell_i}  \frac{R_{v^i_{j}}(G_j^i)}{1+R_{v^i_{j}}(G_j^i)} \, ,
\end{align}
where we used Observation~\ref{obs:identities} part~\eqref{item:vominus} for the penultimate equality. The result follows once we prove the following claim. 

\begin{claim}\label{claim:subtree}
Let $G_j^i$ inherit the ordering of its vertices and edges from $G$ in the natural way. Then there exists an isomorphism
\[
T_{v_j^i}(G_j^i) \cong T_j^i\, .
\]
which maps the root of $T_{v_j^i}(G_j^i)$ to $\bsf{r}_{j}^i$. 
\end{claim}
Indeed, then by the induction hypothesis and Claim~\ref{claim:subtree}, we have 
\[
R_{v^i_{j}}(G_j^i) = R_{\bsf{r}_j^i}(T_j^i)\, .
\]
Our desired equality~\eqref{eq:treejump} then follows from~\eqref{eq:einGprime}. It remains to prove Claim~\ref{claim:subtree}.

\begin{proof}[Proof of Claim~\ref{claim:subtree}]
For brevity, let $F=E_i$ and let $U=\{v,v_1^i\ldots, v_{j-1}^i\}$ so that $G_j^i=(G-F)\ominus U$. Moreover, let $\tilde F=\{f\in E(T_{\mathrm{SAW}, v}(G)): \pi(f)\in F\}$ and $\tilde U=\{u\in V(T_{\mathrm{SAW}, v}(G)): \pi(u)\in U\}$, the set of edges and vertices in $T_{\mathrm{SAW}, v}(G)$ whose label lies in $F$ and $U$ respectively.
We will first show that $T_{\mathrm{SAW}, v_j^i}(G_j^i)$ is isomorphic to the component containing $\bsf{r}_j^i$ in $(T_{\mathrm{SAW}, v}(G)-\tilde F) \ominus \tilde U$. The result will then follow quickly. 

Note that every node of $T_{\mathrm{SAW}, v_j^i}(G_j^i)$ is of the form 
\[
\bsf{w}=(v_1,h_1\backslash U,\ldots, h_\ell\backslash U,v_{\ell+1})
\]
where $v_1=v_j^i$, each $h_i$ is an edge of $G-F$, and each $v_i\notin U$. We will show that the map 
\[
\phi: \bsf{w} \mapsto (v,e_i,v_1,h_1,\ldots, h_\ell,v_{\ell+1})
\]
is our desired isomorphism.  First note that $\phi(\bsf{w})$ is indeed a SAW in $G$ with label $v_{\ell+1}\notin U$ so that $\phi(\bsf{w})$ is a node of $(T_{\mathrm{SAW}, v}(G)-\tilde F) \ominus \tilde U$. Moreover, $\phi$ is injective and maps the root of $T_{\mathrm{SAW}, v_j^i}(G_j^i)$ to $\bsf{r}_{j}^i=(v,e_i,v_j^i)$. Now, any edge of $T_{\mathrm{SAW}, v_j^i}(G_j^i)$ is of the form
\[
\Gamma=\{\bsf{w}\}\cup \{(\bsf{w},h\backslash U,x) : x\in h\backslash (U \cup \pi(\bsf{w})) \}
\]
where $\bsf{w}=(v_1,h_1\backslash U,\ldots, h_\ell\backslash U,v_{\ell+1})$ is a node of $T_{\mathrm{SAW}, v_j^i}(G_j^i)$, $h\in E(G)-F$ is distinct from $e_i,h_1,\ldots, h_\ell$ and each choice of $x$ is distinct from $v,v_1,\ldots, v_{\ell+1}$. Moreover, none of $v,v_1,\ldots, v_{\ell+1}$ belong to $U$. It follows that 
\[
\{\phi(\bsf{w})\}\cup \{(\phi(\bsf{w}),h,x) : x\in h\backslash \pi(\bsf{w}) \}
\]
is an edge of $T_{\mathrm{SAW}, v}(G)$ and its label is $h\notin F$ and so it is in fact an edge of $T_{\mathrm{SAW}, v}(G)-\tilde F$. Removing the elements of this edge whose label belongs to $U$ we obtain
\[
\phi(\Gamma)= \{\phi(\bsf{w})\}\cup \{(\phi(\bsf{w}),h,x) : x\in h\backslash (U \cup \pi(\bsf{w}) \}
\]
and so $\phi(\Gamma)$ is an edge of $(T_{\mathrm{SAW}, v}(G)-\tilde F) \ominus \tilde U$. We conclude that $\phi$ is an isomorphism from $T_{\mathrm{SAW}, v_j^i}(G_j^i)$ to a connected subgraph of $(T_{\mathrm{SAW}, v}(G)-\tilde F) \ominus \tilde U$ containing $\bsf{r}_j^i$. We now show that this subgraph is in fact the entire component containing $\bsf{r}_j^i$. For this it suffices to show that any edge of this component is the image of an edge of $T_{\mathrm{SAW}, v_j^i}(G_j^i)$ under $\phi$.

Note that any edge in the component of $\bsf{r}_j^i$ in $(T_{\mathrm{SAW}, v}(G)-\tilde F) \ominus \tilde U$ is of the form 
\[
\Lambda=\{\bsf{w}\}\cup \{(\bsf{w},h,x) : x\in h\backslash (U \cup \pi(\bsf{w})) \}
\]
where $\bsf{w} =(v,e_1,v_1,h_1,\ldots, v_{\ell+1})$ is a node of $T_{\mathrm{SAW}, v}(G)$, $v_1=v_j^i$ and $h\notin F$. We will show that each $v_i\notin U$ and each $h_i\notin F$. Indeed, let $P$ denote the path from $\bsf{r}_j^i$ to $\bsf{w}$ in $T_{\mathrm{SAW}, v}(G)$. We note that if $X=(v,e_i,v_1,h_1,\ldots,v_i)$ for some $i$ then $X$ is a node of $P$ and $P\ominus \{X\}$ is disconnected. We conclude that $X\notin \tilde U$ else $\bsf{w}$ is not in the component of $\bsf{r}_j^i$ in $(T_{\mathrm{SAW}, v}(G)-\tilde F) \ominus \tilde U$. Thus $v_i \notin U$ for all $i$ as claimed. Note also that $P$ consists of edges $f_1,\ldots, f_\ell$ (say) where the label of $f_i$ is $h_i$. Moreover $P-\{f_i\}$ is disconnected for any $i$. It follows that each $h_i\notin F$ else $\bsf{w}$ is not in the component of $r_j^i$ in $(T_{\mathrm{SAW}, v}(G)-\tilde F) \ominus \tilde U$. It follows that
\[
\phi^{-1}(\bsf{w})=(v_1,h_1\backslash U,\ldots, h_\ell\backslash U,v_{\ell+1})
\]
is a node of $T_{\mathrm{SAW}, v_j^i}(G_j^i)$ and 
\[
\phi^{-1}(\Lambda)=
\{\phi^{-1}(\bsf{w})\}\cup \{(\phi^{-1}(\bsf{w}),h\backslash U,x) : x\in h\backslash( U\cup \pi(\bsf{w})) \}
\]
is an edge of $T_{\mathrm{SAW}, v_j^i}(G_j^i)$. We conclude that $T_{\mathrm{SAW}, v_j^i}(G_j^i)$ is isomorphic to the component containing $\bsf{r}_j^i$ in $(T_{\mathrm{SAW}, v}(G)-\tilde F)\ominus \tilde U$. 

Finally we note that given $T_{\mathrm{SAW}, v}(G)$, to obtain the subtree of $T_v(G)$ rooted at $\bsf{r}_j^i$ we apply operations~\eqref{op1} and~\eqref{op2} with $\bsf{w}=\bsf{r}_j^i$ and obtain the subtree of $(T_{\mathrm{SAW}, v}(G)-\tilde F) \ominus \tilde U$ rooted at $\bsf{r}_j^i$. We then apply operations ~\eqref{op1} and~\eqref{op2} at every other node of this subtree. If $G_j^i$ inherits the ordering of its vertices and edges from $G$ then these operations are identical to the operations performed to obtain $T_{v_j^i}(G_j^i)$ from $T_{\mathrm{SAW}, v_j^i}(G_j^i)$. This completes the proof.  
\end{proof}

\end{proof}

\subsection{Weitz tree of asymptotically tree-like hypergraphs}
Recall that for a hypergraph $G=(V,E)$ and $v\in V$ we write $E_G(v)$ for the set of $e\in E$ that contain $v$.

\begin{lemma}
\label{lemWeitzTreeStructure}
Fix $k\geq 2$ and let $G$ be an asymptotically tree-like $k$-uniform hypergraph with maximum degree $\Delta$. Let $U\subseteq V(G)$ be such that $v\notin U$, $|U|=O(1)$. 
Then the Weitz hypertree $T=T_v(G\ominus U)$ has the following properties: for every fixed $L>0$ and any node $\bsf{w}$ at depth $L$ with $\pi(\bsf w)=w$,
\begin{enumerate}
    \item $|\pi(E_T(\bsf w)) \triangle E_{G}(w)|=o(\Delta)$. \label{tree0}
    \item $\bsf{w}$ has degree $d_G(w) - o(\Delta)$.\label{tree1}
    \item for each $\ell \in \{2, \dots, k-1\}$, $\bsf{w}$ is in  $o( \Delta^{\frac{\ell-1}{k-1}}  )$ edges of size $\ell$.\label{tree2}
     \item $\bsf{w}$ has $o(\Delta)$ neighbors contained in an edge of size $1$.\label{tree3}
\end{enumerate}
\end{lemma}
\begin{proof}
Let $\bsf{w}=(v_1,e_1,\ldots,e_L, v_{L+1})$ be a node of $T'=T_{\mathrm{SAW},v}(G\ominus U)$ and let $w=v_{L+1}$ denote its label. Any non-parent edge $\bsf f$ of $T'$ containing $\bsf{w}$ has the form
\[
\bsf f=\{\bsf{w}\}\cup \{(\bsf{w},e_{L+1},x):x\in e_{L+1}\backslash w\}
\]
where $w\in e_{L+1}\in E(G\ominus U)$ and for all $i\leq L$ 
\begin{enumerate}[label=(\roman*)]
\item $e_{L+1}\neq e_i$ ,
\item $v_i\notin e_{L+1}$.
\end{enumerate}
Note that $\pi(\bsf f)=e_{L+1}$. Given $v_i$, $i\leq L$, there are at most $d_{G\ominus U}(v_i,w)\leq \Delta_2(G)=o(\Delta)$ edges incident to $w$ in $G\ominus U$ that contain $v_i$. We conclude that there are at least 
  \begin{align}\label{eq:localpi0}
 d_{G\ominus U}(w)-L-L\cdot o(\Delta) = d_G(w)-o(\Delta) 
 \end{align}
 choices for $e_{L+1}$ that satisfy (i) and (ii) where we used that $ d_{G\ominus U}(w)=d_G(w)$. On the other hand, there are clearly at most $d_{G\ominus U}(w)$ choices for $e_{L+1}$ and so we conclude that
 \begin{align}\label{eq:localpi}
 |\pi(E_{T'}(\bsf w)) \triangle E_{G\ominus U}(w)|=o(\Delta)
 \end{align}
and the degree of $\bsf{w}$ in $T'$ is $d_G(w)-o(\Delta)$.

Next we observe that for each $u\in U$ at most $d_G(u,w)\leq \Delta_2(G)$ edges in $E_G(w)$ contain $u$. We conclude that 
 \begin{align}\label{eq:localpi2}
|E_G(w)\triangle E_{G\ominus U}(w)|\leq 2|U| \Delta_2(G)=o(\Delta)\, .
\end{align}

Now we bound the number of edges in $E_{T'}(\bsf w)$ that shrink by an application of step~\eqref{op1} or that get deleted by an application of step~\eqref{op2} in the construction of $T$ from $T'$. If an edge in $E_{T'}(\bsf w)$ shrinks, then there is a vertex $u\in e_i$ where $i\leq L$ and a node with label $u$ has been set to occupied. There are at most $k L$ choices for $u$ and for each such $u$, at most $d_{G\ominus U}(u,w)\leq \Delta_2(G)$ edges in $E_{T'}(\bsf w)$ shrink when vertices with label $u$ are set to occupied. We conclude that at most $kL\Delta_2(G \ominus U)\leq kL\Delta_2(G) =o(\Delta)$ edges in $E_{T'}(\bsf w)$ shrink by an application of step~\eqref{op1}.

If an edge $E_{T'}(\bsf w)$ is removed by step~\eqref{op2}, then the edge was removed due to an ancestor $\bsf{w}'$ with label $w'$ (say) and the removed edge has a label $f\in E(G\ominus U)$ where $w,w'\in f$.  The total number of edges incident to $\bsf{w}$ removed in this way is therefore at most $L \Delta_2(G\ominus U)= o(\Delta)$ since $\bsf{w}$ has at most $L$ ancestors and the label of each ancestor shares at most $\Delta_2(G\ominus U)$ edges with $w$ in $G\ominus U$. Taken together with~\eqref{eq:localpi0},~\eqref{eq:localpi} and~\eqref{eq:localpi2}, this establishes claims~\eqref{tree0} and \eqref{tree1}.

We prove~\eqref{tree3} next. Suppose a non-parent edge $\bsf{f}$ incident to $\bsf{w}$ contains a node $\bsf{w}'$ with label $w'$ that is contained in an edge of size $1$. Let $\bsf{f}'$ denote the edge of $T_{\mathrm{SAW},v}(G\ominus U)$ that this edge of size $1$ came from and let $f'= \{w',w_1, \ldots, w_{r-1}\}\in E(G\ominus U)$ denote its label. In particular, the $r-1$ labels $\{w_1, \ldots, w_{r-1}\}$ were set to occupied by applications of step~\eqref{op1}. We note that each $w_i$ must be either (i) contained in $e_i$ for some $i\leq L$ or (ii) contained in $f\backslash \{w,w'\}$ where $f$ is the label of $\bsf{f}$. Suppose $w_1, \ldots, w_{t}$ are of the first type and $w_{t+1},\ldots ,w_{r-1}$ are of the second. In this event we say that $w_1, \ldots, w_{t}$ \emph{endangered} the edge $\bsf{f}$. 

We now bound the number of edges incident to $\bsf w$ a given $w_1, \ldots, w_{t}$ can endanger. 
If $w_1, \ldots, w_{t}$ endangers an edge with label $f\in E(G\ominus U)$ then there must exist $f'\in E(G\ominus U)$ of the form $f'=\{w',w_1,\ldots, w_{r-1}\}$ where $r-1\geq t$ and $f'\backslash\{w_1, \ldots, w_{t}\}\subseteq f\backslash \{w\}$. Moreover, by the definition $G\ominus U$, there exists $U'\subseteq U$ such that $h=f'\cup U'\in E(G)$. Observe also that $h\supseteq \{w_1,\ldots, w_t\}\cup U'$ and $|U'|=k-r$. Note that $t+k-r\geq 1$ else $|f'|=|h|=k$ and $f'\subseteq f\backslash\{w\}$, a contradiction since $|f\backslash\{w\}|\leq k-1$.

Suppose first that $t+k-r\geq 2$. Then, given $U'$ (for which there are $\leq 2^{|U|}$ choices), there are at most $\Delta_{t+k-r}(G)$ choices for $h$ and therefore $f'$. Given $f'$ there are at most $\Delta_{r-t+1}(G)$ choices for $f$ (indeed $f$ must contain $w$ and the $r-t$ vertices of $f'\backslash \{w_1, \ldots, w_{t}\}$). This gives at most $2^{|U|}\Delta_{t+k-r}(G)\Delta_{r-t+1}(G)=o(\Delta)$ (here we use $t\geq 2$) total choices for $f$. If instead  $t+k-r=1$, then $|f'\backslash \{w_1,\ldots, w_t\}|=r-t=k-1$ so that the condition $f'\backslash\{w_1, \ldots, w_{t}\}\subseteq f\backslash \{w\}$ implies that $|f|=k$ i.e.\ $f\in E(G)$. Either $t=1, U'=\emptyset$ in which case $h\backslash\{w_1\}=f'\backslash\{w_1\}= f\backslash \{w\}$ or $t=0, U'=\{u\}$ in which case $h\backslash\{u\}= f'=f\backslash \{w\}$. There are therefore at most $(|U|+1)\Gamma(G)=o(\Delta)$ choices for $f$.

Finally we note that since there are at most $L k$ vertices contained in the edges $e_1,\ldots, e_L$, there are at most $2^{Lk}$ possible choices for the set $\{w_1, \ldots, w_{t}\}$.  We conclude that there are $o(\Delta)$ edges incident to $\bsf{w}$ that are endangered and therefore $o(\Delta)$ neighbors of $\bsf{w}$ contained in an edge of size $1$. This concludes the proof of~\eqref{tree3}. 

For part \eqref{tree2} suppose that $\bsf{w}$ is contained in an edge $\bsf{f}$ of size $\ell$ and let $\bsf{f}'$ denote the edge in $T_{\mathrm{SAW},v}(G\ominus U)$ that $\bsf{f}$ came from. Then the label of $\bsf{f}'$ must have the form $f'=\{w,w_1,\ldots, w_{r-1}\}\in E(G\ominus U)$ where each $w_j$ with $j\leq r-\ell$ is contained in $e_i$ for some $i\leq L$. Moreover, there exists $U'\subseteq U$, $|U'|=k-r$, such that $f'\cup U'\in E(G)$. As before there are at most $2^{Lk+|U|}$ choices for $\{w_1,\ldots, w_{r-\ell}\}\cup U'$ and given such a set there are at most $\Delta_{k-\ell+1}(G)=o(\Delta^{\frac{\ell-1}{k-1}})$ choices for the label $f'$. The result follows. 
\end{proof}

\section{Contraction, Uniqueness, and Fixed Points}
\label{secNonRegular}

Let $G=(V,E)$ be a (simple) hypergraph with maximum degree $\Delta$. Recall~\eqref{eq:BPF} that for $c>0, \zeta\in[0,1]$ we define the function $F_{c,\zeta}^{G}: \mathbb{R}^{V}\to \mathbb{R}^V$ given by 
\begin{equation}
   \left(F_{c,\zeta}^{G} \left(\mathbf{x}\right)\right)_v = c \cdot \exp\left(- \frac{\zeta}{\Delta} \sum_{e \ni v}\prod_{u\in e \backslash\{v\}} x_u\right) \, ,
\end{equation}
for $v\in V$. The main goal of this section is to prove the following lemma which expresses vertex marginals $\mu_{G,\lam,\zeta}(v \in \mathbf S)$ in the regime $\lam  =(1+o(1)) c \cdot\Delta^{- \frac{1}{k-1}}$ in terms of a fixed point of $F_{c,\zeta}^{G}$. In fact, we provide an expression for the vertex marginals under $\mu_{G \ominus U,\lam,\zeta}$ where $U\subseteq V$ has constant size. We recall from Observation~\ref{obs:identities} that the measure $\mu_{G \ominus U,\lam,\zeta}$ is equivalent to the measure $\mu_{G,\lam,\zeta}$ conditioned on the event that the vertices in $U$ are occupied. 

Recall that we say $k, c, \zeta_n$ satisfy Condition~\ref{CondZeta} if $k \ge 2, c>0$ are fixed and $\zeta_n \in (0,1]$ satisfies 
\(
\limsup_{n \to \infty} \zeta_n (k-1) c^{k-1} <e. 
    \)

\begin{lemma}
\label{lemWeitzTreeFP}
Let $k, c, \zeta=\zeta_n$ satisfy Condition~\ref{CondZeta}. Let $G =G_n$ be a sequence of asymptotically tree-like $k$-uniform hypergraphs with maximum degree $\Delta$. Let $v\in V(G)$, $U\subseteq V(G)$ be such that $v\notin U$, $|U|=O(1)$. If  $\lam > 0$ and $\mathbf{x}^\ast=(x^\ast_v:v\in V(G)) \in \R^{V(G)}$ satisfy
 \begin{enumerate}
 \item $\lam  =(1+o(1)) c \cdot\Delta^{- \frac{1}{k-1}} $, and
 \item $\mathbf{x}^\ast=F_{c,\zeta}^G\left(\mathbf{x}^\ast\right)$,
 \end{enumerate} 
 then
    \begin{equation}
        \mu_{G \ominus U,\lam,\zeta}(v \in \mathbf S) = (1+o(1))\cdot(1-\zeta)^{m_v} \cdot x_v^\ast \cdot \Delta^{-\frac{1}{k-1}} \, ,
    \end{equation}
    where $m_v$ denotes the multiplicity of the edge $\{v\}$ in the multihypergraph $G\ominus U$.
\end{lemma}

Before turning to the proof, we need some more notation and some preliminary lemmas. Given a vector $\mathbf{x}\in \R_{>0}^V$, we let $\log\left(\mathbf{x}\right)=(\log(x_v):v\in V)$ and define $\exp(\mathbf x)$ similarly.

\begin{lemma}
\label{lemfkContractionNonReg}
    Fix $k \ge 2$, $c > 0$, and $\zeta \in (0,1]$ satisfying $\zeta(k-1)c^{k-1} < e$. Let $G=(V,E)$ be a $k$-uniform hypergraph of maximum degree at most $\Delta$. Then for every $\mathbf{x},\mathbf{y} \in \mathbb R_{>0}^V$,
    \begin{equation}
     \left\| \log(F^2 (\mathbf{x})) - \log(F^2(\mathbf{y}))\right\|_\infty    \le (1- \del) \left\| \log(\mathbf{x}) - \log(\mathbf{y})\right\|_\infty 
    \end{equation}
    where $F=F_{c,\zeta}^{G}$ and $\del=1-\zeta c^{k-1}(k-1)e^{-1}$. In particular, $F$ has a unique fixed point.
\end{lemma}
\begin{proof}
Fix $v\in V$, and let $f(\mathbf{x})=(F^2(\mathbf{x}))_v$. Our goal is to show that
\begin{equation}\label{eq:psigoal}
    | \log(f (\mathbf{x})) - \log(f(\mathbf{y}))|    \le (1- \del) \left\| \log(\mathbf{x}) - \log(\mathbf{y})\right\|_\infty 
    \end{equation}
    for every $\mathbf{x},\mathbf{y} \in \R_{>0}^V$.
We let $g(\mathbf x)=\log(f(\exp(\mathbf x)))$ so that it suffices to show 
\begin{align}\label{eq:gGoal}
|g(\mathbf{x})-g(\mathbf{y})|\leq (1-\delta)\|\mathbf{x}-\mathbf{y}\|_\infty
\end{align}
for all $\mathbf{x},\mathbf{y} \in \R ^V$. 
By the fundamental theorem of calculus
\[
g(\mathbf{x})-g(\mathbf{y})=\int_{0}^1\nabla g(\mathbf y+\theta(\mathbf x-\mathbf y))\cdot(\mathbf x- \mathbf y)\, d\theta
\]
and so to prove~\eqref{eq:gGoal} it suffices to show that
\begin{align}\label{eq:inftytoinfty}
\|\nabla g(\mathbf z)\|_1=\sum_{w\in V} 
 \left| \frac{\partial g}{\partial z_w}(\mathbf{z}) \right| \le 1-\delta
\end{align}
for all $\mathbf{z}\in \R^{V}$.  

Now fix $\mathbf{z}\in \R^{V}$ and let $\mathbf x=\exp(\mathbf z)$. By the chain rule,
\[
\frac{\partial g}{\partial z_w}(\mathbf z)
    = \frac{x_w}{f(\mathbf x)}\,
      \frac{\partial f}{\partial x_w}(\mathbf x)\, .
\]
It is therefore enough to show that 
\begin{align}
\sum_{w\in V} \frac{x_w}{f(\mathbf x)}
      \left|\frac{\partial f}{\partial x_w}(\mathbf x)\right | \leq 1-\delta\, .
\end{align}

For $S\subseteq V$, let $ x_S=\prod_{u\in S}x_u$. Let $F_u=(F(\mathbf x))_u$ for $u\in V$ and let $F_S=\prod_{u\in S}F_u$.
 Note that
 \[
 f(\mathbf x)= c \exp\left(-\frac{\zeta}{\Delta} \sum_{e\in E(v)} F_{e\backslash\{v\}}\right)\, 
 \]
and so for $w\in V$,
 \[
 \frac{1}{f(\mathbf x)}
      \frac{\partial f}{\partial x_w}(\mathbf x) = -\frac{\zeta}{\Delta} \sum_{e\in E(v)} \frac{\partial}{\partial x_w} F_{e\backslash\{v\}}\, .
 \]
 Next observe that
 \[
 \frac{\partial}{\partial x_w} F_{e\backslash\{v\}} = F_{e\backslash\{v\}}\sum_{u\in e\backslash\{v\}} \frac{1}{F_u} \frac{\partial F_u}{\partial x_w}\, ,
 \]
 and 
 \[
 \frac{1}{F_u} \frac{\partial F_u}{\partial x_w} = -\frac{\zeta}{\Delta} \sum_{g\in E(u): w\in g\backslash\{u\}} \frac{x_{g\backslash\{u\}}}{x_w}\, .
 \]
 Putting everything together we have 
 \[
  \frac{x_w}{f(\mathbf x)}
      \left|\frac{\partial f}{\partial x_w}(\mathbf x)\right |
      =
      \frac{\zeta^2}{\Delta^2} \sum_{e\in E(v)}
     F_{e\backslash\{v\}}
      \sum_{u\in e\backslash\{v\}} \sum_{g\in E(u): w\in g\backslash\{u\}} x_{g\backslash\{u\}}
 \]
 and so
 \[
 \sum_{w\in V} \frac{x_w}{f(\mathbf x)}
      \left|\frac{\partial f}{\partial x_w}(\mathbf x)\right |
      =
        \frac{\zeta^2}{\Delta^2}(k-1) \sum_{e\in E(v)}
     F_{e\backslash\{v\}}
      \sum_{u\in e\backslash\{v\}} \sum_{g\in E(u)} x_{g\backslash\{u\}}
 \]
 Next note that
 \[
 F_{e\backslash\{v\}}=c^{k-1}\exp\left(-\frac{\zeta}{\Delta}\sum_{u\in e\backslash\{v\}}\sum_{g\in E(u)} x_{g\backslash\{u\}} \right) \,,
 \]
 and so
 \[
  \sum_{w\in V} \frac{x_w}{f(\mathbf x)}
      \left|\frac{\partial f}{\partial x_w}(\mathbf x)\right |
      =
      \frac{\zeta}{\Delta}(k-1)\sum_{e\in E(v)}F_{e\backslash\{v\}} \log\left(c^{k-1}/F_{e\backslash\{v\}}\right)\leq \zeta c^{k-1}(k-1)e^{-1}=1-\delta\, ,
 \]
 where for the first inequality we used that $z \log(a/z)\leq a/e$ for $a\geq z>0$ and $|E(v)|\leq \Delta$.
\end{proof}

\begin{proof}[Proof of Lemma~\ref{lemWeitzTreeFP}]
Let $T=T_v(G \ominus U)$ denote  the Weitz hypertree of $G \ominus U$ at $v$ and let $\bsf{r}$ denote its root.

Let $T'$ denote the tree $T$ with any copy of the edge $\{\bsf r\}$ removed. For $\bsf v\in V(T)$, let 
\[
y(\bsf v)= R_{\bsf v}(T',\lam,\zeta)\cdot \Delta^{\frac{1}{k-1}}\, .
\]

Note that $R_{\bsf r}(T,\lam,\zeta)=R_v(G\ominus U,\lam,\zeta)=(1+o(1))\mu_{G\ominus U,\lam, \zeta} ( v \in \mathbf S)$ by Lemma~\ref{lemBencs} and $R_{\bsf r}(T,\lam,\zeta)=(1-\zeta)^{m_v}R_{\bsf r}(T',\lam,\zeta)$ and so our goal is to show that
\[
y(\bsf r)= \mathbf x^\ast_v +o(1)\, .
\]

For a node $\bsf w$ in $T'$, let $T_{\bsf w}$ denote the subtree of $T'$ rooted at $\bsf w$ and let $R(\bsf w)=R_\bsf w(T_\bsf w,\lam,\zeta)$. 

Now fix a node $\bsf w$ in $T'$ such that $\{\bsf w\}\notin E(T')$.
Suppose that $\{\bsf e_1,\dots,\bsf e_d\}$ is the set of incident edges at $\bsf w$ in $T_{\bsf w}$. 
  Using Lemma~\ref{lem:tree_recursion2} we have
\begin{align}\label{eq:tree-rec}
R(\bsf w)&=
\lambda  \prod_{i=1}^d\left(1-\zeta\prod_{\bsf u \in \bsf e_i\backslash\{\bsf w\}}\frac{R(\bsf u)}{1+R(\bsf u)}\right)\, .
\end{align}
We note that $R(\bsf u)=O\left(\Delta^{-\frac{1}{k-1}}\right)$ for all $\bsf u$. For $\ell\in\{2,\ldots,k\}$, let $I_\ell=\{i: |\bsf e_i|=\ell\}$ and recall that  $|I_\ell|=o\left(\Delta^{\frac{\ell-1}{k-1}}\right)$ for $\ell\leq k-1$ by Lemma~\ref{lemWeitzTreeStructure}. It follows that if $\ell\leq k-1$, then
\[
\prod_{i\in I_\ell}
\left(1-\zeta\prod_{\bsf u \in \bsf e_i\backslash\{\bsf w\}}\frac{R(\bsf u)}{1+R(\bsf u)}\right)=\left(1-O\left(\Delta^{-\frac{\ell-1}{k-1}}\right)\right)^{|I_\ell|}=1-o(1).
\]
Let $I_k'\subseteq I_k$ denote the set of $i$ such that $\bsf e_i\backslash\{ \bsf w\}$ contains no vertex contained in an edge of size $1$. Then $|I_k\backslash I_k'|=o(\Delta)$ by Lemma~\ref{lemWeitzTreeStructure} and so 
\[
\prod_{i\in I_k\backslash I_k'}
\left(1-\zeta\prod_{\bsf u \in \bsf e_i\backslash\{\bsf w\}}\frac{R(\bsf u)}{1+R(\bsf u)}\right)=\left(1-O\left(\Delta^{-1}\right)\right)^{|I_k\backslash I_k'|}=1-o(1).
\]

On the other hand, $|I'_k|=(1-o(1))d\leq \Delta$ and
\begin{align}
\prod_{i\in I'_k}
\left(1-\zeta\prod_{\bsf u \in \bsf e_i\backslash\{\bsf w\}}\frac{R(\bsf u)}{1+R(\bsf u)}\right)
&=
\prod_{i\in I'_k}(1+O(\Delta^{-2}))
\exp\left(-\zeta\prod_{\bsf u \in \bsf e_i\backslash\{\bsf w\}}R(\bsf u) \right)\\
&=
(1+o(1))
\exp\left(-\zeta\sum_{i\in I'_k}\prod_{\bsf u \in \bsf e_i\backslash\{\bsf w\}}R(\bsf u) \right)\, .
\end{align}
Returning to~\eqref{eq:tree-rec} we conclude that
\begin{align}\label{eq:yw}
y(\bsf w) = c\cdot\exp\left(-\frac{\zeta}{\Delta}\sum_{i\in I'_k}\prod_{\bsf u \in \bsf e_i\backslash\{\bsf w\}}y(\bsf u) \right)+o(1)\, .
\end{align}
Now consider the subtree $T''$ of $T'$ obtained by removing all edges of size $<k$, removing all edges that contain a node $\bsf v$ such that $\{\bsf v\}\in E(T')$, and then keeping only the connected component of $\bsf r$ in the resulting graph. Note that if $\bsf w\in V(T'')$, then $E_{T''}(\bsf w)=\{\bsf e_i: i\in I_k'\}\cup \{\bsf f\}$ where $\bsf f$ is the parent edge of $\bsf w$ in $T'$ (if $\bsf w \neq \bsf r$). Letting $\mathbf y \in [0,c]^{V(T'')}$ denote the vector given by $\mathbf y_{\bsf v}=y(\bsf v)$, and letting $F=F_{c,\zeta}^{T''}$, 
we conclude that 
\begin{align}\label{eq:yfixed}
\mathbf y_{\bsf w} = F(\mathbf y)_{\bsf w}+o(1)\, .
\end{align}
It follows by continuity of $F$ that
\[
\mathbf y_{\bsf w}= F^2(\mathbf y)_{\bsf w}+o(1).
\]
Now let $\mathbf x\in (0,c]^{V(T'')}$ be given by $\mathbf x_{\bsf v}=\mathbf x^\ast_{\pi(\bsf v)}$. By Lemma~\ref{lemWeitzTreeStructure} and the definition of $\mathbf x, \mathbf x^\ast$,
\[
 F(\mathbf x)_{\bsf w}=F_{c,\zeta}^{G}(\mathbf x^\ast)_{\pi(\bsf w)}+o(1)=\mathbf x^\ast_{\pi(\bsf w)}+o(1)= \mathbf x_{\bsf w} +o(1)\,,
\]
and so 
\[
\mathbf x_{\bsf w}=F^{2}(\mathbf x)_{\bsf w}+o(1).
\]
It follows that 
\begin{align}\label{eq:F2}
| \log(\mathbf y)_{\bsf w} - \log(\mathbf x)_{\bsf w} | = |\log(F^2(\mathbf y))_{\bsf w}-\log(F^2(\mathbf x))_{\bsf w}| + o(1)\, .
\end{align}

For $\ell \in \N$, define the following metric on $(0,c]^{V(T'')}$:
\[
\|\mathbf a - \mathbf b\|_{\ell}:=\sup \{|\log(\mathbf a_\bsf v)- \log(\mathbf b_\bsf v)|: d(\bsf r, \bsf v)\leq \ell \}\, .
\]
Note that if $\bsf w$ is at depth $\ell$, then $ F^2(\cdot)_{\bsf w}$ depends only on coordinates at depth $\leq \ell+2$. It follows from~\eqref{eq:F2} and Lemma~\ref{lemfkContractionNonReg} that
\begin{align}\label{eq:xyelldist}
\| \mathbf y - \mathbf x\|_\ell = \|  F^2(\mathbf y) - F^2(\mathbf x)\|_{\ell} +o(1) \leq (1-\delta)\| \mathbf y - \mathbf x\|_{\ell+2} +o(1)\, ,
\end{align}
where $\delta:=1-\limsup_{n\to\infty}\zeta_n(k-1)c^{k-1}e^{-1}$ and $\delta>0$ by assumption.
Note that since, $\mathbf y_{\bsf v}\in(0,c]$ for all $\bsf v$, then by~\eqref{eq:yfixed}, $\mathbf y_{\bsf v}\geq c\exp(-\zeta c^{k-1})+o(1)$ for all $v$ and similarly for $\mathbf x$. We conclude that $\|\mathbf y -\mathbf x\|_\ell \leq \zeta c^{k-1}+o(1)$ for all $\ell$.

Iterating the inequality~\eqref{eq:xyelldist} $L$ times where $L$ is the least integer for which $(1-\delta)^L\zeta c^{k-1}<\eps/2$ we conclude that
\[
\| \mathbf y - \mathbf x\|_\ell<(1-\delta)^L\| \mathbf y - \mathbf x\|_{\ell+2L}+o(1)<\eps\, ,
\] 
and so in particular 
\[
|\log (\mathbf y_\bsf r) - \log (\mathbf x_\bsf r)|
=
|\log (\mathbf y_{\bsf r}) - \log (\mathbf x^\ast_v)|
<\eps
\]
as desired. 
\end{proof}

\subsection{Facts about the fixed points} 
In this short section we collect some simple facts about the fixed point $\mathbf x^\ast $ guaranteed by Lemma~\ref{lemfkContractionNonReg}.

\begin{lemma}
\label{lem:diffx}
    Fix $k \ge 2$. Let $G=(V,E)$ be a $k$-uniform hypergraph of maximum degree at most $\Delta$. Let 
    \[
R=\{(c,\zeta): c>0, \zeta\in (0,1), \zeta(k-1)c^{k-1} < e\}\, .
    \]
   Let $\mathbf x^\ast:R\to \mathbb{R}^V$ where $\mathbf x^\ast(c,\zeta)$ is the unique fixed point of $F_{c,\zeta}^G$. Then $\mathbf x^\ast$ is analytic on $R$.
\end{lemma}
\begin{proof}
 Fix $(c,\zeta)\in R$. For $\mathbf u\in \R^V$ define,
\[
\Phi_{c,\zeta}(\mathbf u)= \log ((F_{c,\zeta}^G)^2(\exp \mathbf u))\, 
\]
and 
\[
Q(\mathbf u, c,\zeta)=\mathbf{u}- \Phi_{c,\zeta}(\mathbf u)\, .
\]
We note that $Q(\mathbf u, c,\zeta)$ is  analytic on $\R^V \times R$ and 
 that $\mathbf u^\ast(c,\zeta):=\log \mathbf{x}^\ast(c,\zeta)$ is the unique solution to $Q(\mathbf u, c,\zeta)=\mathbf 0$.

Let $D_{\mathbf u}Q(\mathbf u, c,\zeta)$ denote the Jacobian matrix of $Q$ with respect to the $\mathbf u$ coordinates i.e.
\[
D_{\mathbf u}Q(\mathbf u, c,\zeta)= \left(\frac{\partial}{\partial u_w}Q_{v}(\mathbf u, c,\zeta) \right)_{v,w\in V}\, .
\]
We have
\[
D_{\mathbf u}Q(\mathbf u, c,\zeta)= I-D_{\mathbf u}\Phi_{c,\zeta}(\mathbf u)\, .
\]
In the proof of Lemma~\ref{lemfkContractionNonReg} we showed that
\[
\|D_{\mathbf u}\Phi_{c,\zeta}(\mathbf u) \|_{\infty \to \infty}= \max_{v\in V}\sum_{w\in V}\left|\frac{\partial}{\partial u_w} \left(\Phi_{c,\zeta}(\mathbf u)\right)_v \right| \leq 1-\delta\text{ for all }\mathbf u\in \R^V\, 
\]
where $\delta=1-\zeta c^{k-1}(k-1)e^{-1}$
indeed this is a simple reformulation of inequality~\eqref{eq:inftytoinfty}.
It follows that, $D_{\mathbf u}Q(\mathbf u, c,\zeta)= I-D_{\mathbf u}\Phi_{c,\zeta}(\mathbf u) $ is invertible for all $\mathbf u\in \R^V$ and so in particular $D_{\mathbf u}Q(\mathbf u^\ast(c),c,\zeta)$ is invertible. We conclude by the analytic implicit function theorem (see e.g. \cite[Theorem~2.3]{glockner2006implicit}) that $\mathbf u^\ast(c,\zeta)$ is real analytic at $(c,\zeta)$ and so the same holds for $\mathbf x^\ast (c,\zeta)=\exp \mathbf u^\ast (c,\zeta)$.
\end{proof}

\begin{lemma}
    \label{lemRegularFixedPoint}
    Let $k, c, \zeta=\zeta_n$ satisfy Condition~\ref{CondZeta}. Let $G=G_n$ be a sequence of asymptotically tree-like, approximately regular, $k$-uniform hypergraphs with maximum degree $\Delta$.
    Let $\mathbf x^\ast$ denote the unique fixed point of $F_{c,\zeta}^G$. Then $\mathbf x^\ast $ satisfies 
    \[
    x^\ast_v = \left( \frac{W((k-1)c^{k-1}\zeta)}{(k-1)\zeta } \right)^{\frac{1}{k-1}}+o(1)\, \quad \text{for all } v\in V(G)
    \]
\end{lemma}
\begin{proof}
  Let $x_k^\ast(c,\zeta)=\left( \frac{W((k-1)c^{k-1}\zeta)}{(k-1)\zeta } \right)^{\frac{1}{k-1}}$ as  defined in~\eqref{eqxastdef}.
First observe that 
\begin{align}\label{eq:FPsimple}
x_k^\ast(c,\zeta) = c \cdot \exp \left( -\zeta  (x_k^\ast(c,\zeta) )^{k-1}  \right)\, .
\end{align} Let $\mathbf x\in \R^{V(G)}$ denote the constant vector with each coordinate equal to $x_k^\ast(c,\zeta)$. Since $G$ is approximately regular, each vertex has degree $(1+o(1))\Delta$ and so, letting $F=F_{c,\zeta}^G$,
\begin{align}
F(\mathbf x)_v =  c \cdot \exp\left(- \frac{\zeta}{\Delta} \sum_{e\ni v}\prod_{u\in e \backslash\{v\}} x_u\right) &= c\cdot \exp(-(1+o(1))\zeta (x_k^\ast(c,\zeta))^{k-1}) \\
&= x_k^\ast(c,\zeta)+o(1)= x_v+o(1)\, ,
\end{align}
where for the penultimate equality we used~\eqref{eq:FPsimple}. By continuity of $F$ it follows that $F^2(\mathbf x)_v=x_v+o(1)$. 
By Lemma~\ref{lemfkContractionNonReg} we then have
\begin{align}
\left\| \log(\mathbf{x}) - \log(\mathbf{x}^\ast)\right\|_\infty 
&=\left\| \log(F^2 (\mathbf{x})) - \log(F^2(\mathbf{x}^\ast))\right\|_\infty+o(1) \\   &\le (1- \del) \left\| \log(\mathbf{x}) - \log(\mathbf{x}^\ast)\right\|_\infty +o(1)
\end{align}
 where $\delta:=1-\limsup_{n\to\infty}\zeta_n(k-1)c^{k-1}e^{-1}$ and $\delta>0$ by assumption. Therefore 
$\left\| \log(\mathbf{x}) - \log(\mathbf{x}^\ast)\right\|_\infty=o(1)$ which implies that $x^\ast_v = x_k^\ast(c,\zeta) + o(1)$ for all $v\in V(G)$. 
\end{proof}

Before turning to the proof of Lemma~\ref{lemZetaNonRegular}, we record the following basic fact. We recall the definitions of $x^\ast_k(c,\zeta)$ and $c<\overline{c}_k(\eta)$ from~\eqref{eqxastdef} and~\eqref{eqOverlineCeta}.

\begin{lemma}\label{claim:zeta}
For $k\ge 2$, $\eta \in[0,1)$ and $c>0$, 
    the equation
    \begin{align}\label{eqZetaDefAgain}
    (1-\zeta)x^\ast_k(c,\zeta)^kc^{-k}=\eta
    \end{align}
    defines $\zeta$ uniquely. Moreover, if $c<\overline{c}_k(\eta)$, then 
    \begin{equation}
    \label{eqZetaCondition}
       (k-1) \zeta c^{k-1}  < e \,.
    \end{equation}
\end{lemma}
\begin{proof}
    For fixed $k\geq2, c>0$, the LHS of~\eqref{eqZetaDefAgain} is a continuous, strictly decreasing function of $\zeta$ mapping $[0,1]$ to $[0,1]$ and thus there is a unique solution $\zeta=\zeta(c,\eta,k)$.   

   We will show that $\zeta c^{k-1}$ is a strictly increasing function of $c$ (viewing $\zeta$ also as a function of $c$ via~\eqref{eqZetaDefAgain}). We then prove that $(i)$ if $\eta<\eta_k^\ast = e^{-k/(k-1)}$ and $c=\overline c(\eta)$, then $\zeta c^{k-1}=e/(k-1)$ and $(ii)$  $\ell:=\lim_{c\to\infty}\zeta c^{k-1}$ exists and if $\eta\geq\eta_k^\ast$ then $\ell\leq e/(k-1)$. The inequality~\eqref{eqZetaCondition} then follows. 
   
   To show that $\zeta c^{k-1}$ is increasing, we differentiate~\eqref{eqZetaDefAgain} with respect to $c$ and obtain
\[
\frac{\partial \zeta}{\partial c}= -
\frac{k(k-1)(1-\zeta)\zeta\,W\left( (k-1) c^{k-1} \zeta\right)}{c(k-1)\zeta + c(k - \zeta)\,W\left( (k-1) c^{k-1} \zeta\right)}
\]
and so 
\[
\frac{\partial}{\partial c}(\zeta c^{k-1})= 
c^{k-1} \frac{\partial \zeta}{\partial c} + (k-1)\zeta c^{k-2}=
\frac{c^{k-2} (k-1)^2 \zeta^2 \left(1 + W\left( (k-1) c^{k-1} \zeta\right)\right)}{(k-1) \zeta + (k - \zeta) W\left( (k-1) c^{k-1} \zeta\right)}
>0\, .
\]
Suppose now that $\eta<\eta^\ast_k$ and let $c=\overline c_k(\eta) = \left( \frac{e}{(k-1) (1- \eta/\eta_k^\ast)}\right)^{\frac{1}{k-1}}$. With this choice of $c$, \eqref{eqZetaDefAgain} yields $\zeta=1-\eta/\eta^\ast_k$ and so $\zeta c^{k-1}=e/(k-1)$.

Next assume $\eta \ge \eta_k^\ast$ and observe that as $x\to\infty$, $W(x)/x\to 0$. Since $\eta>0$, it follows from~\eqref{eqZetaDefAgain} that as $c\to\infty$ we must have that $\zeta c^{k-1}$ stays bounded (and so in particular $\zeta\to 0$). Since $\zeta c^{k-1}$ is increasing in $c$ we conclude that $\ell=\lim_{c\to\infty}\zeta c^{k-1}$ exists (and $\ell<\infty$).  Taking the limit $c\to \infty$ in~\eqref{eqZetaDefAgain} we conclude that 
\[
\left( \frac{ W((k-1) \ell)}{(k-1)\ell} \right)^{\frac{k}{k-1}}= \eta\, .
\]
Since $W(x)/x$ is decreasing in $x$, $W(e)=1$, and $\eta\geq \eta^\ast_k=e^{-k/(k-1)}$, we conclude that $\ell\leq e/(k-1)$ as desired. 
\end{proof}

We now prove Lemma~\ref{lemZetaNonRegular}.
\begin{proof}[Proof of Lemma~\ref{lemZetaNonRegular}]
Observe that $\mathbf x^\ast(c,0)$ is the constant vector where each coordinate is equal to $c$. Thus at $\zeta=0$ we have
\[
 (1 - \zeta)  \sum_{e \in E_n} \prod_{u \in e} x^\ast_u(c,\zeta)= c^k|E_n|> \eta c^k |E_n|\, . 
\]
 If $\zeta=1-\eta$, then $\zeta(k-1)c^{k-1}<e$, by our assumption on $c$ so that $\mathbf x^\ast=\mathbf x^\ast(c,\zeta)$ is well-defined. Moreover, when $\zeta=1-\eta$ we have 
\begin{align}\label{eq:IVT}
 (1 - \zeta)  \sum_{e \in E_n} \prod_{u \in e} x^\ast_u(c,\zeta)\leq \eta c^k |E_n|\, , 
\end{align}
where for the final inequality we used that $\mathbf x^\ast\in (0,c]^V$. It is clear that $\mathbf x^\ast(c,\zeta)$ is continuous as a function of $\zeta$ at $\zeta=0$ and it is continuous on $(0,1-\eta]$ by Lemma~\ref{lem:diffx}. Therefore by the intermediate value theorem there exists $\zeta_n\in (0,1-\eta]$ such that~\eqref{eq:IVT} holds with equality. Moreover, $\limsup_n \zeta_n(k-1)c^{k-1}\leq (1-\eta)(k-1)c^{k-1}<e$.

If $G_n$ is approximately regular then we instead choose $\zeta$ as in Lemma~\ref{claim:zeta} and note that this choice of $\zeta$ is independent of $n$.
We then have by Lemma~\ref{lemRegularFixedPoint} that
\[
 (1 - \zeta)  \sum_{e \in E_n} \prod_{u \in e} x^\ast_u(c,\zeta)=(1+o(1))(1-\zeta)x^\ast_k(c,\zeta)^k |E_n|\,  .
\] 
The result follows.
\end{proof}

\section{Marginals, variances, and partition function from fixed points}
\label{subsecVariances}
In this section we build on Lemma~\ref{lemWeitzTreeFP} to prove the following lemma which provides an asymptotic expression for the vertex and edge marginals under $\mu_{G,\lam,\zeta}$ and the partition function $Z_G(\lam,\zeta)$ in the regime $\lam  =(1+o(1)) c \cdot\Delta^{- \frac{1}{k-1}}$. The expressions are written in terms of the fixed point of $F_{c,\zeta}^{G}$ and the Belief Propagation operator $\cB^G_{c,\zeta}$. 
\begin{theorem}
    \label{thmMarginals}
    Let $k, c, \zeta=\zeta_n$ satisfy Condition~\ref{CondZeta}. 
 Let $G=G_n$ be a sequence of asymptotically tree-like $k$-uniform hypergraphs on $N$ vertices with maximum degree $\Delta$.  Suppose $\lam \sim c \Delta^{-\frac{1}{k-1}}$.  Then the following hold:
    \begin{enumerate}[label=(\roman*)]
        \item\label{item:marg1} For each $v \in V(G)$,
    \begin{equation}
    \label{eqTreeVertexMarginal}
        \mu_{G,\lam,\zeta}(v \in \mathbf{S}) \sim  x_v^\ast \cdot \Delta^{-\frac{1}{k-1}} \,,
    \end{equation}
    where $\mathbf{x}^\ast \in \R^{V(G)}$ is the unique solution to
    \begin{equation}
    \mathbf{x}^\ast = F_{c,\zeta}^G(\mathbf{x}^\ast).
    \end{equation}

    \item\label{item:marg2} For each $e \in E(G)$,
    \begin{equation}
     \label{eqTreeEdgeMarginal}
        \mu_{G,\lam,\zeta}(e \subseteq \mathbf{S}) \sim (1-\zeta)  \prod_{v\in e} \left(x_v^\ast\cdot\Delta^{-\frac{1}{k-1}}\right)\,.
    \end{equation}

\item\label{item:marg2.5} Letting $\mathbf x^\ast (t)$ denote the unique fixed point of $F_{t,\zeta}^G$, we have
\begin{align}
   \Delta^{\frac{1}{k-1}} \log Z_G(\lam,\zeta)\sim\sum_{v\in V}\int_{0}^{c} \frac{x^\ast_v(t)}{t} \, d t.
\end{align}
    \item\label{item:marg3} \begin{equation}
    \label{eqLogZtree}
     \Delta^{\frac{1}{k-1}} N^{-1}   \log Z_G(\lam,\zeta) =  N^{-1}\mathcal{B}^G_{c,\zeta}(\mathbf{x}^\ast) + o(1)\,.
    \end{equation}
    
    \item \label{item:marg4}
    \begin{align}
        \var_{G,\lam,\zeta} ( |\mathbf{S}|) &= o\left(N^2 \Delta^{-\frac{2}{k-1}}  \right) \quad \text{and} \label{eqn:varvtx}\\ 
        \var_{G,\lam,\zeta} ( |E(\mathbf{S})|) & = o\left(N^2 \Delta^{-\frac{2}{k-1}}\right) \label{eqn:varedge}  \,.
    \end{align}
    \end{enumerate}
\end{theorem}

We will build toward the proof of Theorem~\ref{thmMarginals} in a sequence of lemmas. We begin with the following basic observation.

\begin{lemma}\label{lem:basicobs}
Let $G=G_n$ be a sequence of asymptotically tree-like $k$-uniform hypergraphs on $N$ vertices with maximum degree $\Delta$. Then
\[
N\Delta^{-\frac{1}{k-1}} \to \infty \text{ as } n\to\infty\, .
\]
\end{lemma}
\begin{proof}
We count the edges of $G$ in two different ways.
Note that
\[
|E(G)|\leq \binom{N}{2}\Delta_2(G)=o(N^2\Delta^{\frac{k-2}{k-1}})\, .
\]
where for the last inequality we used the definition of an asymptotically tree-like sequence (Definition~\ref{def:tree}).
On the other hand $|E(G)|=\Omega(\Delta N)$, again by Definition~\ref{def:tree}. The result follows. 
\end{proof}

All of the remaining results of this section assume the setup of Theorem~\ref{thmMarginals} and we let $G=(V,E)$. For brevity we suppress $\lam,\zeta$ in our notation e.g.\ $\E_{G}=\E_{G,\lam,\zeta}$, $\var_{G}=\var_{G,\lam,\zeta}$ etc. Moreover, for $v\subseteq V$, we let 
\[
p_G(v)=\mu_{G,\lam,\zeta}(v\in \mathbf S)
\]
denote the marginal of $v$. We extend this notation in the natural way e.g. $p_G(v\mid u)=\mu_{G,\lam,\zeta}(v\in \mathbf S \mid u\in \mathbf S)$ for $u,v\in V$ and $p_G(U)=\mu_{G,\lam,\zeta}(U\subseteq  \mathbf S)$ for $U\subseteq V$.

\begin{lemma}\label{lem:vtxvariancebd}
 We have
 \[
 \var_{G} |\mathbf{S}| = o\left( \E_{G}[|\mathbf{S}|]^2 \right).
 \]
 \end{lemma}
 \begin{proof}
  First note that (e.g.\ by Lemma~\ref{lemWeitzTreeFP}) $\E_{G}|\mathbf{S}|=\Theta(N\Delta^{-\frac{1}{k-1}})$. 
 We write
\begin{align}\label{eq:varexp}
\var_{G} |\mathbf{S}|=\sum_{(u,v)\in V^2}p_G(u)\left[ p_{G}(v \mid u )  -  p_G(v)\right]\, .
\end{align}
If $u=v$, then the contribution to the sum on the RHS is at most $\E_{G}|\mathbf{S}|$ and $\E_{G}|\mathbf{S}|=o\left( \E_{G}[|\mathbf{S}|]^2 \right)$ by Lemma~\ref{lem:basicobs}.

For distinct $u,v\in V$ such that $\{u,v\}\notin E$, we have

\begin{align*}
 p_{G}(v  \mid u )  
    = p_{G \ominus u}(v )  
    = (1 + o(1)) x^\ast_v\cdot \Delta^{-\frac{1}{k-1}} 
   = (1+o(1)) p_{G}(v),
\end{align*}
where for the final two equalities we used Lemma~\ref{lemWeitzTreeFP}.  The contribution to the RHS of \eqref{eq:varexp} from such pairs $(u,v)$ is therefore also $o\left( \E_{G}[|\mathbf{S}|]^2 \right)$.

The number of pairs $(u,v)$ for which $\{u,v\}\in E$ is at most $N\Delta$ when $k=2$ and is $0$ when $k\geq 3$. We are therefore done if $k\geq3$, and if $k=2$, we bound the contribution to the RHS of~\eqref{eq:varexp} from such pairs by
\[
N\Delta \cdot O\left(\Delta^{-2} \right)= o\left( \E_{G}[|\mathbf{S}|]^2 \right) \,.\qedhere
\]
\end{proof}

Before we proceed to bound the variance of the number of occupied edges, we note the following estimate on the edge occupancy marginals of disjoint edges.

\begin{lemma}\label{lem:edgemarginal}
For $e\in E(G)$, we have 
\begin{enumerate}
\item $p_{G}(e) = (1 + o(1))\cdot (1-\zeta)\cdot \prod_{v\in e}p_{G}(v)$\label{item:pGe1}
\item If $f\in E(G)$ such that $e\cap f = \emptyset$ and $e\cup f$ induces no edge other than $e,f$ in $G$, then $p_{G}(e\cup f) = (1+o(1)) p_{G}(e) \cdot p_{G}(f)$.\label{item:pGe2}
\end{enumerate}
\end{lemma}
\begin{proof}\hfill
\begin{enumerate}
\item Let $e=\{v_1, \ldots, v_{k}\}$, and $U_i := \{v_1,\ldots,v_{i-1}\}$ for $i\in [k]$. We have 
\begin{align}
p_{G}(e) & =  \prod_{i\in [k]}p_{G}(v_i \mid v_1,\ldots, v_{i-1})\\
& = \prod_{i \in [k]}p_{G\ominus U_i}(v_i) \\
& = (1+o(1))\cdot(1-\zeta) \cdot \prod_{v \in e} \left(x^\ast_v \cdot \Delta^{-\frac{1}{k-1}}\right)  \\
& = (1+o(1))\cdot (1-\zeta) \cdot \prod_{v\in e}p_{G}(v)\label{eqn:edgemarginal}  
\end{align}
where for the final two equalities we used Lemma~\ref{lemWeitzTreeFP}.\\

\item Let $e\cup f=\{v_1, \ldots, v_{2k}\}$, and $V_i := \{v_1,\ldots,v_{i-1}\}$ for $i\in [2k]$. We have 
\begin{align*}
p_{G}(e\cup f ) & =  \prod_{i\in [2k]}p_{G}(v_i \mid v_1,\ldots, v_{i-1})\\
& = \prod_{i \in [2k]}p_{G\ominus V_i}(v_i) \\
& = (1+o(1))\cdot(1-\zeta)^{2} \cdot \prod_{v \in e\cup f} \left( x^\ast_v \cdot \Delta^{-\frac{1}{k-1}} \right) \\
& = (1+o(1))\cdot  p_{G}(e)\cdot p_{G}(f),  \\
\end{align*}
where for the final two equalities we used Lemma~\ref{lemWeitzTreeFP} noting that $e \cup f$ induces $2$ edges in $G$ by assumption.\qedhere
\end{enumerate}
\end{proof}

In particular, Lemma~\ref{lem:edgemarginal}  part \eqref{item:pGe1} gives us
\begin{equation}\label{eqn:expedges}
\E_{G}|E(\mathbf{S})| = \Theta\left(|E| \cdot \lam^{k}\right)=\Theta(N\Delta^{-\frac{1}{k-1}})\, ,
\end{equation}
where for the final equality we used that $\lam=\Theta(\Delta^{-\frac{1}{k-1}})$ and $|E|=\Theta(N\Delta)$ since $G$ is asymptotically tree-like.
The variance of the number of edges follows from Lemma~\ref{lem:edgemarginal} and a similar computation to that of Lemma~\ref{lem:vtxvariancebd} except that one needs to account for pairs of edges $e,f$ such that $e\cup f$ induce edges other than $e,f$. Towards this, we have the following.

\begin{lemma}\label{lem:badedgepairs}
Let
\[
Q := \left\{(e,f) \in E^2~|~\left(e\cap f = \emptyset\right)\land\left(E(e\cup f) \setminus \{e,f\} \neq \emptyset\right)\right\}
\]
i.e., the set of pairs of disjoint edges $(e,f)$ whose vertices induce edges other than $e$ and $f$. Then we have
\begin{equation}
|Q| = \begin{cases}
o\left(|E| \cdot \Delta^{\frac{k}{k-1}}\right)&\text{if }k\geq 3,\\
O(|E|\cdot \Delta^2)& \text{if }k = 2.
\end{cases}
\end{equation}
\end{lemma}
\begin{proof}
First consider the case $k \geq 3$. Consider a pair $(e,f) \in Q$ with $f' \in E(e\cup f)\setminus \{e,f\}$. Since $k \geq 3$,  $f'$ can be partitioned into $f' = (e\cap f') \sqcup (f\cap f')$ with $(e\cap f'), (f\cap f') \neq \emptyset$. We may choose such a pair $(e,f)$ as follows:
\begin{align*}
\text{choose }f'\in E&& |E|~\text{ways},\\
\text{partition into}~(f' \cap e),~(f'\cap f)&& \text{at most}~2^k~\text{ways},\\
\text{choose}~e\in N(e\cap f')&& o\left(\Delta^{\frac{k- |e\cap f'|}{k-1}}\right)~\text{ways if $|e\cap f'|>1$ and $\leq \Delta$ ways otherwise} ,\\
\text{choose}~f\in N(f\cap f')&&~o\left(\Delta^{\frac{k - |f\cap f'|}{k-1}}\right)~\text{ways if $|f\cap f'|>1$ and $\leq \Delta$ ways otherwise.}
\end{align*}
Thus noting that $|e\cap f'| + |f\cap f'| = k$, this gives $o\left(|E|\cdot \Delta^{\frac{k}{k-1}}\right)$ ways in total. If $k = 2$, then $|Q|$ is upper bounded by the number of walks of length $3$ in $G$, so $|Q| \leq 2|E|\cdot \Delta^2$.
\end{proof}

 \begin{lemma}\label{lem:variancebd}
 We have
 \[
 \var_{G} |E(\mathbf{S})| = o\left( \E_{G}\left[|E(\mathbf{S})|\right]^2 \right)
 \]
 \end{lemma}
 \begin{proof}

We bound the variance of $|E(\mathbf{S})|$ by bounding the covariance of $\mathbf 1[e\subseteq \mathbf{S}], \mathbf 1[f \subseteq \mathbf{S}]$ for all pairs $e, f\in E(G)$. Write 
\begin{align}\label{eq:varEdgeExp}
\var_{G} |E(\mathbf{S})| = \sum_{(e,f)\in E^2}p_G(e\cup f)-p_G(e)p_G(f)\, .
\end{align}
We bound the contribution to the sum on the RHS from three types of pair $(e,f)$. First suppose that $(e,f)\in Q$. Using Lemma~\ref{lem:badedgepairs}, and stochastic domination (Lemma~\ref{lemStochDom}), we bound the contribution from such pairs by
\begin{equation}\label{eqn:A'bound}
 |Q| \cdot \lam^{2k} = \begin{cases} o\left(|E| \cdot \Delta^{\frac{-k}{k-1}} \right) = o\left(\E_{G}{|E(\mathbf{S})|}\right)&\text{if}~k\geq 3,\\
O\left(|E|\cdot \Delta^{-2}\right) = O\left(\E_{G}|E(\mathbf{S})|\right)&\text{if}~k=2.
\end{cases}
\end{equation}
where for the final equalities we recalled~\eqref{eqn:expedges}.

Suppose now that $(e,f)$ is such that $e\cap f=\emptyset$ and $(e,f)\notin Q$ so that $e\cup f$ induces no edge other than $e,f$. Then by Lemma~\ref{lem:edgemarginal} part \eqref{item:pGe2}, the contribution from such pairs is at most 
\[
|E|^2\cdot o(\lam^{2k})= o\left( \E_{G}\left[|E(\mathbf{S})|\right]^2 \right)\, ,
\]
where again we used~\eqref{eqn:expedges}.

Finally suppose that $(e,f)$ is such that $e\cap f\neq \emptyset$. By stochastic domination (Lemma~\ref{lemStochDom}) the contribution from such pairs is at most
\begin{align} \label{eqn:Bbound}
 \sum_{\substack{(e,f)\in E^2\\e\cap f \neq \emptyset}} \lam^{|e\cup f|}
 \leq 
 \sum_{e\in E} \sum_{j=1}^{k-1}\sum_{\substack{S\subseteq e\\|S| = j}}\sum_{f\in N(S)} \lam^{2k - j} 
 &\leq
 |E|k\Delta\lam^{2k-1} +|E|\sum_{j=2}^{k-1}\binom{k}{j}\cdot o\left(\Delta^{\frac{k-j}{k-1}} \right) \lam^{2k-j}\\
 &\leq |E|k\lam^{k} + |E|\sum_{j=2}^{k-1} \binom{k}{j}\cdot o\left(\lam^k\right)\\
 &=O(\E_{G}|E(\mathbf{S})|)\, .
\end{align}
Putting everything together we conclude that 
 \[
 \var_{G} |E(\mathbf{S})| = O\left(\E_{G}[|E(\mathbf{S})|]\right) + o\left( \E_{G}\left[|E(\mathbf{S})|\right]^2 \right)= 
 o\left( \E_{G}\left[|E(\mathbf{S})|\right]^2 \right)\, ,
 \]
 where for the final equality we used~\eqref{eqn:expedges} and Lemma~\ref{lem:basicobs}.
 \end{proof}

 We now put things together and prove Theorem~\ref{thmMarginals}.
\begin{proof}[Proof of Theorem~\ref{thmMarginals}]\hfill
\begin{enumerate}[label=$(\roman*)$]
\item This follows immediately from Lemma~\ref{lemWeitzTreeFP} (with $U = \emptyset$) and Lemma~\ref{lemfkContractionNonReg}.
\item This follows from part~\ref{item:marg1} and Lemma~\ref{lem:edgemarginal} part \eqref{item:pGe1}.  
\item By part \ref{item:marg1} we have
\begin{align}
    \log Z_G(\lam,\zeta) & = \int_{0}^\lam  \frac{\E_{G,\theta,\zeta}|\mathbf{S}|  }{t}    \, dt =   \sum_{v \in V(G)}  \int_{0}^\lam  \frac{\mu_{G,\theta,\zeta}(v \in \mathbf{S})  }{\theta}    \, d\theta \\
    &= (1+o(1))\Delta^{-\frac{1}{k-1}}\sum_{v\in V(G)}\int_{0}^{c} \frac{x^\ast_v(t)}{t} \, d t.\label{eqn:integral}
\end{align}
where $\mathbf x^\ast (t)$ denotes the unique fixed point of $F_{t,\zeta}^G$ and for the last equality we used the substitution $\theta=t\Delta^{-\frac{1}{k-1}}$.
\item
Recall from~\eqref{eq:Bdef} that for $\mathbf{x}\in [0,c]^V$,
\[
    \mathcal B^G_{c,\zeta}(\mathbf x) =  - \frac{\zeta}{\Delta} \sum_{e \in E}  \prod_{u \in e} x_u - \sum_{v \in V} x_v \left[\log \frac{x_v}{c}  -1\right] \,.
\]
Let  $ \mathcal B_G(t)=\mathcal B^G_{t,\zeta}(\mathbf x^\ast(t))$ for $t\in [0,c]$. We aim to show that
\begin{equation}\label{eqn:derivative}
\frac{d}{dt}\cB_{G}(t) = \sum_{v\in V(G)}\frac{x_v^\ast(t)}{t} ,
\end{equation}
since then the result follows from part~\ref{item:marg2.5} and the fundamental theorem of calculus (noting that $\mathbf x^\ast(0)=\mathbf 0$ and so $\cB_{G}(0)=0.$). We note that $\mathbf x^\ast(t)$ is differentiable with respect to $t$ by Lemma~\ref{lem:diffx}. 

Now,
\begin{equation}\label{eqn:dB/dt}
\frac{d}{dt} \cB_{G}(t)  = -\frac{\zeta}{\Delta}\sum_{e\in E}\sum_{v\in e}\left(\frac{d}{dt}x_v^\ast\prod_{u\in e\backslash\{v\}}x_u^\ast\right) - \sum_{v\in V}\log\left(\frac{x_v^\ast}{t}\right)\frac{d}{d t}x_v^\ast + \sum_{v\in V}\frac{x_v^\ast}{t}.
\end{equation}
Since $\mathbf{x}^\ast(t)$ is a fixed point of $F_{t,\zeta}^G$, we have for every $v\in V$, that 
\[
\log \left(\frac{x_v^\ast}{t}\right) = -\frac{\zeta}{\Delta}\sum_{e\ni v}\prod_{u\in e\setminus \{v\}}x_u^\ast.
\]
Plugging this into~\eqref{eqn:dB/dt} establishes~\eqref{eqn:derivative}.\\

\item This follows from Lemmas~\ref{lem:vtxvariancebd} and~\ref{lem:variancebd}.\qedhere
\end{enumerate}
\end{proof}

\section{Lower tails via  partition functions}
\label{secLowerTailPartition}

\newcommand{\open}{{\mathcal{X}}}

In this section we prove Theorems~\ref{thmMain},~\ref{thmMainGnm} and \ref{thmNonRegularRate}. 

Let  $G=(V,E)$ be a $k$-uniform asymptotically tree-like hypergraph. Recall that for $p \in (0,1)$, we let $V_p$ denote a random subset of $V$ chosen by including each vertex independently with probability $p$ and we let $\mathbf X = |E(V_p)|$ be the number of edges induced by $V_p$. 

In this section we  relate the lower-tail probability $\P_p( \mathbf X \le \eta \E \mathbf X)$ to the log partition function $\log Z_G(\lam,\zeta)$ for $\lam=p/(1-p)$ and an appropriate choice of $\zeta$.   We then prove our results for lower tails in $G(n,m)$ by estimating the probability that $|S| =m$ and the lower-tail event is achieved under the measure $\mu_{G,\lam,\zeta}$. 

Recall the definitions of $\overline c_k(\eta), c_k(\eta)$ from~\eqref{eqOverlineCeta} and~\eqref{eq:ckdef}. Throughout this section we assume the following.

\begin{assumption}
\label{assumTreeEta}
    $G = G_n $ is a sequence of asymptotically tree-like $k$-uniform hypergraphs of maximum degree $\Delta$ on $N$ vertices; $\eta \in [0,1)$ and $0<c<c_k(\eta)$ are fixed. If in addition $G$ is approximately regular, we make the weaker assumption $0<c<\overline c_k(\eta)$.
\end{assumption}

The main result of this section is the following. 
\begin{lemma}
\label{lemZtoP}
     With $p \sim c \Delta ^{- \frac{1}{k-1}}$, $\lam = \frac{p}{1-p}$, and  $\mathbf x^\ast, \zeta$ satisfying the conclusion of Lemma~\ref{lemZetaNonRegular},
    \begin{equation}
          \log \P_p ( \mathbf X \le \eta \E \mathbf X)  \sim  \log Z_G(\lam, \zeta) -   \log (1-\zeta)\eta \E_p\mathbf X  -   pN \,.
    \end{equation}
    Moreover, suppose   $m \sim \Delta^{- \frac{1}{k-1}} \sum_{v \in V} x^\ast_v$.
    Then
    \begin{align}
        \Delta^{\frac{1}{k-1}} N^{-1}  \log \P_p ( |\mathbf S| = m \big | \mathbf X \le \eta \E \mathbf X)  = o(1) \,.
    \end{align}
\end{lemma}
The first statement says that with the right choice of $\zeta$ (that for which the lower-tail event is `achieved' in expectation by the BP fixed point), the lower-tail rate function is given by the log partition function, suitably shifted. The second statement says that conditioned on the lower-tail event, it is not too unlikely to have $|S|=m$, where $m$ is close to the sum of the marginals predicted by the BP fixed point.

Before we prove this lemma, we combine it with Theorem~\ref{thmMarginals} to prove Theorems~\ref{thmMain},~\ref{thmMainGnm}, and~\ref{thmNonRegularRate}.

\begin{proof}[Proof of Theorems~\ref{thmNonRegularRate} and~\ref{thmMain}]
    Theorem~\ref{thmNonRegularRate}  follows from the first part of Lemma~\ref{lemZtoP} by applying Theorem~\ref{thmMarginals} to give the value of  $\Delta^{\frac{1}{k-1}} N^{-1} \log Z_G(\lam, \zeta)$.  Similarly, Theorem~\ref{thmMain}  follows from these steps and the application of Lemma~\ref{lemRegularFixedPoint} to determine the value of $N^{-1}\mathcal{B}^G_{c,\zeta}(\mathbf{x}^\ast)$ in the approximately regular case. We also note in our application of Lemma~\ref{lemZtoP} that in the approximately regular case $\E_p\mathbf X=p^k |E(G)|\sim c^kN\Delta^{-\frac{1}{k-1}}k^{-1}$.
    Moreover we recall that in the proof of Lemma~\ref{lemZetaNonRegular}, we showed that if $G$ is approximately regular then we can take the $\zeta$ guaranteed by the lemma to be the unique solution in $[0,1]$ to the equation 
    $(1-\zeta) (x^\ast)^k= \eta c^k$. 
\end{proof}

\begin{proof}[Proof of Theorem~\ref{thmMainGnm}]
Let 
\[
L=\eta \E_m \mathbf X\sim \eta |E(G_n)|\left(\frac{m}{N} \right)^k\, .
\]
We begin with the identity
\begin{align}\label{eq:mId}
   \P_m ( \mathbf X \le L ) 
 =   \P_p ( |\mathbf S| = m \big | \mathbf X \le L)\cdot \frac{\P_p(\mathbf X \leq L)}{\P_p(|\mathbf S|=m)}\, ,
 \end{align}
which holds for all $p\in(0,1)$. We will apply this identity with $p=c\Delta^{-\frac{1}{k-1}}$ for appropriately chosen $c$. Note that since $m\sim b N \Delta^{-\frac{1}{k-1}}$,
\[
L= \eta' \cdot \E_p\mathbf X \quad \text{where}\quad \eta'\sim(b/c)^k \eta\, .
\]

In order to approximate the probabilities on the RHS of~\eqref{eq:mId} we will use Theorem~\ref{thmMain} and Lemma~\ref{lemZtoP}. Set $\zeta=1-\eta$ and choose $c$ such that 
\[
 x^\ast:=\left( \frac{W((k-1)c^{k-1}\zeta)}{(k-1)\zeta } \right)^{\frac{1}{k-1}} =b\, ,
  \]
i.e. $c=b \exp(\zeta b^{k-1})$. With this choice of $c,\zeta$ note that
\[
(1-\zeta)(x^\ast)^k\sim\eta' c^k
\]
We may therefore apply Lemma~\ref{lemZtoP} to conclude that
\[
 \Delta^{\frac{1}{k-1}} N^{-1}  \log \P_p ( |\mathbf S| = m \big | \mathbf X \le L)  = o(1) \, ,
\]
and Theorem~\ref{thmMain} to obtain
 \begin{align}
    \Delta^{\frac{1}{k-1}} N^{-1} \log \P_p ( \mathbf X \le L) 
   &= x^\ast + (x^\ast)^k  \left ( 1- \frac{1}{k} \right) \zeta  - \log (1-\zeta) \frac{\eta' c^k}{k}  - c +o(1) \\
   &= b + b^k  \left ( 1- \frac{1}{k} \right) (1-\eta)  - \log (\eta) \frac{\eta b^k}{k}  - c +o(1)\, . 
 \end{align}
 Finally we note that standard binomial estimates show that 
 \begin{align}
 \Delta^{\frac{1}{k-1}} N^{-1}\log \P_p(|\mathbf S|=m)
 &= \Delta^{\frac{1}{k-1}} N^{-1}\log\left(\binom{N}{m}p^{m}(1-p)^{N-m}\right)\\
 &= b \log(c/b)+b-c+o(1)\\
 &=(1-\eta)b^{k}+b-c+o(1)\, ,
 \end{align}
where for the final equality we recalled that $c=b \exp(\zeta b^{k-1})$, $\zeta=1-\eta$.

Returning to~\eqref{eq:mId} we conclude that 
\begin{align}
 \Delta^{\frac{1}{k-1}} N^{-1} \log   \P_m ( \mathbf X \le L )&=  b^k  \left ( 1- \frac{1}{k} \right) (1-\eta)  - \log (\eta) \frac{\eta b^k}{k}-(1-\eta)b^{k}+o(1)\\
 &=-\frac{b^k}{k}(1-\eta+\eta \log \eta)+o(1)\, .
\end{align}
as desired.
\end{proof}

\subsection{Point probability estimates}
\label{secPointProbs}

To prove Lemma~\ref{lemZtoP} we  will need some weak lower bounds on the probability that a sample from $\mu_{G,\lam,\zeta}$ has exactly $M$ vertices and at most $T$ edges, when these statements typically hold to first order.

\begin{lemma}
    \label{lemTransferGnm}
    Fix $c,\eps >0$, $\zeta \in[0, 1]$. Suppose $\lam \sim c \Delta^{-\frac{1}{k-1}}$ and that whp for $\mathbf S \sim \mu_{G,\lam,\zeta}$ we have
    \begin{align}
        |\mathbf S| &= (1+o(1)) M,\quad\text{and}\quad
        |E(\mathbf S)| = (1+o(1)) T,
    \end{align}
    where (i) $M = \Theta(N\Delta^{-\frac{1}{k-1}})$ and (ii) either $(\zeta,T) = (1,0)$, or $T = \Theta(N\Delta^{-\frac{1}{k-1}})$. Set $T_0 = T - \varepsilon N\Delta^{-\frac{1}{k-1}}$. Then
    \begin{equation}
        \P_{G,\lam,\zeta} \left(\left(|\mathbf S| = M\right) \wedge\left( T_0 \leq |E(\mathbf S)| \leq T\right) \right) = \exp\left( - o\left(N\Delta^{-\frac{1}{k-1}}\right)  \right) \,.
    \end{equation}
\end{lemma}

To prove Lemma \ref{lemTransferGnm} we first show that there are many subsets of $\mathbf S$ with the desired number of edges and approximately the desired number of vertices. We begin with a basic lemma bounding the number of pairs of incident edges in $E(\mathbf S)$.

For the remainder of this subsection, we assume the setup of Lemma~\ref{lemTransferGnm}.

\begin{claim}\label{clm:efintersect}
Let $\mathbf S\sim \mu_{G,\lam,\zeta}$ then whp the number of pairs $(e,f)\in E(\mathbf S)^2$ such that $e\cap f\neq \emptyset$ is at most
\[
\left(N \Delta^{-\frac{1}{k-1}}\right)^{3/2}\, .
\]
\end{claim}
\begin{proof}
Let $\mathbf Y$ denote the number of pairs $(e,f)\in E(\mathbf S)^2$ such that $e\cap f\neq \emptyset$. We recall from Lemma~\ref{lem:basicobs} that $N \Delta^{-\frac{1}{k-1}}=\omega(1)$. Therefore, by Markov's inequality it suffices to show that $\E_{G,\lam,\zeta} \mathbf Y =O\left(N \Delta^{-\frac{1}{k-1}}\right)$. Let us denote $|\{(e,f)~:~|e\cap f|=\ell \}| = m'_{\ell}$. Since $G=(V,E)$ is asymptotically tree-like we have 
\begin{equation}
|\{(e,f)~:~|e\cap f|=\ell \}| = \begin{cases} o\left(|E|\cdot\Delta^{ \frac{k-\ell}{k-1}}\right)&~\text{if}~\ell\geq 2\\ 
O(|E|\Delta)&~\text{if}~\ell=1.\end{cases}
\end{equation}
With $p=\frac{\lam}{1+\lam}$ we have by stochastic domination (Lemma~\ref{lemStochDom}) that
\[
\E_{G,\lam,\zeta} \mathbf Y \leq \E_p \mathbf Y \leq O\left(|E|\Delta\cdot p^{2k-1}+ |E|\sum_{\ell=2}^k \Delta^{\frac{k-\ell}{k-1}}\cdot p^{2k-\ell}\right) = O\left(N \Delta^{-\frac{1}{k-1}}\right)\, ,
\]
as desired.
\end{proof}

\begin{claim}\label{clmLoweringTriangleCount}
    Suppose that $S\subseteq V(G)$ and $r\in \R_{\geq 0}$ satisfy 
    \begin{enumerate}
    \item $\frac{|E(S)|}{T} = (1+o(1)) $ and $|E(S)| > T$, 
    \item $\frac{|S|}{M} = (1+o(1))$, 
    \item $\left|\left\{(e,f)\in E(S)^2~|~e\cap f\neq \emptyset \right\}\right| \leq \left(N \Delta^{-\frac{1}{k-1}}\right)^{3/2}$, and
    \item $r = 2\left\lfloor \frac{M}{|E(S)|}\cdot\left(|E(S)|-T+\left(N \Delta^{-\frac{1}{k-1}}\right)^{3/4}\right) \right\rfloor$.
    \end{enumerate}
    
     Then there are at least $(1-o(1))\binom{|S|}{r}$ sets $U \subseteq S$ satisfying 
     \begin{enumerate}
     \item $|U|=|S|-r$, and
     \item $(1-o(1))T \leq |E(U)| \leq T$.
     \end{enumerate}
\end{claim}

\begin{proof}
    Let $\mathbf U \subseteq S$ be a uniformly random subset of $S$ of size $|S|-r$. Equivalently, $\mathbf U$ is obtained by deleting a uniformly random set of $r$ vertices from $S$. The parameter $r$ is chosen so that whp slightly more than $|E(S)|-T$ edges are deleted in this process. We prove this with a second-moment calculation. Let $\mathbf Y = |E(S)| - |E(\mathbf U)|$ be the number of deleted edges. We have
    \[
    \E \mathbf Y = (1+o(1)) \cdot |E(S)| \cdot \frac{kr}{|S|} = (1+o(1))\cdot 2k\cdot \left(|E(S)|-T+ \left(N \Delta^{-\frac{1}{k-1}}\right)^{3/4}\right).
    \]
    For $e\in E(S)$ let $\mathbf X_e$ denote the indicator of the event that the edge $e$ is deleted. Note that for any pair $e,f\in E(S)$ we have $\cov(\mathbf X_e, \mathbf X_f)\leq \P(\mathbf X_e=1)=(1+o(1))kr/|S|$. Moreover, if $e\cap f=\emptyset$ then $\cov(\mathbf X_e, \mathbf X_f)=o((r/|S|)^2)$.
    
    Since there are at most $(N\Delta^{-\frac{1}{k-1}})^{3/2}$ pairs of edges in $S$ that are incident to each other
    \[
    \var (\mathbf Y) \leq (1+o(1))\cdot \frac{kr}{|S|}\cdot \left(N\Delta^{-\frac{1}{k-1}}\right)^{3/2} +o\left(|E(S)|^2 \frac{r^2}{|S|^2} \right)\, .
    \]
    Thus 
\[
\frac{\var( \mathbf Y)}{(\E \mathbf Y)^2} \leq (1+o(1))\cdot  \frac{|S|}{kr} \cdot \frac{\left(N\Delta^{-\frac{1}{k-1}}\right)^{3/2}}{|E(S)|^2}
    +
    o(1)
    =o(1)\, ,
\]
where we used that $|E(S)|=\Omega\left(N \Delta^{-\frac{1}{k-1}}\right)$, $|S|=O\left(N \Delta^{-\frac{1}{k-1}}\right)$ and $r\geq \left(N\Delta^{-\frac{1}{k-1}}\right)^{3/4}$. Thus Chebyshev's inequality gives us that, whp
    \[
   |E(S)|-T\leq \mathbf Y\leq  4k\left(|E(S)|-T+\left(N\Delta^{-\frac{1}{k-1}}\right)^{3/4}\right),
    \]
    and so $(1-o(1))T \leq |E(\mathbf U)| \leq T$. As a consequence, we have for a uniformly random $U \in \binom{S}{r}$, that
    \[
    \P\left((1-o(1))\cdot T  \leq |E(\mathbf U)| \leq T\right) = 1 - o(1)
    \]
as desired. 
\end{proof}

\begin{claim}\label{clmDownwardsConstruction}
      Suppose that $S\subseteq V(G)$ with $|E(S)| = (1+o(1))\cdot T$ and $\frac{M'}{M} = (1+o(1)) > 1$. There are at least $(1-o(1))\cdot \binom{M'}{M}$ sets $U \subseteq S$ of size $M$ that also satisfy
    $|E(U)| \geq (1-o(1)) \cdot |E(S)|$.
\end{claim}

\begin{proof}
    Let $\mathbf U$ be a uniformly random subset of $S$ of size $M$. Since there are $\binom{M'}{M}$ choices for $\mathbf U$ it suffices to show that whp $|E(S)|-|E(\mathbf U)| = o\left(N\Delta^{-\frac{1}{k-1}}\right)$.
    Indeed, the expected number of edges in $S$ that are not in $\mathbf U$ is at most $k \cdot |E(S)|  \cdot \frac{M'-M}{M'} = o\left(N\Delta^{-\frac{1}{k-1}}\right)$. Thus Markov's inequality gives us that whp $|E(S)|-|E(\mathbf U)| = o\left(N\Delta^{-\frac{1}{k-1}}\right)$, as claimed. 
\end{proof}

\begin{defn}
    For a set $S\subseteq V(G)$, we say that a vertex $v\in V(G)$ is \emph{open} (wrt to $S$) if $v\notin S$, and there is no edge $e\in E(G)$ such that $e\setminus\{v\}\subseteq S$. 
 We write
$\open(S)$ for the number of \textit{open} vertices.
\end{defn}

We begin by observing that with high probability, a constant fraction of vertices in a set sampled from $\mu_{G,\lambda,\zeta}$ are open.

\begin{claim}\label{clmManyOpenEdges}
    Fix $c>0$ and $\zeta \in [0,1]$, and suppose that $\lam \sim c \Delta^{-\frac{1}{k-1}}$ and $\alpha = \frac{1}{2}e^{-c^{k-1}}$. Then we have
    \[
\mathbb{P}_{G,\lambda,\zeta}\left(\open(\mathbf S) \geq \alpha N\right) = 1 - o(1).
    \]
\end{claim}
\begin{proof}
Set $p = \frac{\lam}{1 + \lam}$. By stochastic domination (Lemma~\ref{lemStochDom}), it suffices to show $\P_{p}(\open(\mathbf{S}) \geq \alpha N) = 1-o(1)$. For $v\in V$, let 
\[
E_v  := \{e\setminus \{v\}~:~e\ni v\}\, .
\]
Then by the Harris-FKG inequality we observe that
\begin{align}
\P_p( v~\text{open})&=\P_p(v \notin \mathbf S \wedge e\not \subseteq \mathbf S~\text{for all}~e\in E_v)\\
&\geq (1-p)(1-p^{k-1})^{\deg_G(v)}\\
&=(1+o(1))\exp(-p^{k-1}\deg_G(v))\label{eq:vopendeg}\\
&\geq(1+o(1))2\alpha\, .
\end{align}
In particular
\begin{equation}\label{eqn:openmean}
\E_{p}[\open(\mathbf S)] \geq (1+o(1))2\alpha\cdot N\, .
\end{equation}
Next, we upper-bound $\var_p(\open(\mathbf S))$. For $u,v\in V$, $u\neq v$, define 
\[
\mathbf{Y}_{u,v}=|\{e\in E_u\cup E_v~|~e\subseteq \mathbf S\}|\, ,
\]
so we have
\begin{align}\label{eq:Yuv0}
\P_p(u,v~\text{open}) \leq \P_p\left(\mathbf{Y}_{u,v}=0\right)\, .
\end{align}
We will upper bound the RHS using Janson's inequality.  Recalling that $\deg_{G}(\{u,v\}) = o(\Delta^{\frac{k-2}{k-1}})$, we obtain
\begin{equation}\label{eqn:openJmean}
\E_{p}[\mathbf{Y}_{u,v}] = p^{k-1}\left(\deg_G(u) + \deg_G(v)\right) + o\left(p^{k-1}\Delta^{\frac{k-2}{k-1}}\right).
\end{equation}
With Janson's inequality still in mind, we next wish to bound
\[
\Lambda:= \sum_{\substack{(e,f)\in (E_u \cup E_v)^2\\ e\neq f,\,  e\cap f\neq \emptyset}} \P(e \cup f \subseteq \mathbf S)\, .
\]
To this end, it will be useful to note that for $\ell\in\{1,\ldots, k-2\}$,
\begin{align*}
B_{\ell} := &\left|\{(e,f)\in (E_u \cup E_v)^2~| ~|e\cap f| = \ell\} \right| \\
\leq
& 
\sum_{e \in E_u\cup E_v}\left|\bigcup_{\substack{L\subseteq e,
|L| = \ell}}\left\{f \in E(G)~|~\left(f\supseteq L \cup \{u\}\right) \lor \left(f\supseteq L \cup \{v\}\right) \right\} \right| \\
\leq 
&
\sum_{e\in E_u\cup E_v}\sum_{\substack{L\subseteq e\\|L| = \ell}}\left(\deg_{G}(L\cup \{u\}) + \deg_G(L\cup \{v\})\right) \\
\leq 
&
2\sum_{e\in E_u\cup E_v}\sum_{\substack{L\subseteq e\\|L| = \ell}}\Delta_{\ell+1}(G) =  o\left(\Delta^{1+\frac{k-\ell-1}{k-1}}\right).
\end{align*}
It follows that 
\begin{align}\label{eqn:openJcum}
\Lambda\leq \sum_{\ell=1}^{k-2} B_\ell\cdot p^{2(k-1)-\ell}=o(1)\, .
\end{align}
Returning to~\eqref{eq:Yuv0} and applying Janson's inequality using~\eqref{eqn:openJmean} and~\eqref{eqn:openJcum}, we obtain
\begin{align*}
\P_p(u,v~\text{open}) & \leq \P_p(\mathbf Y_{u,v}=0) \\
& \leq \exp\left(-p^{k-1}\left(\deg_G(u) + \deg_G(v)\right) + o(1)\right) \\
& \leq (1 + o(1))\cdot \P_p(u~\text{open})\cdot \P_p(v~\text{open})\, ,
\end{align*}
where for the final inequality we used~\eqref{eq:vopendeg}.
It follows that $\var_{p}(\open(\mathbf S)) = o\left(\E_{p}[\open(\mathbf S)]^2\right)$ and our claim follows from Chebyshev's inequality and~\eqref{eqn:openmean}.
\end{proof}

Given $S\subseteq V(G)$, $v\in V(G)$ and $\ell\in\{0,\ldots,k\}$ let
\[
E_\ell(S)=\{e\in E(G) : |e\cap S|=\ell\}\, .
\]
and let 
$E_{\geq\ell}(S)=\bigcup_{j\geq \ell}E_{j}(S)$.

\begin{claim}\label{clmSdegrees}
    We have
    \[
    \P_{G,\lam,\zeta}\left(\left(|E_\ell(\mathbf S)|\leq k 2^{k+1} |E(G)| p^{\ell}\right)~\forall\ell \in \{0,\ldots, k-2\}\right) \geq 1/2.
    \]
\end{claim}
\begin{proof}
 By stochastic domination (Lemma~\ref{lemStochDom})
\[
\E_{G,\lam,\zeta}[|E_\ell(\mathbf S)|]\leq \E_{G,\lam,\zeta}[|E_{\geq\ell}(\mathbf S)|]\leq \E_{p}[|E_{\geq\ell}(\mathbf S)|]\leq |E(G)|2^k p^\ell\, ,
 \]
 where for the final inequality we used that for a given edge $e\in E(G)$, there are at most $2^k$ subsets of $e$ of size $\geq \ell$ and the probability (under $\mu_p$) that $\mathbf S$ contains such a subset is at most $p^\ell$.
 The result follows by Markov's inequality and a union bound over $\ell$.
\end{proof}

\begin{claim}\label{clmUpwardsConstruction}
    Fix $\alpha>0$. Suppose that $S\subseteq V(G)$ has size $M'$, where $1>\frac{M'}{M} = (1-o(1)) $. Suppose further that $\open(S) \geq \alpha N$ and $|E_\ell(S)|\leq k2^{k+1}|E(G)| p^\ell$ for all $\ell\in\{0,\ldots, k-2\}$. Then there are at least 
    \[
    \frac{(\alpha N)^{M-M'}}{(M-M')^{M-M'}} \cdot  \exp\left(-o\left(N\Delta^{-\frac{1}{k-1}}\right) \right)
    \]
     sets $U$ satisfying 
     \begin{enumerate}
     \item $S\subseteq U \subseteq V(G)$, 
     \item $|U| = M$, and 
     \item $E(U) =E(S)$.
     \end{enumerate}
\end{claim}

\begin{proof}
Let $q=2(M-M')/\open(S)$ and let $\mathbf W$ be a $q$-random subset of $\open(S)$. Note that by the assumption on $\open(S)$ and $M'$ and recalling that $M = \Theta(N\Delta^{-\frac{1}{k-1}})$, we have $q=o(\Delta^{-\frac{1}{k-1}})$. In particular, recalling that $p=\Theta(\Delta^{-\frac{1}{k-1}})$, we have $q=o(p)$ .

By the Harris-FKG inequality
\[
\P(E(S\cup \mathbf W)=E(S))\geq 
\prod_{\ell=0}^{k-2}(1-q^{k-\ell})^{|E_\ell(S)|} \geq \exp\left(-(1+o(1))\cdot k2^{k+1}|E(G)|\sum_{\ell=0}^{k-2}q^{k-\ell}p^\ell \right)\, .
\]
Noting that $|E(G)|\leq N\Delta/k$ and
\[
\sum_{\ell=0}^{k-2}q^{k-\ell}p^\ell=\frac{p^{k-1}q^2-q^{k+1}}{p-q}\leq 2p^{k-2}q^2
\]
since $q=o(p)$, we conclude that
\[
\P(E(S\cup \mathbf{W})=E(S))\geq \exp\left(-O\left(q^2N\Delta^{\frac{1}{k-1}} \right) \right)\, .
\]
By Chernoff's bound
\[
\P(||\mathbf{W}|-q\open(S)|\geq q\open(S)/2) \leq 2\exp\left(-q\open(S)/12 \right) \leq 2\exp\left(-q\alpha N/12 \right)\, .
\]
It follows that
\begin{align}
\P\left(E(S\cup \mathbf{W})=E(S)\wedge |\mathbf{W}|=(1\pm 1/2)q\open(S)\right)
&\geq \exp\left(-O(q^2N\Delta^{\frac{1}{k-1}})\right)-2\exp(-q\alpha N/12)\\
&= \exp\left(-o(N\Delta^{-\frac{1}{k-1}}) \right)
\end{align}
where for the final inequality we used that $q=o(p)=o(\Delta^{-\frac{1}{k-1}})$. Recall that $q=2(M-M')/\open(S)$ and that $M = \Theta(N\Delta^{-\frac{1}{k-1}})$ so that $M-M'=o(N\Delta^{-\frac{1}{k-1}})$ by our assumption on $M'$.
It follows by the pigeonhole principle, there exists $w\in [(M-M'), 3(M-M')]$ such that 
\[
\P\left(E(S\cup \mathbf{W})=E(S)\wedge |\mathbf{W}|=w\right)
\geq
 \exp\left(-o(N\Delta^{-\frac{1}{k-1}}) \right)\, .
\]
Writing $\cE$ for the event that $E(S\cup \mathbf{W})=E(S)$ and $|\mathbf{W}|=w$ we may rewrite the above as follows:
\[
(1-q)^{\open(S)}\sum_{W \in\cE }\left(\frac{q}{1-q}\right)^{w}\geq 
\exp\left(-o(N\Delta^{-\frac{1}{k-1}}) \right)\, .
\]
Noting that $(1-q)^{\open(S)}=\exp(-o(N\Delta^{-\frac{1}{k-1}}))$ and $w=o(N\Delta^{-\frac{1}{k-1}})$ it follows that
\[
|\cE|  \geq q^{-w}\exp\left(-o(N\Delta^{-\frac{1}{k-1}}) \right) =\left(\frac{\open(S)}{M-M'} \right)^{M-M'} \exp\left(-o(N\Delta^{-\frac{1}{k-1}}) \right)
\]
To construct a set $U$ satisfying the conditions of the lemma, we pick a subset $W'$ of $W\in\cE$ of size $(M-M')$ and set $U=S \cup W'$. The number of such $U$ is at least
\[
|\cE|\binom{w}{M-M'} \binom{\open(S)}{w-(M-M')}^{-1} \geq \frac{\open(S)^{M-M'}}{(M-M')^{M-M'}} \cdot \exp\left(-o\left(N\Delta^{-\frac{1}{k-1}}\right) \right)\, .\qedhere
\]
\end{proof}

We now prove Lemma \ref{lemTransferGnm}.
\begin{proof}[Proof of Lemma \ref{lemTransferGnm}]
     By Claims~\ref{clm:efintersect},~\ref{clmManyOpenEdges} and~\ref{clmSdegrees} and the assumption of the lemma we may assume that with probability at least $1/3$ we have
     \begin{enumerate}
\item $|\mathbf S| = (1+o(1))M$, $|E(\mathbf S)| = (1+o(1))T$, 
\item the number of pairs $(e,f)\in E(\mathbf S)^2$ such that $e\cap f\neq \emptyset$ is at most
\(
(N \Delta^{-\frac{1}{k-1}})^{3/2}\, ,
\)
\item $\open(\mathbf S) \geq \alpha N$,
\item   \(
    |E_\ell(\mathbf S)|\leq k2^{k+1}|E(G)| p^\ell 
    \)
     for all $\ell\in \{0,\ldots, k-2\}$.
     \end{enumerate}
     
    Hence, by the pigeonhole principle, there exist $M' = (1+o(1))M$ and $T' = (1+o(1))T$ such that the probability that $|\mathbf S|=M'$ and $|E(\mathbf S)| = T'$ and $\mathbf S$ satisfies $(2)-(4)$ is at least $1/o(N \Delta^{-\frac{1}{k-1}})\geq \exp(-o(N \Delta^{-\frac{1}{k-1}}))$.  Let $\mathcal S$ be the collection of sets $S$ such that $|S|=M'$ and $|E(S)| = T'$ and $S$ satisfies $(2)-(4)$.

    We now show a lower bound on the number of sets with their edge count in $[T_0,T]$. We consider two cases. First, if $T' \leq T$ then set $\mathcal S' = \mathcal S$ and $M''=M'$. Note that since $T'=(1-o(1))T$ in this case all sets in $\mathcal S'$ have an edge count in $[(1-o(1))T,T]$ and size $M'' = (1+o(1))M$.

    For the second case we consider $T' > T$. Set $r = 2\lfloor M\cdot (T'-T+(N \Delta^{-\frac{1}{k-1}})^{3/4})/T' \rfloor$ and $M''=M'-r$. By Claim \ref{clmLoweringTriangleCount} each set in $\mathcal S$ contains at least $(1-o(1)) \binom{M'}{r}$ sets of size $M''$ and edge count in $[(1-o(1))T,T]$. Let $\mathcal S'$ be the collection of these sets. Since each set in $\mathcal S'$ is contained in fewer than $\binom{N}{r}$ sets in $\mathcal S$,
    \[
    |\mathcal S'| \geq (1-o(1)) \cdot  |\mathcal S| \cdot \frac{\binom{M'}{r}}{\binom{N}{r}} = \lam^r \cdot |\mathcal S| \cdot \exp \left(-o\left(N\Delta^{-\frac{1}{k-1}}\right)\right).
    \]
    
    In either case we have now constructed $\mathcal S'$ as a collection of sets with size $M''=(1+o(1))M$, edge count in $[(1-o(1))T,T]$ and which satisfy (3) and (4). Additionally in either case we have
    \[
    |\mathcal S'| \geq  \lam^{M'-M''} \cdot |\mathcal S|\cdot \exp\left(-o\left(N\Delta^{-\frac{1}{k-1}}\right)\right). 
    \]
    Now, let $\mathcal S''$ be the collection of sets of size $M$ with edge count in $[T_0,T]$. To obtain a lower bound on $|\mathcal S''|$, once again we consider two cases. If $M'' \geq M$ then, by Claim \ref{clmDownwardsConstruction} each set in $\mathcal S'$ contains at least $(1-o(1))\binom{M''}{M''-M}$ sets in $\mathcal S''$. Since each set in $\mathcal {S''}$ is contained in at most $\binom{N}{M''-M}$ sets in $\mathcal S'$, we conclude that
    \[
    |\mathcal S''| \geq (1-o(1)) \cdot |\mathcal S'| \cdot \frac{\binom{M''}{M''-M}}{\binom{N}{M''-M}} = \lam^{M''-M} \cdot |\mathcal S'| \cdot  \exp\left(-o\left(N\Delta^{-\frac{1}{k-1}}\right)\right).
    \]
    Similarly, if $M>M''$ then Claim \ref{clmUpwardsConstruction} implies that
    \begin{align*}
    |\mathcal S''| & \geq |\mathcal S'| \cdot \frac{(\alpha N)^{M-M''}}{(M-M'')^{M-M''}\binom{M}{M-M''}} \cdot 
    \exp\left(-o\left(N\Delta^{-\frac{1}{k-1}}\right)\right) \\
    & = \lam^{M''-M} \cdot |\mathcal S'| \cdot \exp\left(-o\left(N\Delta^{-\frac{1}{k-1}}\right)\right).
    \end{align*}
    By a pigeonhole argument there exists some $T'' \in [(1-o(1)) T,T]$ such that there are at least $\omega\left(|\mathcal S''|/N\Delta^{-\frac{1}{k-1}}\right)$ sets in $\mathcal S''$ with exactly $T''$ edges. Let $\mathcal S''' \subseteq \mathcal S''$ be the collection of these sets. By the calculations above,
    \[
    |\mathcal S'''| = \omega\left( \frac{|\mathcal S''|}{N\Delta^{-\frac{1}{k-1}}}\right) \geq \lam^{M'-M}\cdot |\mathcal S| \cdot \exp\left(-o\left(N\Delta^{-\frac{1}{k-1}}\right)\right).
    \]
    To complete the proof we observe that
    \begin{align*}
        \P(|\mathbf S| = M \land |E(\mathbf S)| \in [T_0,T]) & \geq \P(\mathbf S \in \mathcal S''')
        = \frac{\lam^M (1-\zeta)^{T''}}{Z_G(\lam,\zeta)} |\mathcal S'''|\\
        & \geq \frac{\lam^M (1-\zeta)^{T''}}{Z_G(\lam,\zeta)} \lam^{M'-M} |\mathcal S| \cdot \exp \left(-o\left(N\Delta^{-\frac{1}{k-1}}\right)\right)\\
        & \geq \frac{\lam^{M'} (1-\zeta)^{T''}}{Z_G(\lam,\zeta)} |\mathcal S| \cdot \exp\left(-o\left(N\Delta^{-\frac{1}{k-1}}\right)\right)\\
        & \geq \frac{\lam^{M'} (1-\zeta)^{T'}}{Z_G(\lam,\zeta)} |\mathcal S| \cdot \exp\left(-o\left(N\Delta^{-\frac{1}{k-1}}\right)\right)\\
        & = \P (S \in \mathcal S) \cdot \exp\left(-o\left(N\Delta^{-\frac{1}{k-1}}\right)\right) \geq \exp\left(-o\left(N\Delta^{-\frac{1}{k-1}}\right)\right),
    \end{align*}
    which implies the claim.
\end{proof}

With Lemma \ref{lemTransferGnm} in hand, we finally prove  Lemma~\ref{lemZtoP}.

\subsection{Proof of Lemma~\ref{lemZtoP}}
Let $T=\eta \E_p\mathbf X = \eta p^k |E(G_n)|$. With a view to apply Lemma~\ref{lemTransferGnm}, let $\mathbf S \sim \mu_{G,\lam,\zeta}$ and note that by Theorem~\ref{thmMarginals} parts~\ref{item:marg1} and~\ref{item:marg2}, 
\[
\E |\mathbf S| = (1+o(1))m\quad\text{and}\quad  \E |E(\mathbf S)|=(1+o(1))T\, ,
\]
where the second equality holds since $\mathbf x^\ast, \zeta$ satisfy the conclusion of Lemma~\ref{lemZetaNonRegular}.
Moreover, by Theorem~\ref{thmMarginals} part~\ref{item:marg4} $\var |\mathbf S|, \var |E(\mathbf S)|=o(N^2 \Delta^{-\frac{2}{k-1}})$. We note that $m,T=\Theta(N\Delta^{-\frac{1}{k-1}})$ and so by Chebyshev's inequality we have, whp, 
\begin{align}\label{eq:ChebySES}
        |\mathbf S| &= (1+o(1)) m,\quad\text{and}\quad
        |E(\mathbf S)| = (1+o(1)) T\, .
    \end{align}

To upper bound the probability in question, we note 
\begin{align}
\P_p(\mathbf X \leq T) 
&= 
(1-p)^{N}\sum_{S : |E(S)| \leq T } \lam^{|S|}\\
&\leq 
(1-p)^{N}(1-\zeta)^{-T}\sum_{S : |E(S)| \leq T } \lam^{|S|}(1-\zeta)^{|E(S)|}\\
&\leq 
(1-p)^{N}(1-\zeta)^{-T}Z_G(\lam,\zeta)\, .
\end{align}

Now we prove the lower bound. Fix $\eps>0$  and let $T_0 = T - \varepsilon N\Delta^{-\frac{1}{k-1}}$. 
 \begin{align}
      \P_p \left(\mathbf X \le T\right) &\ge  (1-p)^{N} \sum_{S : |E(S)| \in [T_0,T]} \lam ^{|S|}  \\
      &\ge  (1-p)^{N} (1-\zeta)^{- T_0} \sum_{S : |E(S)| \in [T_0,T]} \lam ^{|S|} (1-\zeta)^{ |E(S)|} \\
      &= (1-p)^{N} (1-\zeta)^{- T_0} Z_G(\lam,\zeta) \P_{G,\lam,\zeta}(|E(\mathbf S)| \in [T_0,T]) \\
      &= (1-p)^{N} (1-\zeta)^{- T} Z_G(\lam,\zeta) \exp\left( - \eps \Theta\left(N\Delta^{-\frac{1}{k-1}}\right)  \right) \,, \label{eq:LTLB}
 \end{align}
where the last line follows from Lemma~\ref{lemTransferGnm} whose conditions are met by~\eqref{eq:ChebySES}. The first part of the lemma now follows after noting that $\eps$ was arbitrarily small and that $pN, T$ and $\log Z_{G}(\lam,\zeta)$ are all $\Theta(N \Delta^{-\frac{1}{k-1}})$ (the last follows from Theorem~\ref{thmMarginals} part~\ref{item:marg2.5}).  

For the second assertion of the lemma, we note that  
\begin{align}
\P_p ( |\mathbf S| = m \big | \mathbf X \le T)
 &=
 \frac{(1-p)^{N}}{\P_{p}(\mathbf X \le T)}\sum_{S: |S|=m,\, |E(S)| \leq T}\lam^{|S|}\\
 &\geq 
  \frac{(1-p)^{N}}{\P_{p}(\mathbf X \le T)}(1-\zeta)^{-T_0}\sum_{S: |S|=m,\, |E(S)| \in [T_0,T] }\lam ^{|S|}(1-\zeta)^{|E(S)|}\\
  &=
   \frac{(1-p)^{N}}{\P_{p}(\mathbf X \le T)}(1-\zeta)^{-T_0} Z_G(\lam,\zeta) \P_{G,\lam,\zeta} \left(\left(|\mathbf S| = m\right) \wedge\left( T_0 \leq |E(\mathbf S)| \leq T\right) \right)\\
   &\geq
    \exp\left( - \eps \Theta\left(N\Delta^{-\frac{1}{k-1}}\right)  \right)\, ,
\end{align}
where for the final inequality we used the first part of the lemma and Lemma~\ref{lemTransferGnm}. 
\qed

\section{Lower tails for subgraphs}
\label{secSubgraphs}

In this section we apply the main results to the case of lower tails for subgraph counts in $G(n,p)$ and $G(n,m)$.  

Let $H$ be a strictly $2$-balanced graph with $h$ vertices and $k$ edges.  Given $n$, define the hypergraph $G^H$ with $\binom{n}{2}$ vertices representing the edges of the complete graph $K_n$ and $k$-uniform hyperedges, one for each set of vertices whose corresponding set of edges forms a copy of $H$ in $K_n$.

We first verify that $G^H$ satisfies the conditions of the main theorem.  We set $N = \binom{n}{2}$ and $\Delta = \frac{2k}{ |\aut(H)|}(n-2)_{h-2}$.

\begin{lemma}
    \label{lemSubgraphHypergraph}
    Let $H$ be a strictly $2$-balanced graph, and define $G^H$ as above. Then
    \begin{enumerate}
        \item $G^H$ is $\Delta$-regular 
        
        \item For $\ell  \in \{2, \dots, k-1 \} $, $\Delta_\ell(G^H) =  o \left( \Delta^ {\frac{k- \ell}{k- 1}} \right)$
        \item $\Gamma(G) = o(\Delta)$.
    \end{enumerate}
\end{lemma}
\begin{proof}
    For the first, the total number of copies of $H$ in $K_n$ is $\frac{(n)_{h}}{|\aut(H)|}$.  Each has $k$ edges, so the sum of degrees of $G^H$ is $\frac{ k \cdot (n)_{h}}{|\aut(H)|}$. Dividing by $N=\binom{n}{2}$ gives the result.

    For the second, consider a set of $\ell$ edges $e_1, \dots, e_\ell$ in $K_n$.  The number of copies of $H$ that contain these edges is bounded by $n^t$ where $t$ is $h$ minus the number of vertices spanned by the $\ell$ edges.   Since $\Delta = \Theta(n^{h-2}) $, we have $\Delta^{ \frac{k- \ell}{k-1}} = \Theta \left(   n^{\frac{k- \ell}{m_2(H)}  } \right)$, and so it is enough to show that $t < \frac{k- \ell}{m_2(H)}  $.  By the strictly $2$-balanced condition on $H$, we have $t < v_H - 2 - \frac{\ell -1}{m_2(H)} = \frac{k-1}{m_2(H)} - \frac{\ell -1}{m_2(H)} = \frac{k -\ell}{m_2(H)}$ as desired. 

    For the third, we use the fact that a strictly $2$-balanced $H$ must be $2$-connected~\cite[Lemma 3.1]{bohman2010early}; in particular $H$ cannot have a pendant edge.  In particular, any $k-1$ edges of a copy of $H$ in $K_n$ determine the vertex set of that copy.     Now consider two edges $e_1, e_2$ in $K_n$. These edges span at least $3$ vertices.  The number of choices for edges $f_1, \dots , f_{k-1}$ in $K_n$ so that $\{e_1, f_1, \dots, f_{k-1} \}$ and  $\{e_2, f_1, \dots, f_{k-1} \}$ form copies of $H$ is at most $O(n^{h-3})$, since any such copy must contain the  vertices spanned by $e_1,e_2$ by the property above.  Since $\Delta = \Theta( n^{h-2})$, we have $\Gamma(G) = o(\Delta)$.
\end{proof}

Now we prove Theorem~\ref{thmHLowerTail}; Theorem~\ref{thmHFree} follows immediately by taking $\eta =0$.
\begin{proof}[Proof of Theorem~\ref{thmHLowerTail}]
 By the construction of $G^H$, we have 
 \begin{align}
     \P_p(\mathbf X_H \le \eta \E \mathbf X_H) = \P_p( \mathbf X \le \eta \E \mathbf X)
 \end{align}
 where the probability on the LHS is with respect to the random graph $G(n,p)$ and the probability on the RHS is with respect to a $p$-random subset of the vertices of $G^H$.  By Lemma~\ref{lemSubgraphHypergraph}, $G^H=G^H_n$ is asymptotically tree-like and regular and so the result follows from Theorem~\ref{thmMain} and Theorem~\ref{thmMainGnm}.
\end{proof}

Finally, we prove Corollary~\ref{corPhaseTransition}.
\begin{proof}
Let $r=\chi(H) -1$.
    By considering $r$-partite graphs, we can lower bound $\varphi_H(c)$ and $\widehat\varphi_H(b)$.
For $p =o( 1)$ and $m =o( n^2)$, we have $\P_p( G \text{ is $r$-partite}) \ge (1-p)^{(1+o(1))n^2/(2r) } $ and $\P_m ( G \text{ is $r$-partite}) \ge (1-1/r +o(1))^m$
    and so
    \begin{align}
        \varphi_H(c)  &\ge  - \frac{c}{2r}  \quad \text{and} \quad 
        \widehat \varphi_H(b) \ge  b \log (1 - 1/r)  \,.
    \end{align}

For $c$ large enough and $b$ large enough respectively, these bounds are larger than the value given by the functions on the RHS in~Theorem~\ref{thmHFree}.  Since these functions are analytic functions of $c$ and $b$, by uniqueness of analytic continuation, $\varphi_H$ and $\widehat \varphi_H$ must be non-analytic at some $c^\ast$ and $b^\ast$ respectively.
\end{proof}

As was done for the case of triangles in~\cite{jenssen2024lower}, one can similarly show that for sufficiently small but positive $\eta$, the lower-tail problem for copies of $H$ in $G(n,p)$ undergoes a phase transition in the sense of a non-analyticity of the rate function.  On the other hand, for $\eta \ge \eta_k^\ast$, there is no phase transition: the rate function is analytic for $c \in (0,\infty)$. 

For lower tails for copies of $H$ in $G(n,m)$, the same proof as above shows that there is a phase transition for every $\eta \in [0,1)$. 

\section{Avoiding arithmetic progressions}
\label{secAPs}

In this section we prove Theorem~\ref{thmApavoidInterval} on avoiding $k$-term arithmetic progressions.

To do so, we phrase the problem in terms of independent sets in a  hypergraph.  
Let $G^{(k)}_{n}$ be the $k$-uniform hypergraph with vertex set $[n]$ and edges for each $k$-subset on $[n]$ forming a $k$-AP.

Recall that for $k\geq 3$ we define
\[
\alpha_k:= \frac{1}{2}\sum_{i=1}^{k}
\min\left(
 \frac{1}{i-1},
 \frac{1}{k-i}\right)\, .
\]

We start with basic properties of $G^{(k)}_{n}$.
\begin{lemma}
    \label{lemAPhypergraph}
    The hypergraph $ G^{(k)}_{n}$ satisfies the following: 
    \begin{enumerate}
        \item $G^{(k)}_{n}$ has maximum degree  $\Delta\sim \alpha_k\cdot n$. 

        \item $\Delta_\ell(G^{(k)}_{n}) = O(1) $ for $\ell\in\{2,\ldots, k-1\}$.
        \item $\Gamma(G^{(k)}_{n}) = O(1)$.
    \end{enumerate}
\end{lemma}

For the rest of this section we write $G$ for $G^{(k)}_{n}$ and $\Delta$ for $\Delta(G)$.
 We will consider the edge-penalty model $\mu_{G,\lam,\zeta}$ with partition function $Z_{G}(\lam,\zeta)$.  We work in the  general setting of the edge-penalty model since it adds no additional difficulties beyond the $\zeta =1$ case we need for Theorem~\ref{thmApavoidInterval}.  One could deduce more general lower-tail results from the results in this section. 

After scaling by powers of $\alpha_k$, Lemma~\ref{lemkAPfixedpoint} and Theorem~\ref{thmApavoidInterval} follow directly from the following lemma and theorem which are stated in a form closer to our  hypergraph results.

Define the operator $F^{(k)}_{c,\zeta}$ on the set of bounded measurable functions $f: [0,1] \to \R_{>0}$ by
\begin{equation}
\label{eqFAPop}
    F^{(k)}_{c,\zeta} f(t) = c \exp \left[ - \frac{\zeta}{ \alpha_k}\sum_{\ell=1}^k  \int_{0}^{ \min \{\frac{t}{\ell-1} , \frac{1-t}{k-\ell} \}}   \,\prod_{i\in I_\ell}  f(t +i s)   \, ds \right]  \,.
\end{equation}

\begin{lemma}
\label{lemkAPfixedpoint2}
   Fix $k \ge 3$, $c > 0$, and $\zeta \in (0,1]$ satisfying $\zeta(k-1)c^{k-1} < e$. There is a unique  function $f_{k,\zeta,c}^* : [0,1] \to [0,c]$ so that 
   \begin{equation}
        F^{(k)}_{c,\zeta} f_{k,\zeta,c}^* = f_{k,\zeta,c}^* \,.
   \end{equation}
   Moreover the function $ f_{k,\zeta,c}^*$ is a continuous function on $[0,1]$ and for each $t \in[0,1]$, $f_{k,\zeta,c}^*(t)$ is an analytic function of $c$ on the interval $\left(0, \left(\frac{e} {\zeta(k-1)} \right)^{ \frac{1}{k-1}}  \right) $.
\end{lemma}

\begin{theorem}
    \label{thmApavoidInterval2}
      Fix $k \ge 3$, $c > 0$, and $\zeta \in (0,1]$ satisfying $\zeta(k-1)c^{k-1} < e$. Then with $\lam \sim c \Delta^{-\frac{1}{k-1}}$,
    \begin{equation}
       n^{-1} \Delta^{\frac{1}{k-1}}   \cdot \log Z_G ( \lam,\zeta) = \int_0^1 \int_0^c  \frac{f^*_{k,\zeta,t}(s)}{t}\, dt  \, ds   +o(1)  \,.
    \end{equation}
    Moreover, for each $j \in [n] = V(G)$,
    \begin{equation}
    \Delta^{\frac{1}{k-1}} \cdot \mu_{G,\lam,\zeta} \left ( j \in \mathbf S \right)   =  f_{k,\zeta,c}^*(j/n) +o(1)  \,.
    \end{equation}
\end{theorem}

We will first verify the conditions of Theorem~\ref{thmMarginals} in preparation to apply the result.  We will then show that the resulting sequence of BP fixed points has a limit described by the functional fixed point $f_{k,\zeta,c}^*$.

We start by proving Lemma~\ref{lemAPhypergraph}.
\begin{proof}[Proof of Lemma~\ref{lemAPhypergraph}]
We prove the claims in order. For positive integers $a, d$, let $P(a,d)$ denote the $k$-AP starting at $a$ with common difference $d$. Fix $t\in [n]$ and let $E(t)$ denote the set of $k$-APs contained in $[n]$ that contain $t$. Note that any $k$-AP containing $t$ is uniquely determined by the position $i\in [k]$ of $t$ in the AP and the common difference $d\geq 1$. In other words,
\begin{align}\label{eq:kAPst}
E(t)=\left\{P(t-(i-1)d,d): i\in [k], d\in \left[\min\!\left(
\left\lfloor \frac{t-1}{i-1}\right\rfloor,
\left\lfloor \frac{n-t}{k-i}\right\rfloor
\right)\right]\right\}\, .
\end{align}

Therefore the degree of $t$ in $G$ is
\begin{align}
\deg_G(t)
=
\sum_{i=1}^{k}
\min\left(
\left\lfloor \frac{t-1}{i-1}\right\rfloor,
\left\lfloor \frac{n-t}{k-i}\right\rfloor
\right)
 \sim 
\sum_{i=1}^{k}
\min\left(
 \frac{t}{i-1},
 \frac{n-t}{k-i}
\right)=:H(t)\, .
\end{align}
We observe that $H(x)$, considered as a function on $[0,n]$, is concave (as a sum of concave functions) and is symmetric: $H(t)=H(n-t)$. We conclude that $H$ is maximized at $t=n/2$ and so 
      \[
      \Delta(G)\sim \frac{n}{2} \sum_{i=1}^{k}
\min\left(
 \frac{1}{i-1},
 \frac{1}{k-i}\right)\, .
      \]
 This establishes (1).

 For (2) fix $X\subseteq [n]$ of size $\ell\in\{2,\ldots,k-1\}$. We note that any $k$-AP that contains $X$ is uniquely determined by  specifying the positions of the elements of $X$ in the AP. Thus $\Delta_{\ell}(G) \leq \binom{k}{\ell}$.

Finally we prove (3). Fix $x,y\in [n]$, $x<y$, and suppose that $S\subset[n]$ has size $k-1$ and $S\cup \{x\}$ and $S\cup \{y\}$ are both $k$-APs.

Assume first that $k\geq 4$. Note that if $S$ is a $(k-1)$-AP with common difference $d$ then we must have $x=\min(S)-d$, $y=\max(S)+d$ and this uniquely determines $S$. If $S$ is not a $(k-1)$-AP, then there is at most one way to extend $S$ to a $k$-AP. We therefore conclude that $\Gamma(G)=1$.

Suppose now that $k=3$ and write $S=\{a,b\}$ where $a<b$. Let $d=b-a$ and let $m=(a+b)/2$. Either $(i)$ $x=a-d, y=b+d$, $(ii)$ $x=m, y=b+d$ or $(iii)$ $y=m$, $x=a-d$. In each case $S$ is uniquely determined and so $\Gamma(G)\leq 3$.
\end{proof}

 With the goal of applying Theorem~\ref{thmMarginals}, we now consider the BP functional $F_n:=F_{c,\zeta}^{G}$ as defined at~\eqref{eq:BPF}.  For $\ell\in\{1,\ldots, k\}$, let $I_\ell=\{-\ell+1,\ldots, k-\ell\}\backslash\{0\}$. We view $F_n$ as an operator on the set of functions $f:\{1,\ldots,n\}\to \R$ and note that
\begin{equation}
    F_n f(t) = c \cdot \exp\left[- \frac{\zeta}{\Delta} \sum_{e\in \mathcal P(t)}\prod_{u\in e \backslash\{t\}} f(u)\right] =  c \exp \left [ - \frac{\zeta}{\Delta}\sum_{\ell=1}^k  \sum_{s=1}^{ \min \{\lfloor \frac{t-1}{\ell-1}\rfloor , \lfloor\frac{n-t}{k-\ell}\rfloor \}}   \,\prod_{i\in I_\ell}  f(t +i s)   \right]  \, ,
\end{equation}
where for the second equality we used~\eqref{eq:kAPst}. Theorem~\ref{thmMarginals} will allow us to express $\log Z_G(\lam,\zeta)$ in terms of a fixed point of $F_n$. To prove Theorem~\ref{thmApavoidInterval2}, we show that these fixed points are well approximated (in the limit $n\to\infty$) by a fixed point of the  limit  operator $F = F^{(k)}_{c,\zeta}$ defined in~\eqref{eqFAPop}.

We will take advantage of the contractive properties of $F$ and $F_n$. To this end, we prove the following analogue of Lemma~\ref{lemfkContractionNonReg} for $F$.
\begin{lemma}\label{lemfkContractionNonReg2}
    Fix $k \ge 3$, $c > 0$, and $\zeta \in (0,1]$ satisfying $\zeta(k-1)c^{k-1} < e$. Then for bounded measurable $f,g:[0,1]\to \R_{>0}$,
    \begin{equation}
     \left\| \log(F^2 (f)) - \log(F^2(g))\right\|_\infty    \le (1- \del) \left\| \log(f) - \log(g)\right\|_\infty\, , 
    \end{equation}
    where $\del=1-\zeta c^{k-1}(k-1)e^{-1}$.
\end{lemma}
The proof of this lemma is very similar to that of Lemma~\ref{lemfkContractionNonReg} and so we defer it to Appendix~\ref{App:Contract}.  Using this  we can prove Lemma~\ref{lemkAPfixedpoint2}.
\begin{proof}[Proof of Lemma~\ref{lemkAPfixedpoint2}]
The uniqueness of the fixed point follows immediately from the contraction in Lemma~\ref{lemfkContractionNonReg2}.  Continuity of the fixed point $f_{k,\zeta,c}^\ast(t)$ in $t$ follows from the continuity of the operator $F$ in $t$.  Analyticity in $c$ follows from the analytic dependence of $F$ on $c$ and the analytic implicit function theorem as in the proof of Lemma~\ref{lem:diffx}.
\end{proof}

Henceforth, we let  $f^\ast, f^\ast_n$ denote the fixed points of $F, F_n$ respectively. 

We will abuse notation somewhat and identify a function $f:\{1,\ldots, n\}\to [0,c]$ with the step function $[0,1]\to[0,c]$ which maps the interval $(\frac{i-1}{n},\frac{i}{n}]$ to $f(i)$ for $i=1,\ldots, n$ (and maps $0$ to $f(1)$). 

Our aim is to prove the following
\begin{lemma}\label{lem:kAPdisctocont}
As $n\to\infty$,
\[
\|f_n^\ast-f^\ast\|_\infty \to 0\, .
\]
\end{lemma}
Our strategy is to show that $f_n^\ast$ is an approximate fixed point of $F^2$ which will imply, by the contractive property of $F^2$ (Lemma~\ref{lemfkContractionNonReg2}), that $f_n^*$ is close to $f^\ast$. We do this in a sequence of lemmas.  

Given $f:\{1,\ldots,n\}\to[0,c]$, we let $D_t(f)=|f(t+1)-f(t)|$ and $L(f)=\max_{t}D_t(f)$.
\begin{lemma}\label{lem:discLip}
If $f:\{1,\ldots,n\}\to [0,c]$, then
\[
L(F_nf)\leq \beta L(f)+\gamma\, ,
\]
where $\beta=\zeta k c^k$ and $\gamma=\zeta k c^{k}\Delta^{-1}$.
\end{lemma}
\begin{proof}
Define the operator $A_n$ by $F_nf(t)=c\exp[-\zeta A_nf(t)]$ and note that
\begin{align}\label{eq:FwdDiff}
    D_t(F_nf)\leq c\zeta D_t(A_nf)
\end{align}
since $|e^{-x}-e^{-y}|\leq |x-y|$ for all $x,y\geq 0$.  Let $u(t,\ell)=\min \{\lfloor \frac{t-1}{\ell-1}\rfloor , \lfloor\frac{n-t}{k-\ell}\rfloor \}$ and
 fix $1\leq \ell\leq k$ and $1\leq s \leq \min\{u(t,\ell), u(t+1,\ell)\}$. We have
\begin{align}
    \left| \prod_{i\in I_\ell} f(t+1+is)-\prod_{i\in I_\ell}f(t+is) \right|
    &\leq 
    c^{k-1}\sum_{i\in I_\ell}|f(t+1+is)-f(t+is)|\\
    & \leq 
    kc^{k-1} L(f)\, ,
\end{align}
where for the first inequality we used the basic fact that if $x_1,\ldots, x_k, y_1,\ldots,y_k\in[0,c]$ then $|\prod_{i=1}^kx_i - \prod_{i=1}^ky_i|\leq c^{k-1}\sum_{i=1}^k|x_i-y_i|$.
We conclude that
\begin{align}
 \left|\sum_{s=1}^{u(t+1,\ell)} \prod_{i\in I_\ell} f(t+1+is)-\sum_{s=1}^{u(t,\ell)} \prod_{i\in I_\ell}f(t+is) \right|
    \leq 
   u(t,\ell) k c^{k-1}L(f) + c^{k-1}
\end{align}
where we used that the difference between $u(t,\ell), u(t+1,\ell)$ is at most $1$, and $|\prod_{i\in I_\ell}f(t+is)|\leq c^{k-1}$ for all $t,\ell,s$. 
Returning to~\eqref{eq:FwdDiff} we conclude that
\[
 D_t(F_nf)\leq c\zeta D_t(A_nf)\leq 
  \frac{c\zeta}{\Delta}   \sum_{\ell=1}^k(u(t,\ell)kc^{k-1}L(f) + c^{k-1})\, .
\]
The result follows by recalling that $\Delta\geq \text{deg}_G(t)=\sum_{\ell=1}^k u(t,\ell)$ for all $t$. 
\end{proof}

\begin{cor}\label{cor:lip}
    \[
    L(f_n^\ast)=o(1)\, .
    \]
\end{cor}
\begin{proof}
Fix $\eps>0$. Let $h_0:\{1,\ldots,n\}\to[0,c]$ denote the constant function $h_0\equiv c$ and note that 
\begin{align}\label{eq:Lfnh0}
L(f_n^\ast)\leq L(F_n^{(m)}h_0)+2\| F_n^{(m)}h_0-f_n^\ast\|_\infty\, ,
\end{align}
for $m>0$. We bound each of these terms separately (for appropriately chosen $m$). For the second term, note that by Lemma~\ref{lemfkContractionNonReg} 
\[
\| \log F_n^{(m)}h_0-\log f_n^{\ast}\|_\infty\leq (1-\delta)^{m/2}\|\log h_0 - \log f_n^{\ast}\|_\infty 
\]
for even $m>0$. Since the functions  $h_0, F_n^{(m)}h_0, f_n^\ast$ are all bounded from below by $ce^{-\zeta c^{k-1}}$ (and from above by $c$) we may choose $m=m(\zeta,c,k,\eps)$ large enough so that 
\[
\| F_n^{(m)}h_0- f_n^{\ast}\|_\infty<\eps\, .
\]
To bound  $L(F_n^{(m)}h_0)$ note that $L(h_0)=0$ and iteratively apply Lemma~\ref{lem:discLip}, recalling that $\Delta\sim \alpha_kn$, to obtain
 \[
 L(F_n^{(m)}h_0)\leq (1+\beta+\ldots+\beta^{m-1})\gamma <\eps
 \]
 for $n$ sufficiently large (as a function of $m,\zeta,c,k,\eps$). The result follows from~\eqref{eq:Lfnh0}.
 \end{proof}
 
 \begin{lemma}\label{lem:hgstep}
 If $g:\{1,\ldots,n\}\to [0,c]$, then
 \[
 \| Fg -  F_n g\|_\infty =   O(L(g))+o(1)\, .
  \]
 \end{lemma}
 
 \begin{proof}

Define the operators $A, A_n$ by $Ff(t)=c\exp[-\zeta Af(t)]$ and $F_nf(t)=c\exp[-\zeta A_nf(t)]$, so in particular
\begin{equation}
    A g(t) = \frac{1}{\alpha_k}\sum_{\ell=1}^k  \int_{0}^{ \min \{\frac{t}{\ell-1} , \frac{1-t}{k-\ell} \}}   \,\prod_{i\in I_\ell}  g(\lceil (t +i s)n\rceil)   \, ds   \,.
\end{equation}
Then since $|e^{-x}-e^{-y}|\leq |x-y|$ for all $x,y\geq 0$ we have
\begin{align}\label{eq:FtoA}
\| Fg -  F_n g\|_\infty \leq c\zeta \|Ag - A_n g\|_{\infty}\, .
\end{align}
We now bound the RHS. To this end we derive a convenient expression for $Ag$ for comparison with $A_ng$.
Fix $\ell\in [k]$ and $t\in [0,1]$, let $w=\min \{\frac{t}{\ell-1} , \frac{1-t}{k-\ell}\}$ 
and let
\[
p(s):= \prod_{i\in I_\ell}   g(\lceil (t +i s)n\rceil) \, .
\]
Splitting $[0,w]$ into intervals of length $1/n$ we have
\begin{align}\label{eq:intsplit}
\int_{0}^w  p(s)  \, ds
=
\int_{\lfloor w n \rfloor/n}^{w}  p(s)  \, ds+\sum_{r=1}^{\lfloor w n \rfloor }\int_{(r-1)/n}^{r/n}  p(s)  \, ds\, .
\end{align}
Fix $i\in I_\ell$ and $r\leq \lfloor w n \rfloor$. As $s$ ranges over the interval $[(r-1)/n,r/n]$, the quantity $\lceil (t +i s)n\rceil$ ranges over the interval $\{\lceil tn \rceil+i(r-1),\ldots,\lceil tn \rceil+ir \}$, an interval of length $i$, and so 
\[
g(\lceil (t +i s)n\rceil)= g(\lceil tn \rceil+ir) + O(L(g)i)
\]
for all $s\in [(r-1)/n,r/n]$. 
Since $g$ is bounded by $c$ we conclude that 
\[
p(s)=\prod_{i\in I_\ell}    g(\lceil tn \rceil+ir) + O(L(g))\, ,
\]
for all $s\in [(r-1)/n,r/n]$ and so 
\[
\int_{(r-1)/n}^{r/n}  p(s)  \, ds = \frac{1}{n}\prod_{i\in I_\ell}    g(\lceil tn \rceil+ir) + O(L(g)/n)\, .
\]
Therefore, by~\eqref{eq:intsplit}
\[
\int_{0}^{w}  p(s)  \, ds = \frac{1}{n}\sum_{r=1}^{\lfloor w n \rfloor }\prod_{i\in I_\ell}    g(\lceil tn \rceil+ir) + O(L(g)+1/n)\, .
\]
Recalling that $w=w(t,\ell)=\min \{\frac{t}{\ell-1} , \frac{1-t}{k-\ell}\}$ we have, for $t\in[0,1]$,
\begin{align}
Ag(t)&= \frac{1}{\alpha_k n} \sum_{\ell=1}^k \sum_{r=1}^{\lfloor w(\ell,t)n\rfloor} \prod_{i\in I_\ell}    g(\lceil tn \rceil+ir) + O(L(g)+1/n)\, .
\end{align}
We compare this quantity to 
\[
A_ng(t)=\frac{1}{\Delta}\sum_{\ell=1}^k  \sum_{r=1}^{ u(\lceil tn \rceil,\ell)}   \,\prod_{i\in I_\ell}  g(\lceil tn \rceil +i r)
\]
where we recall that
$u(m,\ell)=\min \{\lfloor \frac{m-1}{\ell-1}\rfloor , \lfloor\frac{n-m}{k-\ell}\rfloor \}$. 
 First we observe that $\lfloor w(t,\ell)n\rfloor=u(\lceil tn \rceil,\ell )+O(1)$. Recalling that $g$ is bounded by $c$ and that $\Delta\sim \alpha_kn$, we conclude that 
 \[
 Ag(t)=A_ng(t)+ O(L(g)) + o(1)\, .
 \]
 Returning to~\eqref{eq:FtoA} completes the proof.
\end{proof}
We deduce that $f_n^\ast$ is an approximate fixed point of $F^2$.
\begin{lemma}\label{lem:approxfixedpt}
\[
\|  F^2 f_n^\ast - f_n^\ast\|_\infty=o(1)\, .
\]
\end{lemma}
\begin{proof}
By the triangle inequality
\[
\| F^2f_n^\ast- f_n^\ast\|_\infty\leq \| F^2f_n^\ast- Ff_n^\ast\|_\infty+\| Ff_n^\ast- f_n^\ast\|_\infty\, .
\]
By Corollary~\ref{cor:lip} and Lemma~\ref{lem:hgstep}, we have \[
\| Ff_n^\ast- f_n^\ast\|_\infty= \| Ff_n^\ast- F_nf_n^\ast\|_\infty=o(1)\, .\]
Then since $F$ is continuous with respect to $\|\cdot\|_\infty$ we conclude that $\| F^2f_n^\ast- Ff_n^\ast\|_\infty=o(1)$ also. 
\end{proof}

\begin{proof}[Proof of Lemma~\ref{lem:kAPdisctocont}]
Note that by Lemma~\ref{lemfkContractionNonReg2}
\[
\|\log f^\ast  - \log F^2f_n^\ast\|_\infty=\|\log F^2f^\ast  - \log F^2f_n^\ast\|_\infty \leq (1-\delta)\|\log f^\ast-\log f_n^\ast\|_\infty\, .
\]
On the other hand by Lemma~\ref{lem:approxfixedpt} (noting that $F^2 f_n^\ast, f_n^\ast$ are bounded below by $ce^{-\zeta c^{k-1}}$), 
\[
\|  \log F^2 f_n^\ast - \log f_n^\ast\|_\infty=o(1)
\]
and so 
\[
\|\log f^\ast  - \log f_n^\ast\|_\infty \leq (1-\delta)\|\log f^\ast-\log f_n^\ast\|_\infty + o(1)\, .
\]
The result follows. 
\end{proof}

\begin{proof}[Proof of Theorem~\ref{thmApavoidInterval2}]
 Fix $k \ge 3$, $c > 0$, and $\zeta \in (0,1]$ satisfying $\zeta(k-1)c^{k-1} < e$.
Let $G=G_n^{(k)}$, $\Delta=\Delta(G)$, and  $\lam \sim c \Delta^{-\frac{1}{k-1}}$. We apply
Theorem~\ref{thmMarginals}\ref{item:marg1} (which is justified by Lemma~\ref{lemAPhypergraph}) to conclude that 
\[
 \mu_{G,\lam,\zeta}(j \in \mathbf S)=(1+o(1)) \Delta^{-\frac{1}{k-1}}f^\ast_{n,k,\zeta,c}(j)=(1+o(1))\Delta^{-\frac{1}{k-1}}f^\ast_{k,\zeta,c}(j/n)\, ,
\]
    where for the final equality we used Lemma~\ref{lem:kAPdisctocont}. 

    For the first claim of Theorem~\ref{thmApavoidInterval2} we apply Theorem~\ref{thmMarginals}\ref{item:marg2.5} to conclude that
\begin{align}\label{eq:kAPnonregApp}
       \Delta^{\frac{1}{k-1}} n^{-1}  \log Z_G(\lam,\zeta)
       &=  \frac{1}{n} \sum_{i=1}^n \int_{0}^c\frac{f_{n,k,\zeta,t}^\ast(i)}{t}\, dt    +o(1)\,,
    \end{align}
    We next aim to show that
    \begin{align}\label{eq:DCT}
    \frac{1}{n}\sum_{i=1}^n \int_{0}^c \frac{f_{n,k,\zeta,t}^\ast(i)}{t}\, dt = \int_0^1 \int_0^c  \frac{f^*_{k,\zeta,t}(s)}{t}\, dt  \, ds +o(1)\, .
    \end{align}
At this point (dropping $k,\zeta$ from the notation) it will be useful to distinguish $f^\ast_{n,t}:\{1,\ldots,n\}\to\R$ from its associated step function which we denote by $\tilde f^\ast_{n,t}:[0,1]\to\R$. Recall that $\tilde f^\ast_{n,t}(s)=f^\ast_{n,t}(\lceil sn \rceil)$ for $s\in(0,1]$ and $\tilde f^\ast_{n,t}(0)=f^\ast_{n,t}(1)$. 
We then have
    \[
     \frac{1}{n}\sum_{i=1}^n 
     f_{n,t}^\ast(i)= \int_{0}^1 \tilde f^\ast_{n,t}(s)\, ds\, .
    \]

Define
\[
A_n(t) := \int_0^1 \frac{\tilde f_{n,t}^*(s)}{t}\,ds,
\qquad
A(t) := \int_0^1 \frac{f_t^*(s)}{t}\,ds.
\]
Then
\[
|A_n(t) - A(t)|
\le
\int_0^1 \frac{|\tilde f_{n,t}^*(s) - f_t^*(s)|}{t}\,ds
\le
\frac{1}{t}\, \|\tilde f_{n,t}^* - f_t^*\|_\infty.
\]
Since the right-hand side tends to $0$ for each fixed $t\in (0,c]$ by Lemma~\ref{lem:kAPdisctocont}, we have
$A_n(t)\to A(t)$ pointwise on $(0,c]$. Moreover, since $0 \le f_{n,t}^* \le t$, we have $0 \le A_n(t) \le 1$.
Therefore, by the dominated convergence theorem,
\[
\int_0^c A_n(t)\,dt = \int_0^c A(t)\,dt +o(1)\,,
\]
which completes the proof of~\eqref{eq:DCT} (after swapping the order of integration) and thus the proof of the first claim of Theorem~\ref{thmApavoidInterval2}.
\end{proof}

We remark that we could have applied Theorem~\ref{thmMarginals}\ref{item:marg3} (the Bethe free energy formulation) in place of Theorem~\ref{thmMarginals}\ref{item:marg2.5} in the proof above.
 This leads to the alternative formula (recalling that $x^\ast_{k,c}$ denotes the fixed point of the operator $\Phi^{(k)}_c$ from~\eqref{eqPhikc})
  \begin{align}
        & \lim_{n\to\infty} n^{- \frac{k-2}{k-1} } \cdot \log \P [ \mathbf X_k([n]_p)=0] =\\ &
       -\int_0^1   \int_{0}^{ \frac{1-t}{k-1}}   \,\prod_{i=0}^{k-1}  x^*_{k,c}(t +is)    \, ds \, dt  - \int_0^1 x^*_{k,c}(t) \left[\log \frac{x^*_{k,c}(t)}{c}  -1\right] \, dt  -c \,,
     \end{align}
   under the assumptions of Theorem~\ref{thmApavoidInterval}. Though this is visually more cumbersome than the formula presented in Theorem~\ref{thmApavoidInterval}, it is computationally simpler since it only involves the fixed point $x_{k,c}^\ast$ as opposed to the infinite family of fixed points $x_{k,t}^\ast$, $t\in[0,c]$.

\section*{Acknowledgments}

MJ is supported by a UK Research and Innovation Future Leaders Fellowship MR/W007320/2.  WP supported in part by NSF grant DMS-2348743.

\appendix
\section{Proof of Lemma~\ref{lemfkContractionNonReg2}}
\label{App:Contract}

\begin{proof}[Proof of Lemma~\ref{lemfkContractionNonReg2}]

Recall that we define the operator $F=F^{(k)}_{c,\zeta}$ on the set of bounded measurable functions $f: [0,1] \to \R_{>0}$ by
\begin{equation}
    F f(t) = c \exp \left[ - \frac{\zeta}{ \alpha_k}\sum_{\ell=1}^k  \int_{0}^{ w(t,\ell) }   \,\prod_{i\in I_\ell}  f(t +i s)   \, ds \right]  \,,
\end{equation}
where $\alpha_k$ is as in Lemma~\ref{lemAPhypergraph}, $I_\ell=\{-\ell+1,\ldots, k-\ell\}\backslash\{0\}$ and $w(t,\ell)=\min \{\frac{t}{\ell-1} , \frac{1-t}{k-\ell}\}$.

Let $\mathcal X, \mathcal X_{>0}$ denote the set of measurable functions $[0,1] \to \R$, $[0,1] \to \R_{>0}$ respectively.
Fix $t\in [0,1]$, and for $f\in\mathcal X_{>0}$, let $\psi(f)=(F^2f)(t)$. Our goal is to show that
\begin{equation}\label{eq:psigoal}
    | \log(\psi(f)) - \log(\psi(g))|    \le (1- \del) \left\| \log(f) - \log(g)\right\|_\infty 
    \end{equation}
    for every $f,g\in \mathcal X_{>0}$.

Let $\varphi(f)=\log(\psi(\exp(f)))$. It suffices to show that
\begin{align}\label{eq:varphicontraction}
|\varphi(f)-\varphi(g)|\leq (1-\delta)\|f-g\|_\infty
\end{align}
for all $f,g\in \mathcal X$.
 For $h,u\in \mathcal X$, denote the derivative of $\varphi$ at $h$ in the direction $u$ by
\[
D\varphi(h)[u]:=\lim_{\eps\to 0} \frac{\varphi(h+\eps u)-\varphi(h)}{\eps}\, .
\]
By the fundamental theorem of calculus
\begin{align}
\varphi(f)-\varphi(g)=\int_{0}^1D\varphi(g+\theta(f-g))[f-g] d\theta\, ,
\end{align}
and so to prove~\eqref{eq:varphicontraction} it suffices to show that
\begin{align}\label{eq:Dbound}
|D\varphi(h)[u]|\leq (1-\delta)\|u\|_\infty\, 
\end{align}
for all $h,u\in \mathcal{X}$.
Fix $h,u\in \mathcal{X}$ and set $f=e^h$. Then 
\[
\varphi(h)= \log(F^2f(t))= \log c- \frac{\zeta}{ \alpha_k}\sum_{\ell=1}^k  \int_{0}^{w(t,\ell)}   \,\prod_{i\in I_\ell}  Ff(t +i s)   \, ds\, .
\]
Differentiating we obtain
\begin{align}\label{eq:Dphiath}
D\varphi(h)[u]
=
-\frac{\zeta}{\alpha_k}
\sum_{\ell=1}^k
\int_0^{w(t,\ell)}
\left(
\prod_{i\in I_\ell} Ff(t+is)
\right)
\left(
\sum_{j\in I_\ell} D(\log Ff)(t+js)[u]
\right)
\, ds\, ,
\end{align}
where all derivatives are at the fixed function $h$ (note that differentiation under the integral sign is justified by dominated convergence).

Since
\[
\log Ff(x)
=
\log c
-
\frac{\zeta}{\alpha_k}
\sum_{m=1}^k
\int_{0}^{w(x,m)}
\prod_{p\in I_m} f(x+ps)\,ds,
\]
we obtain
\begin{align}
D(\log Ff)(x)[u]
&=
-\frac{\zeta}{\alpha_k}
\sum_{m=1}^k
\int_{0}^{w(x,m)}
\left(
\prod_{p\in I_m} f(x+ps)
\right)
\left(
\sum_{q\in I_m} D(\log f)(x+qs)[u]
\right) \,
ds\\
&=-\frac{\zeta}{\alpha_k}
\sum_{m=1}^k
\int_{0}^{w(x,m)}
\left(
\prod_{p\in I_m} f(x+ps)
\right)
\left(
\sum_{q\in I_m} u(x+qs)
\right) \,
ds
\end{align}

Hence
\begin{align}
\bigl|D(\log Ff)(x)[u]\bigr|
&\le (k-1)\,\|u\|_\infty
\frac{\zeta}{\alpha_k}
\sum_{m=1}^k
\int_{0}^{w(x,m)}
\prod_{p\in I_m} f(x+ps)\,ds\\
&=
(k-1)\,\|u\|_\infty \log\frac{c}{Ff(x)}
\end{align}
where for the first inequality we recall that $|I_m|=k-1$ and for the second we recalled the definition of the operator $F$.

Substituting this bound into the expression~\eqref{eq:Dphiath} for $D\varphi(h)[u]$ yields
\[
|D\varphi(h)[u]|
\le
\frac{\zeta}{\alpha_k}(k-1)\|u\|_\infty
\sum_{\ell=1}^k
\int_{0}^{w(t,\ell)}
Y_{\ell,s}
\log\!\left(\frac{c^{k-1}}{Y_{\ell,s}}\right)\,
ds,
\]
where
\[
Y_{\ell,s} := \prod_{i\in I_\ell} Ff(t+is),
\qquad 0 < Y_{\ell,s} \le c^{k-1}.
\]
Using the inequality
\(
z \log (a/z) \le a/e
\)
for $0<z\le a$ with $a = c^{k-1}$, we obtain
\[
|D\varphi(h)[u]|
\le
\frac{\zeta}{\alpha_k}(k-1)\|u\|_\infty
\frac{c^{k-1}}{e}
\sum_{\ell=1}^k w(t,\ell)\leq (1-\delta) \|u\|_\infty\, ,
\]
as desired, where for the final inequality we note that $\sum_{\ell=1}^k w(t,\ell) \leq \alpha_k$ (with equality when $t=1/2$) and recall that $\del=1-\zeta c^{k-1}(k-1)e^{-1}$.
\end{proof}

\end{document}